%% file: main.tex
\documentclass{article}

\usepackage[final,nonatbib]{neurips_2020}

\usepackage[utf8]{inputenc} 
\usepackage[T1]{fontenc}    
\usepackage[colorlinks=true,allcolors=cyan,pagebackref=true]{hyperref}
\usepackage{url}            
\usepackage{booktabs}       
\usepackage{amsfonts}       
\usepackage{nicefrac}       
\usepackage{microtype}      
\usepackage{bbm}
\usepackage{multirow}
\usepackage{xr}
\usepackage{amsthm}
\usepackage{blkarray}
\usepackage{xcolor}
\usepackage{graphicx}
\usepackage{footmisc}
\graphicspath{{./figures/}}

\input{preamble_packages}

\input{preamble_symbols}

\input{shortcuts}

\title{One Ring to Rule Them All: Certifiably Robust Geometric Perception with Outliers}

\author{%
  Heng Yang and Luca Carlone\\
  Laboratory for Information and Decision Systems (LIDS) \\
  Massachusetts Institute of Technology \\
  \texttt{ \{hankyang,lcarlone\}@mit.edu }
}

\begin{document}
\maketitle

\input{abstract}
\input{introduction}
\input{relatedWork}
\input{pop}
\input{primal-moment}
\input{dual-certification}
\input{experiments}

\input{conclusions}
\input{impact}
\input{acknowledgements}

\clearpage

\begin{center}
{\Large\bf Supplementary Material}
\end{center}
\renewcommand{\thesection}{A\arabic{section}}
\renewcommand{\theequation}{A\arabic{equation}}
\renewcommand{\thetheorem}{A\arabic{theorem}}
\renewcommand{\thefigure}{A\arabic{figure}}
\renewcommand{\thetable}{A\arabic{table}}

\setcounter{equation}{0}
\setcounter{section}{0}
\setcounter{theorem}{0}
\setcounter{figure}{0}

\input{supp-notations}
\input{supp-proof-pop}
\input{supp-proof-denseRelaxation}

\input{supp-proof-sparseRelaxation}
\input{supp-proof-sufficientOptimality}
\input{supp-proof-DRS}
\input{supp-chordalSparse}
\input{supp-experiments}

\clearpage
{\small
\bibliographystyle{plain}
\bibliography{../../../references/refs.bib,myRefs.bib}
}

\end{document}

%% file: preamble_packages.tex
\usepackage{comment}
\usepackage{siunitx}
\usepackage{relsize}
\usepackage{ifthen}
\usepackage[colorinlistoftodos]{todonotes}

\usepackage[caption=false]{subfig}

\usepackage[vlined,ruled,linesnumbered]{algorithm2e}
\usepackage{graphics} 
\usepackage{rotating}
\usepackage{color}
\usepackage{enumerate}
\usepackage[T1]{fontenc}
\usepackage{psfrag}
\usepackage{epsfig} 
\usepackage{booktabs}
\usepackage{graphicx,url}
\usepackage{multirow}
\usepackage{array}
\usepackage{latexsym}
\usepackage{amsfonts}
\usepackage{amsmath}
\usepackage{amssymb}
\usepackage{xstring}
\usepackage[noend]{algorithmic}
\usepackage{multirow}
\usepackage{xcolor}
\usepackage{prettyref}
\usepackage{flexisym}
\usepackage{bigdelim}
\usepackage{breqn} 
\usepackage{listings}

\usepackage{enumitem}
\usepackage{xspace}
\usepackage{bm}
\graphicspath{{./figures/}}
\usepackage{tikz}
\usetikzlibrary{matrix,calc}

\usepackage{mdwlist}

\makecompactlist{itemize}{stditemize}


%% file: preamble_symbols.tex
\newrefformat{prob}{Problem\,\ref{#1}}
\newrefformat{def}{Definition\,\ref{#1}}
\newrefformat{sec}{Section\,\ref{#1}}
\newrefformat{sub}{Section\,\ref{#1}}
\newrefformat{prop}{Proposition\,\ref{#1}}
\newrefformat{app}{Appendix\,\ref{#1}}
\newrefformat{alg}{Algorithm\,\ref{#1}}
\newrefformat{cor}{Corollary\,\ref{#1}}
\newrefformat{thm}{Theorem\,\ref{#1}}
\newrefformat{lem}{Lemma\,\ref{#1}}
\newrefformat{fig}{Fig.\,\ref{#1}}
\newrefformat{tab}{Table\,\ref{#1}}

\newtheorem{theorem}{Theorem}

\newtheorem{lemma}[theorem]{Lemma}

\newtheorem{definition}[theorem]{Definition}
\newtheorem{proposition}[theorem]{Proposition}
\newtheorem{remark}[theorem]{Remark}
\newtheorem{example}[theorem]{Example}

\newcommand{\cf}{\emph{cf.}\xspace}

\newcommand{\bdmath}{\begin{dmath}}
\newcommand{\edmath}{\end{dmath}}
\newcommand{\beq}{\begin{equation}}
\newcommand{\eeq}{\end{equation}}
\newcommand{\bdm}{\begin{displaymath}}
\newcommand{\edm}{\end{displaymath}}
\newcommand{\bea}{\begin{eqnarray}}
\newcommand{\eea}{\end{eqnarray}}
\newcommand{\beal}{\beq \begin{array}{ll}}
\newcommand{\eeal}{\end{array} \eeq}
\newcommand{\beas}{\begin{eqnarray*}}
\newcommand{\eeas}{\end{eqnarray*}}
\newcommand{\ba}{\begin{array}}
\newcommand{\ea}{\end{array}}
\newcommand{\bit}{\begin{itemize}}
\newcommand{\eit}{\end{itemize}}
\newcommand{\ben}{\begin{enumerate}}
\newcommand{\een}{\end{enumerate}}

\newcommand{\calA}{{\cal A}}
\newcommand{\calB}{{\cal B}}
\newcommand{\calC}{{\cal C}}

\newcommand{\calF}{{\cal F}}

\newcommand{\calK}{{\cal K}}

\newcommand{\calN}{{\cal N}}

\newcommand{\calP}{{\cal P}}

\newcommand{\calX}{{\cal X}}

\newcommand{\etal}{\emph{et~al.}\xspace}

\newcommand{\M}[1]{{\bm #1}} 
\renewcommand{\boldsymbol}[1]{{\bm #1}}

\newcommand{\wrt}{w.r.t.\xspace}

\newcommand{\hiddenText}{{\color{gray} hidden text.}}
\newcommand{\hideWithText}[1]{\hiddenText}

\newcommand{\kron}{\otimes}

\newcommand{\subject}{\text{ subject to }}

\DeclareMathOperator*{\argmin}{arg\,min}

\newcommand{\tran}{^{\mathsf{T}}}

\newcommand{\diag}[1]{\mathrm{diag}\left(#1\right)}
\newcommand{\trace}[1]{\mathrm{tr}\left(#1\right)}

\newcommand{\rank}[1]{\mathrm{rank}\left(#1\right)}

\newcommand{\inv}{^{-1}}

\newcommand{\zero}{{\mathbf 0}}
\newcommand{\eye}{{\mathbf I}}

\newcommand{\Real}[1]{ { {\mathbb R}^{#1} } }

\newcommand{\SOtwo}{\ensuremath{\mathrm{SO}(2)}\xspace}
\newcommand{\SOthree}{\ensuremath{\mathrm{SO}(3)}\xspace}
\newcommand{\Othree}{\ensuremath{\mathrm{O}(3)}\xspace}

\newcommand{\MA}{\M{A}}
\newcommand{\MB}{\M{B}}

\newcommand{\MD}{\M{D}}

\newcommand{\MM}{\M{M}}

\newcommand{\MQ}{\M{Q}}
\newcommand{\MU}{\M{U}}
\newcommand{\MR}{\M{R}}
\newcommand{\MS}{\M{S}}

\newcommand{\MV}{\M{V}}

\newcommand{\MW}{\M{W}}

\newcommand{\va}{\boldsymbol{a}} 
\newcommand{\vh}{\boldsymbol{h}} 
\newcommand{\vb}{\boldsymbol{b}}

\newcommand{\vd}{\boldsymbol{d}}

\newcommand{\vg}{\boldsymbol{g}}

\newcommand{\vp}{\boldsymbol{p}}

\newcommand{\vr}{\boldsymbol{r}}

\newcommand{\vu}{\boldsymbol{u}}
\newcommand{\vv}{\boldsymbol{v}}
\newcommand{\vt}{\boldsymbol{t}}
\newcommand{\vxx}{\boldsymbol{x}} 
\newcommand{\vy}{\boldsymbol{y}}
\newcommand{\vw}{\boldsymbol{w}}

\newcommand{\vlambda}{\boldsymbol{\lambda}}
\newcommand{\vtheta}{\boldsymbol{\theta}}
\newcommand{\valpha}{\boldsymbol{\alpha}}
\newcommand{\vbeta}{\boldsymbol{\beta}}

\newcommand{\vepsilon}{\boldsymbol{\epsilon}}

\newcommand{\scenario}[1]{{\smaller \sf#1}\xspace}

\newcommand{\blue}[1]{{\color{blue}#1}}
\newcommand{\green}[1]{{\color{green}#1}}
\newcommand{\red}[1]{{\color{red}#1}}

\newcommand{\linkToPdf}[1]{\href{#1}{\blue{(pdf)}}}
\newcommand{\linkToPpt}[1]{\href{#1}{\blue{(ppt)}}}
\newcommand{\linkToCode}[1]{\href{#1}{\blue{(code)}}}
\newcommand{\linkToWeb}[1]{\href{#1}{\blue{(web)}}}
\newcommand{\linkToVideo}[1]{\href{#1}{\blue{(video)}}}
\newcommand{\linkToMedia}[1]{\href{#1}{\blue{(media)}}}
\newcommand{\award}[1]{\xspace} 

\newcommand{\vz}{\boldsymbol{z}}


%% file: shortcuts.tex
\newcommand{\final}[1]{#1\xspace}
\newcommand{\finalLC}[1]{#1\xspace}

\renewcommand{\subject}{\text{s.t.}} 

\newcommand{\barc}{\bar{c}}
\newcommand{\barcsq}{\barc^2}

\newcommand{\eg}{\emph{e.g.,}\xspace}
\newcommand{\ie}{\emph{i.e.,}\xspace}

\newcommand{\rorder}{\kappa}
\newcommand{\nchoosek}[2]{\left( \substack{#1 \\ #2}\right)}
\newcommand{\TLS}{\scenario{TLS}}
\newcommand{\GNC}{\scenario{GNC}}
\newcommand{\gnc}{\scenario{GNC}}

\newcommand{\ransac}{\scenario{RANSAC}}
\newcommand{\SPEED}{\scenario{SPEED}}

\newcommand{\Rgt}{\MR^\circ}
\newcommand{\sgt}{s^\circ}
\newcommand{\tgt}{\vt^\circ}
\newcommand{\inlier}{_{\scenario{in}}}
\newcommand{\outlier}{_{\scenario{out}}}
\newcommand{\vPhi}{\boldsymbol{\Phi}}
\newcommand{\vPsi}{\boldsymbol{\Psi}}
\newcommand{\chiinv}{\scenario{chi2inv}}

\newcommand{\singlerotation}{\scenario{SRA}}
\newcommand{\shapealign}{\scenario{SA}}
\newcommand{\pointcloud}{\scenario{PCR}}
\newcommand{\mesh}{\scenario{MR}}

\newcommand{\relaxtime}{t_{\scenario{relax}}}
\newcommand{\certifytime}{t_{\scenario{certify}}}

\newcommand{\vectorize}[1]{\text{vec}\left(#1 \right)}

\newcommand{\hatvtheta}{\hat{\vtheta}}

\newcommand{\hatvt}{\hat{\vt}}

\newcommand{\sumallpoints}{\sum_{i=1}^N}
\newcommand{\DRS}{\scenario{DRS}}
\newcommand{\PRS}{\scenario{PRS}}
\newcommand{\polyring}{\Real{}[\vxx]}
\newcommand{\polyringin}[1]{\Real{}[#1]}
\newcommand{\nnint}{\mathbb{Z}_+}
\newcommand{\monod}[1]{\left( #1 \right)_d}
\newcommand{\monoindeg}[2]{\left( #1 \right)_{#2}}
\newcommand{\monoleqd}[1]{\left[ #1 \right]_d}
\newcommand{\monoleq}[2]{\left[ #1 \right]_{#2}}
\newcommand{\sym}{\mathcal{S}}
\newcommand{\psd}{\sym_{+}}
\newcommand{\mineig}[1]{\lambda_1\left(#1 \right)}
\newcommand{\sub}{\bar{s}}
\newcommand{\xextend}{\vp}
\newcommand{\calxextend}{\calP}
\newcommand{\dimxextend}{\tilde{n}}
\newcommand{\hatxextend}{\hat{\xextend}}
\newcommand{\hatf}{\hat{f}}
\newcommand{\dimbasis}[2]{m_{#1}(#2)}
\newcommand{\sosin}[1]{\Sigma\left[#1\right]}
\newcommand{\sosindeg}[2]{\sosin{#1}_{#2}}
\newcommand{\rbasisset}{\calB}
\newcommand{\rbasis}[1]{\left[ #1 \right]_{\rbasisset}}
\newcommand{\dualVar}{\vd}

\newcommand{\measured}{\vy}
\newcommand{\moments}{\vz}
\newcommand{\moment}{z}
\newcommand{\setmeasured}{\mathcal{Y}}

\newcommand{\proj}{\text{proj}}
\newcommand{\bracket}[1]{\left[ #1 \right]}
\newcommand{\parentheses}[1]{\left( #1 \right)}
\newcommand{\cbrace}[1]{\left\{ #1 \right\}}
\newcommand{\svec}[1]{\text{svec}\left( #1 \right)}

\newcommand{\tldvz}{\tilde{\vz}}
\newcommand{\tldMB}{\tilde{\MB}}

\newcommand{\nrIneq}{l_g}
\newcommand{\nrEq}{l_h}

\newcommand{\indicator}{\mathbbm{1}}
\newcommand{\subopt}{\varepsilon}
\newcommand{\suboptbound}{\bar{\subopt}}
\newcommand{\twonorm}[1]{\left\| #1 \right\|}

\newcommand{\optional}[2]{#2\xspace}

\newcommand{\perception}{geometric perception\xspace}
\newcommand{\Perception}{Geometric perception\xspace}
\newcommand{\PERCEPTION}{Geometric Perception\xspace}
\newcommand{\supp}{Supplementary Material\xspace}

\newcommand{\NLS}{NLS\xspace}
\renewcommand{\deg}[1]{\text{deg}\left( #1 \right)}

\newcommand{\shrink}{\vspace{-2mm}}

\newcommand{\SAR}{\bar{\MR}}
\newcommand{\SAr}{\bar{\vr}}
\newcommand{\SAlr}{\bar{r}}
\newcommand{\bmat}{\left[ \begin{array}}
\newcommand{\emat}{\end{array}\right]}

\newcommand{\MRR}{\tilde{\MR}}
\newcommand{\MRr}{\tilde{\vr}}
\newcommand{\MRt}{\tilde{\vt}}
\newcommand{\MRlr}{\tilde{r}}
\newcommand{\MRlt}{\tilde{t}}

\newcommand{\barMS}{\bar{\MS}}

\newcommand{\monoindicator}{\MW_{\valpha}}
\newcommand{\monoindicatorr}{\MW}
\newcommand{\monoindicatorVec}{\vw_{\valpha}}
\newcommand{\inner}[2]{\left\langle #1, #2 \right\rangle}

\newcommand{\probin}[1]{\Omega\parentheses{#1}}
\newcommand{\prox}{\text{prox}}

\newcommand{\edit}[1]{#1\xspace}
\renewcommand{\natural}{\mathbb{N}}
\newcommand{\expect}[2]{\mathbb{E}_{#1}\bracket{#2}}
\newcommand{\myversus}{\emph{vs.}\xspace}

\newcommand{\TIMb}{\bar{\vb}}
\newcommand{\TIMB}{\bar{\MB}}
\newcommand{\TIMeps}{\bar{\vepsilon}}

\newcommand{\hatMR}{\hat{\MR}}
\newcommand{\hats}{\hat{s}}
\newcommand{\sgn}[1]{\text{sgn}\parentheses{#1}}
\newcommand{\Otwo}{\text{O}(2)}
\newcommand{\codelink}{\url{https://github.com/MIT-SPARK/CertifiablyRobustPerception}}

%% file: abstract.tex
\vspace{-6mm}
\begin{abstract}
 We propose \final{the first} general and practical framework to design \emph{certifiable algorithms} for robust \perception in the presence of a large amount of outliers. 
 We investigate the use of a \emph{truncated least squares} (\TLS) cost function, which is known to be robust to outliers, but leads to hard, {nonconvex}, and {nonsmooth} optimization problems. 
 Our first contribution is to show that\edit{~--for a broad class of \perception problems--} \TLS estimation can be reformulated as an optimization over the {ring of polynomials} and \emph{Lasserre's hierarchy of convex moment relaxations} is empirically tight at the \emph{minimum relaxation order} (\ie~{certifiably} obtains the \emph{global minimum} of the nonconvex \TLS problem). Our second contribution is to exploit the structural sparsity of the objective and constraint polynomials and leverage \emph{basis reduction} to significantly reduce the size of the {semidefinite program} (SDP) resulting from the moment relaxation, without compromising its tightness. Our third contribution is to develop scalable \emph{dual optimality certifiers} from the lens of \emph{sums-of-squares} (SOS) relaxation, that can 
compute the suboptimality gap and possibly certify global optimality
 of any candidate solution 
 (\eg~returned by {fast heuristics} such as \ransac or {graduated non-convexity}). 
 Our dual certifiers leverage \emph{Douglas-Rachford Splitting} 
 to solve a convex feasibility \optional{SDP and their convergence is further accelerated by solving a \emph{chordal sparse} SOS program to provide a high-quality initialization.}{SDP.} Numerical experiments across different perception problems, including \final{single rotation averaging, shape alignment, 3D point cloud and mesh registration, and} high-integrity satellite pose estimation, demonstrate the tightness of \edit{our} relaxations, the correctness of the certification, and the scalability of the proposed dual certifiers to large problems, beyond the reach of current SDP solvers.\footnote{Code available at \codelink.}
\end{abstract}

%% file: introduction.tex
\shrink
\section{Introduction}
\label{sec:introduction}
\shrink
\emph{\Perception}, estimating unknown geometric models (\eg~rotations, poses, 3D structure) from visual measurements (\eg~images and point clouds), is a fundamental problem in computer vision, robotics, and graphics. 
It finds extensive applications to object detection and localization~\cite{Yang19rss-teaser,Yang20arxiv-teaser}, motion estimation and 3D reconstruction~\cite{Choi15cvpr-robustReconstruction,Zhang15icra-vloam}, simultaneous localization and mapping~\cite{Cadena16tro-SLAMsurvey,Rosen18ijrr-sesync}, shape analysis~\cite{Maron16tog-PMSDP,Ovsjanikov12TOG-functionalMaps}, virtual and augmented reality~\cite{Klein07ismar-PTAM}, and medical imaging~\cite{Audette00mia-surveyMedical}.

A common formulation for \perception resorts to optimization to perform estimation:
\bea
\min_{\vxx \in \calX}  \;\; \textstyle  \sumallpoints \rho \left( r\left( \vxx, \measured_i \right) \right), \label{eq:generalPerception}
\eea
where $\measured_i \in \setmeasured,i=1,\dots,N,$ are the visual measurements, 
$\vxx \in \calX \subseteq \Real{n}$ is the to-be-estimated geometric model, $r: \calX \times \setmeasured \rightarrow \Real{}_+$ is the \emph{residual  function} that quantifies the disagreement between each measurement $\measured_i$ and the geometric model $\vxx$, 
and $\rho: \Real{}_+ \rightarrow \Real{}_+$ is the \emph{cost function} that determines how residuals are penalized. When the distribution of the 
measurement noise is known, maximum likelihood estimation provides a systematic way to design $\rho$; for instance, assuming Gaussian noise leads to the popular \emph{least squares} cost function $\rho(r) = r^2$~\cite{Horn87josa,Poggio85nature-computationalVision,Hartley04}. 
However, in practice, \emph{a large amount of} measurements, called \emph{outliers}, depart from the assumed noise distribution (\eg~due to sensor failure or incorrect data association). Therefore, a \emph{robust} cost function, such as the $\ell_1$-norm~\cite{Wang13ima}, Huber~\cite{Huber81}, Geman-McClure~\cite{Yang20ral-GNC}, and truncated least squares~\cite{Yang19iccv-QUASAR}, is necessary to prevent the outliers from corrupting the estimate. Both the constraints --defining the domain $\calX$-- 
and the objective function in~\eqref{eq:generalPerception}  are typically nonconvex in \perception problems.


Solving \perception with \emph{optimality guarantees} is of paramount importance for {safety-critical} and {high-integrity} applications such as autonomous driving and space robotics. Indeed, suboptimal solutions of~\eqref{eq:generalPerception} typically correspond to poor or outlier-contaminated estimates~\cite{Yang20ral-GNC}. 
However, obtaining globally optimal solutions, particularly in the presence of outliers, remains a challenging task. Related work is divided into (i) \emph{fast heuristics},~\eg~\ransac~\cite{Fischler81} and graduated non-convexity (\GNC)~\cite{Yang20ral-GNC}, that are efficient but brittle against high outlier rates and offer no optimality guarantees, and (ii) \emph{global solvers},~\eg~{Branch and Bound}~\cite{Izatt17isrr-MIPregistration,Yang16pami-goicp}, that guarantee optimality but run in worst-case exponential time. Recently, \emph{certifiable algorithms}~\cite{Bandeira16crm,Yang20arxiv-teaser,Boumal16nips,Carlone16tro-duality2D} are rising as a new paradigm for solving \perception with both \emph{a posteriori} optimality guarantees 
and polynomial-time complexity.
A popular framework for constructing a certifiable algorithm requires (i) a \emph{tight} \emph{convex relaxation} of problem~\eqref{eq:generalPerception}; (ii) a \emph{fast heuristics} that computes a candidate solution to problem~\eqref{eq:generalPerception} with high probability of success; and (iii) a fast \emph{duality-based} \emph{certifier} that verifies if the candidate solution is globally optimal for the relaxation.\footnote{Global optimality of the relaxation implies global optimality of problem~\eqref{eq:generalPerception} when the relaxation is tight.} However, although a growing body of tight convex relaxations
have been discovered for various instances of \perception \emph{without} outliers~\cite{Kahl07IJCV-GlobalOptGeometricReconstruction,Briales17cvpr-registration,Briales18cvpr-global2view,Rosen18ijrr-sesync,Eriksson18cvpr-strongDuality,Zhao20cvpr-certifiablyEssential,Yang20cvpr-shapeStar,Probst19ICCV-convexRelaxationNonminimal,Maron16tog-PMSDP,Chaudhury15Jopt-multiplePointCloudRegistration,Aholt12ECCV-qcqpTriangulation,Fredriksson12accv-rotationaveragingLagrangian,Agostinho2019arXiv-cvxpnpl,Giamou19ral-SDPExtrinsicCalibration,Heller14icra-handeyePOP,Wise20arXiv-certifiablyHandeye}, only a few (problem-specific) tight relaxations exist for \emph{outlier-robust} \perception~\cite{Yang19rss-teaser,Yang19iccv-QUASAR,Wang13ima,Lajoie19ral-DCGM,Carlone18ral-robustPGO2D}\optional{. Moreover, fast certifiers exist only when the dual variables can be computed in closed-form from KKT optimality conditions~\cite{Carlone16tro-duality2D,Eriksson18cvpr-strongDuality,Iglesias20cvpr-PSRGlobalOptimality,Garcia20arXiv-certifiableRelativePose}, except Yang~\etal~\cite{Yang20arxiv-teaser} proposed a certifier to numerically compute dual certificates in the case of an under-constrained KKT system.}{.} 

\input{fig-summary}

{\bf Contributions.} We contribute \final{the first} general and practical framework for designing certifiable algorithms for robust \perception with outliers \final{(Fig.~\ref{fig:summary})}. Our first contribution is to show that common \perception problems with the truncated least squares (\TLS) cost function can be reformulated as an optimization over the ring of polynomials, and \emph{Lasserre's hierarchy of moment relaxations}~\cite{Lasserre01siopt-LasserreHierarchy,lasserre10book-momentsOpt} is tight at the \emph{minimum relaxation order}, \finalLC{despite the strong non-convexity and non-smoothness of the problem.} Our second contribution is to propose a \emph{basis reduction} technique, that exploits the structural \emph{sparsity} of the polynomials and significantly reduces the size of the semidefinite programs (SDP) resulting from moment relaxation, \finalLC{while surprisingly maintaining tightness of the relaxation.} These two contributions lead to the {first} set of \emph{certifiably robust} solvers for a broad class of \perception problems. 
While scaling better than the standard moment relaxation, these solvers still rely on existing SDP solvers, whose runtime restricts 
their use to small-scale problems (\eg~$N=20$). 
Therefore, our third contribution is to study the \emph{dual} sums-of-squares (SOS) relaxation and design fast \emph{dual optimality certifiers} that scale to realistic problem sizes (\eg~$N=100$). 
Our certifiers leverage \emph{Douglas–Rachford Splitting} (\DRS), \final{initialized by solving an SOS program with \emph{correlative sparsity}~\cite{Waki06jopt-SOSSparsity,Wang20arXiv-cs-tssos},} to compute a \emph{suboptimality} gap for any candidate solution, 
and possibly certify \emph{global optimality}
when the suboptimality is zero. 
\finalLC{Dual certifiers enhance existing heuristics (\eg~\ransac and \gnc) with a fast certification that asserts the quality of their estimates and rejects failure cases, thus enhancing trustworthiness in  safety-critical applications.} We demonstrate our tight relaxations and fast certifiers on several perception problems including single rotation averaging~\cite{Hartley13ijcv,Lee20arXiv-robustSRA}, image-based pose estimation (also called shape alignment)~\cite{Yang20ral-GNC}, point cloud registration~\cite{Yang19rss-teaser}, mesh registration~\cite{Briales17cvpr-registration}, and in a satellite pose estimation application~\cite{Chen19ICCVW-satellitePoseEstimation}.

{\bf Notation.} Let $\polyring$ be the ring of real-valued multivariate polynomials in $\{x_i\}_{i=1}^n$. 
Using standard notation~\cite{lasserre10book-momentsOpt}, we denote every $f \in \polyring$ as $f = \sum_{\valpha \in \calF} c(\valpha) \vxx^{\valpha}$, where $\calF \subseteq \nnint^n$ is a finite set of nonnegative integer exponents, $c(\valpha)$ are real coefficients, and $\vxx^\valpha \doteq x_1^{\alpha_1}x_2^{\alpha_2}\cdots x_n^{\alpha_n}$ are standard monomials. The degree of a monomial $\vxx^{\valpha}$ is $\deg{\vxx^{\valpha}} \doteq \sum_{i=1}^n \alpha_i$, and the degree of a polynomial $f$ is $\deg{f} = \max\{\deg{\vxx^{\valpha}}: \valpha \in \calF\}$. 
We use $\monod{\vxx}$ (resp. $\monoleqd{\vxx}$) to denote the set of monomials with degree $d$ (resp. with degree up to $d$). We use $\dimbasis{n}{d} \doteq \nchoosek{n+d}{d}$ to denote the dimension of $\monoleqd{\vxx}$. Similarly, we use $\monoleq{\vxx}{\calF} \doteq \{\vxx^\valpha: \valpha \in \calF \}$ to denote the set of monomials with exponents in $\calF$, and we use $\dimbasis{}{\calF}$ to denote its dimension. 
We use $\sym^n$ to denote the set of $n \times n$ symmetric matrices, and $\psd^n$ for the set of symmetric positive semidefinite (PSD) matrices. We also write $\MA \succeq 0$ to indicate $\MA \in \psd^n$. 
For $\MA \in \sym^n$ we use $\svec{\MA}$ to denote its symmetric vectorization~\cite{Toh12handbook-SDPT3Implementation}.
A polynomial $q \in \polyring$ is a sums-of-squares (SOS) polynomial if and only if $q$ can be written as $q = \monoleq{\vxx}{\calF}\tran \MQ \monoleq{\vxx}{\calF}$ for some monomial basis $\monoleq{\vxx}{\calF}$ and PSD matrix $\MQ \succeq 0$, in which case $q \geq 0, \forall \vxx \in \Real{n}$.


%% file: fig-summary.tex
\begin{figure}
\centering
\includegraphics[width=\columnwidth]{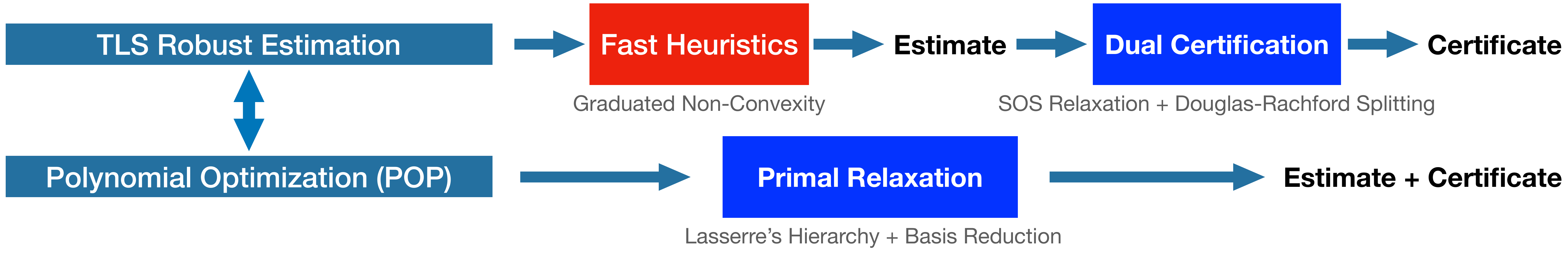}
\vspace{-6mm}
\caption{ A general and practical framework for certifiably robust geometric perception with outliers.}
\label{fig:summary}
\vspace{-6mm}
\end{figure}

%% file: relatedWork.tex

\shrink
\section{Related Work}
\label{sec:relatedWork}
\shrink


{\bf Outlier-free \PERCEPTION} algorithms can be divided into \emph{minimal solvers} and \emph{non-minimal solvers}. Minimal solvers assume \emph{noiseless} measurements (\ie~$r(\vxx,\measured_i)=0,\forall \; i$ in~\eqref{eq:generalPerception}) and use the minimum number of measurements necessary to estimate $\vxx$, which leads to solving a system of polynomial equations~\cite{PajdlaXXwebsite-minimalProblemsInVision,Kukelova2008ECCV-automaticGeneratorofMinimalProblemSolvers,Gao03PAMI-P3P,Nister04pami}. 
Non-minimal solvers account for measurement noise and estimate $\vxx$ via nonlinear least squares (\NLS), \ie~$\rho(r) = r^2$ in~\eqref{eq:generalPerception}.
While in rare cases NLS can be solved in closed form~\cite{markley1988jas-svdAttitudeDeter,Markley14book-fundamentalsAttitudeDetermine,Horn87josa,Arun87pami} or by solving the polynomial equations arising from the first-order optimality conditions~\cite{Wientapper18cviu-absolutePose,Kneip2014ECCV-UPnP,Zhou20ICRA-GRegAlgebraicSolver}, in general 
they lead to
nonconvex problems and are attacked using local solvers~\cite{Kuemmerle11icra,Agarwal10eccv} 
or exponential-time methods (\eg \emph{Branch and Bound}~\cite{Olsson09pami-bnbRegistration,Hartley09ijcv-globalRotationRegistration}).

\emph{Certifiable algorithms} for outlier-free perception have recently emerged as an approach to compute globally optimal 
\NLS
solutions in polynomial time.
These algorithms
relax
the \NLS minimization
into a convex optimization, using Shor's semidefinite relaxation for \emph{quadratically constrained quadratic programs}~\cite{Goemans95JACM-maxCutBound,Luo2010SP-sdpRelaxationQuadratic} or Lasserre's hierarchy of moment relaxations for \emph{polynomial optimizations}~\cite{lasserre10book-momentsOpt}. By solving the SDP resulting from the convex relaxations, 
certifiable algorithms compute global solutions to \NLS problems and provide a certificate of optimality, 
which usually depends on the rank of the SDP solution or the duality gap.
Empirically tight convex relaxations have been discovered in pose graph optimization~\cite{Carlone16tro-duality2D,Rosen18ijrr-sesync}, rotation averaging~\cite{Eriksson18cvpr-strongDuality,Fredriksson12accv}, triangulation~\cite{Aholt12ECCV-qcqpTriangulation}, 3D registration~\cite{Briales17cvpr-registration,Maron16tog-PMSDP,Chaudhury15Jopt-multiplePointCloudRegistration}, absolute pose estimation~\cite{Agostinho2019arXiv-cvxpnpl}, relative pose estimation~\cite{Briales18cvpr-global2view,Zhao20cvpr-certifiablyEssential}, hand-eye calibration~\cite{Heller14icra-handeyePOP} and 3D shape reconstruction from 2D landmarks~\cite{Yang20cvpr-shapeStar}. More recently, theoretical analysis of when and why the relaxations are tight is also emerging~\cite{Aholt12ECCV-qcqpTriangulation,Eriksson18cvpr-strongDuality,Rosen18ijrr-sesync,Cifuentes17arxiv,Zhao19arxiv-efficientTwoView,Chaudhury15Jopt-multiplePointCloudRegistration,Dym17Jopt-exactPMSDP,Iglesias20cvpr-PSRGlobalOptimality}.
Tight relaxations also enable
optimality certification (\ie checking if a given solution is optimal), which --in outlier-free perception-- can be \edit{sometimes} performed in closed form~\cite{Carlone16tro-duality2D,Eriksson18cvpr-strongDuality,Garcia20arXiv-certifiableRelativePose,Boumal16nips,Burer03mp,Rosen20wafr-scalableLowRankSDP,Cifuentes19arXiv-BMguarantees,Iglesias20cvpr-PSRGlobalOptimality}. \finalLC{Despite being certifiably optimal, these solvers assume all measurements are inliers (\ie~have small noise), which rarely occurs in practice, and hence give poor estimates even in the presence of a single outlier. In stark contrast, this paper develops certifiable algorithms in the presence of large amounts of outliers.}

{\bf Robust \PERCEPTION} algorithms can be divided into \emph{fast heuristics} and \emph{globally optimal solvers}. Two general frameworks for designing fast heuristics are \ransac~\cite{Fischler81} and \emph{graduated non-convexity} (\GNC)~\cite{Yang20ral-GNC,Black96ijcv-unification,Antonante20arxiv-outlierRobustEstimation}. \ransac robustifies minimal solvers and acts as a fast heuristics to solve \emph{consensus maximization}~\cite{Chin17slcv-maximumConsensusAdvances,Tzoumas19iros-outliers}, while \GNC robustifies non-minimal solvers and acts as a fast heuristics to solve \emph{M-estimation} (\ie~using a robust cost function $\rho$ in~\eqref{eq:generalPerception})~\cite{Bosse17fnt}. Local optimization 
is also a popular and fast heuristics~\cite{Chatterjee13iccv,Hartley11cvpr-l1rotationaveraging,Schonberger16cvpr-SfMRevisited,Bouaziz13acmsig-sparseICP,Agarwal13icra,Crandall11cvpr} for the case where an initial guess is available.
On the other hand, globally optimal solvers are typically designed using Branch and Bound~\cite{Bazin12accv-globalRotSearch,Bustos18pami-GORE,Izatt17isrr-MIPregistration,Jiao20arXiv-VIOpointline,Yang2014ECCV-optimalEssentialEstimationBnBConsensusMax}, or boost 
robustness via a preliminary outlier-pruning scheme~\cite{Yang19rss-teaser,Bustos18pami-GORE}.

\emph{Certifiably robust algorithms} relax problem~\eqref{eq:generalPerception} with a robust cost into a tight convex optimization. 
While certain robust costs, such as the $\ell_1$-norm~\cite{Wang13ima} and Huber~\cite{Carlone18ral-robustPGO2D}, are already convex, they have low breakdown points (\ie 
they can be compromised by a single outlier)~\cite{Yu12NIPS-robustRegression,Maronna19book-robustStats}. A few problem-specific certifiably robust algorithms have been proposed 
to deal with high-breakdown-point formulations, such as the \TLS cost~\cite{Yang19iccv-QUASAR,Yang19rss-teaser,Bohorquez20arXiv-exactRelaxationRobustRegistration,Lajoie19ral-DCGM}. 
Even optimality certification becomes harder and problem-specific in the presence of outliers, due to the lack of a closed-form characterization of the 
dual variables~\cite{Yang20arxiv-teaser}.
In this paper, we introduce a \emph{general} framework \edit{to design} certifiably robust algorithms and \edit{optimality certifiers} for a broad class of \perception problems with 
\TLS cost.


\optional{{\bf Optimality Certification},~\ie~certifying the global optimality or declaring suboptimality of any candidate solution returned by NLS and fast heuristics, becomes tractable when the convex relaxation is tight. More importantly, it allows designing duality-based optimality certifiers without solving large-scale SDPs. In outlier-free perception, optimality certification can sometimes be performed in closed form~\cite{Carlone16tro-duality2D,Eriksson18cvpr-strongDuality,Garcia20arXiv-certifiableRelativePose}, including the seminal Burer-Monteiro approach~\cite{Boumal16nips,Burer03mp,Rosen20wafr-scalableLowRankSDP,Cifuentes19arXiv-BMguarantees}. While in robust perception, optimality certification relies on problem-specific analysis~\cite{Yang20arxiv-teaser}. In this paper, we introduce a general approach to design \emph{scalable} dual optimality certifiers by attacking the dual SOS relaxation with Douglas-Rachford Splitting.}{}


%% file: pop.tex
\shrink
\section{Robust \PERCEPTION as Polynomial Optimization}
\label{sec:pop}
\shrink

In this paper we develop certifiable algorithms to solve~\eqref{eq:generalPerception} 
 for the case when the cost $\rho$ is a  \emph{truncated least squares} (\TLS) cost:
\begin{equation} \label{eq:generalTLS}
f^\star = \min_{\vxx \in \calX} \; \; \sumallpoints \min \left\{ \frac{r^2(\vxx,\measured_i)}{\beta_i^2}, \barcsq \right\}, \tag{\TLS}
\end{equation}
where $\min \{\cdot,\cdot\}$ denotes the minimum between two scalars, $\beta_i$ is a known constant that can be 
used to model the inlier standard deviation (potentially different for each measurement $i$), 
and $\barc$ is the maximum admissible residual for a measurement to be considered an inlier. 
Intuitively, problem~\eqref{eq:generalTLS} implements a nonlinear least squares where measurements with large residuals
 (\ie outliers) do not influence the estimate (\ie lead to a constant cost of $\barcsq$).
Problem~\eqref{eq:generalTLS}
is known to be robust to large amounts of outliers~\cite{Yang10NIPS-relaxedClipping,Maronna19book-robustStats}.
However, its global minimum
$f^\star$ is hard to compute due to the non-convexity and non-smoothness of the cost (which adds to the typical non-convexity of the 
domain $\calX$).
In the following, we briefly review a few instantiations of robust \perception.

\begin{example}[Single Rotation Averaging~\cite{Hartley13ijcv}] \label{eg:singleRotationAveraging}
Given $N$ measurements of an unknown 3D rotation: $\MR_i, i=1,\dots,N$, single rotation averaging seeks to find the best \emph{average} rotation $\MR$. In this case, $\vxx = \MR \in \SOthree$, $\measured_i = \MR_i$, and the residual function can be chosen as $r(\vxx,\measured_i) = \left\| \MR - \MR_i \right\|_F$ (the {chordal distance} between two rotations~\cite{Hartley13ijcv}), where $\| \cdot \|_F$ denotes the Frobenius norm.
\end{example}


\begin{example}[Shape Alignment~\cite{Yang20ral-GNC}] \label{eg:shapeAlignment}
Given a set of 3D points $\MB_i \in \Real{3}$ and a set of 2D pixels $\vb_i \in \Real{2}$ ($i=1,\dots,N$), with putative correspondences $\vb_i \leftrightarrow \MB_i$, shape alignment seeks to find the best scale $s \in [0,\sub]$ (where $\sub$ is a given upper bound for the scale)
and 3D rotation $\MR \in \SOthree$ of the point set, 
such that the 3D points project onto the corresponding pixels.
In this case, $\vxx = (\MR,s)$, $\measured_i = (\MB_i,\vb_i)$ and the residual function is the reprojection error under the weak perspective camera model: $r(\vxx,\measured_i) = \| \vb_i - s\Pi\MR\MB_i \|$, where $\Pi = [1,0,0;0,1,0] \in \Real{2\times 3}$. 
\end{example}

\begin{example}[Point Cloud Registration~\cite{Yang19rss-teaser}] \label{eg:pointCloudRegistration}
Given two sets of 3D points $\va_i,\vb_i \in \Real{3}, i=1,\dots,N$, with putative correspondences $\va_i \leftrightarrow \vb_i$, point cloud registration seeks the best 3D rotation $\MR \in \SOthree$ and translation $\vt \in \Real{3}$ to align them.\footnote{\label{footnote1}For mathematical convenience, we assume the translation is bounded by a known value $T$, \ie $\| \vt \| \leq T$.} 
In this case, $\vxx = (\MR, \vt)$, $\measured_i = (\va_i,\vb_i)$ and the residual function is the Euclidean distance between registered pairs of points: $r(\vxx, \measured_i) = \|  \vb_i - \MR\va_i - \vt  \|$. 
\end{example}

\begin{example}[Mesh Registration~\cite{Briales17cvpr-registration}] \label{eg:meshRegistration}
Consider a 3D mesh $\{\va_i,\vu_i\}_{i=1}^N$ and a 3D point cloud with estimated normals $\{\vb_i,\vv_i\}_{i=1}^N$, where $\va_i \in \Real{3}$ is an arbitrary point on a face of the mesh, and $\vu_i$ is the unit normal of the same face, while $\vb_i \in \Real{3}$ is a 3D point and $\vv_i$ is the estimated unit normal at \edit{$\vb_i$}. 
Given putative correspondences $(\va_i,\vu_i) \leftrightarrow (\vb_i,\vv_i)$,
mesh registration seeks the best 3D rotation $\MR \in \SOthree$ and translation $\vt \in \Real{3}$ to align the mesh with the point cloud.\footref{footnote1}
 In this case, $\vxx = (\MR, \vt)$, $\measured_i = (\va_i,\vu_i,\vb_i,\vv_i)$, and the residual function is the weighted sum of the point-to-plane distance and normal-to-normal distance: $r^2(\vxx,\measured_i) = \| (\MR\vu_i)\tran(\vb_i - \MR\va_i - \vt) \|^2 + w_i\| \vv_i - \MR \vu_i \|^2$, where $w_i>0$ is the relative weight between normal-to-normal distance and point-to-plane distance.
\end{example}

The following proposition states that all the four examples above lead \edit{to~\eqref{eq:generalTLS} problems} that can be cast as polynomial optimization problems (POPs).

\begin{proposition}[\PERCEPTION as POP]\label{prop:pop}
Robust \perception~\eqref{eq:generalTLS}, with residual functions as in Examples~\ref{eg:singleRotationAveraging}-\ref{eg:meshRegistration}, is equivalent to the following polynomial optimization (POP):
\bea
f^\star = \min_{\xextend \in \Real{\dimxextend}} & f(\vp)  \label{eq:pop}\\
\subject & \quad \quad h_j(\vp) = 0, j = 1,\dots,\nrEq,  \nonumber \\
& 1 \geq g_k(\vp) \geq 0, k = 1,\dots,\nrIneq, \nonumber
\eea
with $\dimxextend \doteq n + N,$ and $\xextend \doteq [\vxx\tran,\vtheta\tran]\tran \in \Real{\dimxextend}$, 
where $\vxx \in \calX$ contains the to-be-estimated geometric model, and the vector of binary variables $\vtheta \in \{\pm 1 \}^N$ 
is such that $\theta_i = +1$ (resp. $\theta_i = -1$) when the $i$-th measurement $\measured_i$ is estimated to be an inlier (resp. outlier). In this POP, $f$ is a polynomial in $\xextend$ with $\deg{f} \leq 3$, while $h_j,g_k$ are quadratic (degree-2) polynomials 
 in $\xextend$ that are used to define the domains $\calX$ and $\{\pm 1 \}^N$.
 The polynomials $f,h_j,g_k$ possess the following structural properties:
\begin{enumerate}[label=(\roman*)]
\item \label{prop:pop-objective} (objective function sparsity) $f$ can be written as a sum of $N$ polynomials $f_i,i=1,\dots,N$, and each $f_i$ is a polynomial in $\vxx$ and $\theta_i$ of degree lower or equal to 3,~\ie~$f = \sumallpoints f_i, f_i \in \polyringin{\vxx,\theta_i},\deg{f_i} \leq 3$;

\item \label{prop:pop-constraint} (constraints sparsity) let $\vh \doteq \{h_j\}_{j=1}^{l_h}$ and $\vg = \{ g_k \}_{k=1}^{l_g}$. 
Then, $\vg \subset \polyringin{\vxx}$ are polynomials in $\vxx$ (\ie do not depend on $\vtheta$).
Moreover, $\vh$ can be partitioned into $N+1$ disjoint subsets: $\vh = \vh^{\theta}  \cup \vh^{x}$, with $\vh^{\theta} = \cup_{i=1}^N \vh^{\theta_i}$, where $\vh^{\theta_i} \subset \polyringin{\theta_i}$ are polynomials in $\theta_i$ (\ie do not depend on $\vxx$ and $\theta_j,\forall j \neq i$), $\vh^x \subset \polyringin{\vxx}$ are polynomials in $\vxx$ (\ie do not depend on $\vtheta$);

\item \label{prop:pop-archimedean} (Archimedeanness) the feasible set $\calxextend$ of the POP~\eqref{eq:pop} is \emph{Archimedean}.\footnote{Archimedeanness is a stronger condition than compactness, see~\cite[Definition 3.137, p.~115]{Blekherman12Book-sdpandConvexAlgebraicGeometry}.}


\end{enumerate} 
\end{proposition}
\vspace{-2mm}
The \supp provides a proof of Proposition~\ref{prop:pop} and the expressions of $f,h_j,g_k$ for Examples~\ref{eg:singleRotationAveraging}-\ref{eg:meshRegistration}. Proposition~\ref{prop:pop} is based on three insights. \edit{First, each inner minimization $\min \{a,b\}$ ($a,b \in \Real{}$) 
can be written as $\substack{\min \\ \theta \in \{\pm 1\}} \frac{1+\theta}{2} a + \frac{1-\theta}{2} b$, 
which gives rise to
the
binary variables and leads to the objective sparsity in~\ref{prop:pop-objective}. 
Second, the constraint sets of $\vxx$ and each $\theta_i$ are mutually independent, and can be described by quadratic equality and inequality constraints, leading to the constraints sparsity in~\ref{prop:pop-constraint}. Third, the unknown variables, including $\MR \in \SOthree, s \in [0,\sub], \| \vt \| \leq T$, and $\theta_i \in \{\pm 1\}$, live in compact domains described by polynomials, leading to the Archimedeanness property~\ref{prop:pop-archimedean}.}


%% file: primal-moment.tex
\shrink
\section{The Primal View: Tight Moment Relaxation}
\label{sec:primalMoment}
\shrink
In this section, we develop dense (Section~\ref{sec:lasserreHierarchy}) and sparse (Section~\ref{sec:basisReduction}) convex moment relaxations to the POP~\eqref{eq:pop}. The dense relaxation is a standard application of Lasserre's hierarchy~\cite{Lasserre01siopt-LasserreHierarchy,lasserre10book-momentsOpt}, while the sparse relaxation is based on a basis reduction that leverages the structural properties in 
Proposition~\ref{prop:pop}.

\subsection{Lasserre's Hierarchy}
\label{sec:lasserreHierarchy}

\optional{We first describe some definitions for stating Lasserre's hierarchy. Let $\mu_{\xextend}$ be any probability distribution supported on the feasible set $\calxextend$ of the POP~\eqref{eq:pop}, and let $\moments_\rorder = \{ \moment_{\valpha} \} \in \Real{\dimbasis{\dimxextend}{\rorder}}$ be the vector of \emph{moments} up to degree $\rorder$, where $\moment_{\valpha} \doteq \int \xextend^{\valpha} d\mu_{\xextend}$ for any $\xextend^{\valpha} \in [\xextend]_\rorder$. Clearly, $\moment_\zero = 1$. The \emph{moment matrix} $\MM_{\rorder}(\moments_{2\rorder}) \in \sym^{\dimbasis{\dimxextend}{\rorder}}$ is the assembly of $\moments_{2\rorder} \in \Real{\dimbasis{\dimxextend}{2\rorder}}$ into a symmetric matrix whose rows and columns are indexed by $\monoleq{\xextend}{\rorder}$, with the $(i,j)$-th entry being:
\bea
\left[  \MM_{\rorder}(\moments_{2\rorder}) \right]_{ij} = \moment_{\valpha_i + \valpha_j},
\eea  
where $\valpha_i$ and $\valpha_j$ are the exponents of the $i$-th and the $j$-th monomials in $\monoleq{\xextend}{\rorder}$. Given a moment matrix $\MM_{\rorder}(\moments_{2\rorder})$ and a polynomial $q = \sum_{\vbeta} c(\vbeta) \xextend^{\vbeta} \in \polyringin{\xextend}$, define the \emph{localizing matrix} to be $\MM_\rorder(q \moments_{2\rorder}) \in \sym^{\dimbasis{\dimxextend}{\rorder}}$, whose $(i,j)$-th entry is:
\bea
\left[ \MM_\rorder(q \moments_{2\rorder}) \right]_{ij} = \sum_{\vbeta} c(\vbeta) \moment_{\valpha_i + \valpha_j + \vbeta}.
\eea}

The following theorem describes  
Lasserre's hierarchy of dense moment relaxations for the POP~\eqref{eq:pop}.
\begin{theorem}[Dense Moment Relaxation~\cite{lasserre10book-momentsOpt}]\label{thm:denseMoment}
The dense moment relaxation at order $\rorder \; (\geq 2)$ for the POP~\eqref{eq:pop} is the following SDP:
\bea
p^\star_\rorder	 = \min_{\moments_{2\rorder} \in \Real{\dimbasis{\dimxextend}{2\rorder}} }& \sum_{\valpha \in \calF} c(\valpha) \moment_{\valpha} \label{eq:denseMoment} \\
\subject & \moment_{\zero} = 1, \MM_{\rorder}(\moments_{2\rorder}) \succeq 0, \nonumber \\
& \MM_{\rorder - 1}(h_j \moments_{2\rorder-2}) = \zero, j=1,\dots,\nrEq, \nonumber \\
& \MM_{\rorder - 1}(g_k \moments_{2\rorder-2}) \succeq 0, k = 1,\dots,\nrIneq. \nonumber
\eea
where $\moments_{2\rorder} = \{ \moment_{\valpha} \} \in \Real{\dimbasis{\dimxextend}{2\rorder}}$ is the vector of \emph{moments} up to degree $2\rorder$,
$c(\valpha)$ are the real coefficients of the objective function $f(\xextend)$ corresponding to monomials $\xextend^{\valpha}$ in~\eqref{eq:pop},
$\MM_{\rorder}(\moments_{2\rorder}) \in \sym^{\dimbasis{\dimxextend}{\rorder}}$ is the \emph{moment matrix}, and $\MM_{\rorder - 1}(h_j \moments_{2\rorder-2}),\MM_{\rorder - 1}(g_k \moments_{2\rorder-2}) \in \sym^{\dimbasis{\dimxextend}{\rorder-1}}$ are the \emph{localizing matrices}.\footnote{We refer the non-expert reader to~\cite{lasserre10book-momentsOpt} for a
comprehensive introduction to moment relaxations, and provide extra definitions and accessible examples in the \supp.\label{footnote2} }
 Let $\moments_{2\rorder}^\star$ be the optimal solution of~\eqref{eq:denseMoment}, then the following holds true: 
\begin{enumerate}[label=(\roman*)]
\item (lower bound) $p_\rorder^\star$ is a lower bound for $f^\star$,~\ie~$p_\rorder^\star \leq f^\star,\forall \rorder \geq 2$;
\item (finite convergence) $p_{\rorder_1}^\star \leq p_{\rorder_2}^\star $ for any $\rorder_1 \leq \rorder_2$, and $p_\rorder^\star = f^\star$ at some finite $\rorder$;
\item \label{prop:dense-certificate} (optimality certificate) if $\rank{\MM_{\rorder}(\moments_{2\rorder}^\star)} = 1$, then $\moments_{\rorder}^\star = \monoleq{\xextend^\star}{\rorder}$, where $\xextend^\star$ is the unique global minimizer of the POP~\eqref{eq:pop}, and the relaxation is said to be \emph{tight};
\item \label{prop:dense-gap} (rounding and duality gap) if $\rank{\MM_{\rorder}(\moments_{2\rorder}^\star)} > 1$, 
let $\hatxextend$ be a rounded estimate computed from a rank-1 approximation of $\MM_{\rorder}(\moments_{2\rorder}^\star)$,\footref{footnote2} 
and 
\optional{let $\hatxextend = {\proj_{\calxextend} (\vv_{\xextend})}$ where $\vv \in \Real{m(\rorder)}$ is the eigenvector of $\MM_{\rorder}(\moments_{2\rorder}^\star)$ corresponding to the largest eigenvalue, with the first entry normalized to 1, $\vv_{\xextend}$ extracts the sub-vector of $\vv$ at locations corresponding to $\xextend$ in $\monoleq{\xextend}{\rorder}$, and $\proj_{\calxextend}$ projects $\vv_{\xextend}$ to $\calxextend$ 
denote $\hatf = f(\hatxextend)$. }{denote $\hatf = f(\hatxextend)$.}
Then, $p_\rorder^\star \leq f^\star \leq \hatf$ and we say that the relative duality gap is $\eta_\rorder = (\hatf - p^\star_\rorder)/ \hatf$. 
\end{enumerate}
\end{theorem}

Theorem~\ref{thm:denseMoment} is a standard application of Lasserre's hierarchy~\cite{Lasserre01siopt-LasserreHierarchy} and the finite convergence result~\cite{Nie14mp-finiteConvergenceLassere} to problem~\eqref{eq:pop}. 
\optional{We remark that computing the rank of $\MM_{\rorder}(\moments_{2\rorder}^\star)$ is subject to numerical inaccuracy and instead we use the relative duality gap $\eta_\rorder$ as a metric for evaluating tightness. In Section~\ref{sec:experiments}, we show that the relaxation is 
empirically tight at the minimum relaxation order $\rorder = 2$.}{\finalLC{Although Lasserre's hierarchy is guaranteed to be tight at some finite $\rorder$, the relaxation becomes computationally impractical for large $\rorder$. Therefore, it is desirable to obtain tight relaxations with small $\rorder$.} In the \supp, we show that the dense moment relaxation is 
empirically tight at the \emph{minimum} relaxation order $\rorder = 2$ for Examples~\ref{eg:singleRotationAveraging}-\ref{eg:meshRegistration}, \finalLC{despite the fact that the POPs have both binary variables (a notoriously challenging setup~\cite{lasserre01ipco-lasserrehierarchybinary}) and non-convex constraints $\MR \in \SOthree$.} } 

\subsection{Basis Reduction}
\label{sec:basisReduction}
Although the dense relaxation is tight at $\rorder = 2$, the size of the SDP~\eqref{eq:denseMoment} (\ie~the size of the moment matrix $\MM_{\rorder}(\moments_{2\rorder})$ for $\rorder=2$) is $\nchoosek{n+N+2}{2}$, which grows \emph{quadratically} in the number of measurements $N$ and quickly becomes intractable even for small $N$ (\eg~$N=20$). In this section, we exploit the \emph{monomial sparsity} of the POP~\eqref{eq:pop} and use basis reduction to construct a sparse moment relaxation whose size grows linearly with $N$.

\begin{theorem}[Sparse Moment Relaxation]\label{thm:sparseMoment}
\edit{Define $\rbasis{\xextend} \doteq [1,\vxx\tran,\vtheta\tran,\monoindeg{\vxx}{2}\tran,\vtheta\tran \kron \vxx\tran]\tran$ to be a reduced set of monomials, with $\rbasisset$ being the set of monomial exponents in $\rbasis{\xextend}$,~\ie~$\rbasisset \doteq \{\valpha \in \nnint^{\dimxextend}: \xextend^{\valpha} \in  \rbasis{\xextend} \}$. Similarly, define $\monoleq{\xextend}{\rbasisset_x} \doteq [1,\vxx\tran]\tran$ and let $\rbasisset_x$ be its set of exponents. 
Let $2\rbasisset \doteq \{ \valpha \in \nnint^{\dimxextend}: \valpha = \valpha_1 + \valpha_2, \valpha_1,\valpha_2 \in \rbasisset \}$ (resp. $2\rbasisset_x$) be the Minkowski sum of $\rbasisset$ (resp. $\rbasisset_x$) with itself.
Define $\moments_{2\rbasisset} \in \Real{\dimbasis{}{2\rbasisset}}$ (resp. $\moments_{2\rbasisset_x} \in \Real{\dimbasis{}{2\rbasisset_x}} $) to be the vector of moments for all monomials in $\left[\xextend\right]_{2\rbasisset}$ (resp. $\monoleq{\xextend}{2\rbasisset_x}$), and $\MM_{\rbasisset}(\moments_{2\rbasisset}) \in \sym^{\dimbasis{}{\rbasisset}}$ 
(resp. $\MM_{\rbasisset_x}(\moments_{2\rbasisset_x}) \in \sym^{\dimbasis{}{\rbasisset_x}}$) to be the moment matrix that assembles $\moments_{2\rbasisset}$ (resp. $\moments_{2\rbasisset_x}$) in rows and columns indexed by $\rbasis{\xextend}$ (resp. $\monoleq{\xextend}{\rbasisset_x}$). }
Then, the sparse moment relaxation is:
\bea
p_{\rbasisset}^\star = \min_{\moments_{2\rbasisset} \in \Real{\dimbasis{}{2\rbasisset}}} & \sum_{\valpha \in \calF} c(\valpha) \moment_{\valpha} \label{eq:sparseMoment} \\
\subject & \moment_\zero = 1, \MM_{\rbasisset}(\moments_{2\rbasisset}) \succeq 0, \nonumber \\
& \MM_1(h \moments_2) = \zero, \forall h \in \vh^{x}; \quad \MM_{\rbasisset_x}(h \moments_{2\rbasisset_x}) = \zero, \forall h \in \vh^{\theta}, \nonumber \\
& \MM_1(g \moments_2) \succeq 0, \forall g \in \vg, \nonumber
\eea
where $\vh^x,\vh^{\theta},\vg$ are defined as in Proposition~\ref{prop:pop}. Moreover, we have $p^\star_{\rbasisset} \leq p^\star_2 \leq f^\star$ and properties~\ref{prop:dense-certificate}-\ref{prop:dense-gap} in Theorem~\ref{thm:denseMoment}
hold for the sparse relaxation~\eqref{eq:sparseMoment}.
\end{theorem}
The key idea behind Theorem~\ref{thm:sparseMoment} is to reduce the size of the SDP by only considering the reduced monomial basis $\rbasis{\xextend}$, which essentially removes all the monomials of the form $\theta_i\theta_j$ that do not appear in $f$ as per property \ref{prop:pop-objective} in Proposition~\ref{prop:pop}.
The size of the SDP~\eqref{eq:sparseMoment} (\ie~the size of $\MM_{\rbasisset}(\moment_{2\rbasisset})$) is $\dimbasis{}{\rbasisset} = \frac{(n+1)(n+2)}{2} + (1+n)N$, which grows linearly in $N$. In Section~\ref{sec:experiments}, we show that the sparse moment relaxation~\eqref{eq:sparseMoment} is also tight, even in the presence of a large amount of outliers. \finalLC{Although there exist  other efficient sparse variants~\cite{Waki06jopt-SOSSparsity,Weisser18mpc-SBSOS} of Lasserre's hierarchy that exploit correlative sparsity, in the \supp we show they break the tightness at the minimum relaxation order and produce poor estimates. Nevertheless, they can be used to bootstrap our dual certifiers (Section~\ref{sec:drs}).}

%% file: dual-certification.tex
\shrink
\section{The Dual View: Fast Optimality Certification}
\label{sec:dualCertification}
\shrink

Despite scaling linearly in $N$, the sparse relaxation~\eqref{eq:sparseMoment} is still too large to be solved efficiently using current interior point methods (IPM)~\cite{Nesterov18book-convexOptimization} when $N > 20$. 
On the other hand, fast heuristics such as graduated non-convexity~\cite{Yang20ral-GNC} can compute globally optimal solutions to the POP~\eqref{eq:pop} with high probability of success. In this section, we show that, by taking the dual perspective of sums-of-squares (SOS) relaxations, we can develop efficient \emph{certifiers} to verify the optimality of a candidate solution $(\hatxextend,\hatf)$ for large $N$ (\eg~$N=100$), for which the SDP relaxation~\eqref{eq:sparseMoment} is not even implementable. 

\subsection{Sums-of-Squares Relaxation}
A candidate solution $(\hatxextend,\hatf)$ is globally optimal for the POP~\eqref{eq:pop} if and only if $f(\xextend) - \hatf \geq 0, \forall \xextend \in \calxextend$. 
However, testing \emph{nonnegativity} of a polynomial on a constraint set is NP-hard~\cite{Blekherman12Book-sdpandConvexAlgebraicGeometry}, so instead we test if the polynomial is SOS on the constraint set and provide a sufficient condition for global optimality.
\begin{theorem}[Sufficient Condition for Global Optimality] \label{thm:sufficientOptimalityCondition}
Given any candidate solution $(\hatxextend,\hatf)$ to the POP~\eqref{eq:pop}, if the following optimization is \emph{feasible} (\ie has at least one solution):
\bea
\text{find} & \vlambda_j^x \in \Real{\dimbasis{\dimxextend}{2}}, \vlambda_j^\theta \in \Real{\dimbasis{n}{2}}, \MS_0 \in \psd^{\dimbasis{}{\rbasisset}}, \MS_k \in \psd^{\dimbasis{\dimxextend}{1}}   \label{eq:sosfeasiblity}\\
 \subject &  \hspace{-3mm} \displaystyle f(\xextend)\! - \! \hatf \! - \!\!\!\! \sum_{h_j \in \vh^x} \!\! h_j \! \left(\monoleq{\xextend}{2}\tran\vlambda_j^x \right) \! - \!\!\!\! \sum_{h_j \in \vh^{\theta} } \!\! h_j \! \left(\monoleq{\vxx}{2}\tran \vlambda_j^\theta \right) \! = \! \monoleq{\xextend}{\rbasisset}\tran \MS_0 \monoleq{\xextend}{\rbasisset} \! + \!\! \sum_{k=1}^{\nrIneq} g_k \! \left(\monoleq{\xextend}{1}\tran \MS_k \monoleq{\xextend}{1} \right) \! ,\edit{ \forall \xextend}, \label{eq:matchingcoefficients}
\eea
then $\hatf$ (resp. $\hatxextend$) is the global minimum (resp. global minimizer) of the POP~\eqref{eq:pop}. Moreover, problem~\eqref{eq:sosfeasiblity} can be written compactly as a feasibility SDP:
\bea
\text{find} \quad \dualVar, \quad   
\subject \quad  \dualVar \in \calK \cap \calA, \label{eq:sdpfeasibility}
\eea 
where $\dualVar = [(\vlambda_1^x)\tran,\dots,(\vlambda^x_{|\vh^x|})\tran,(\vlambda_1^\theta)\tran,\dots,(\vlambda^\theta_{|\vh^\theta|})\tran,\svec{\MS_1}\tran,\dots,\svec{\MS_{\nrIneq}}\tran,\svec{\MS_0}\tran]\tran$ concatenates all variables in~\eqref{eq:sosfeasiblity}, $\calK$ defines a convex cone, and $\calA \doteq \left\{ \dualVar: \MA\dualVar = \vb \right\}$ defines an affine subspace, \edit{where $\vb$ is a vector and $\MA$ is a matrix} satisfying the \emph{partial orthogonality} property~\cite{Zheng18TAC-partialOrthogonality,Bertsimas13OMS-acceleratedSOSRelaxation}.
\end{theorem}

In the \supp, we provide a proof of Theorem~\ref{thm:sufficientOptimalityCondition}. 
Intuitively, 
 if problem~\eqref{eq:sosfeasiblity} is feasible, then for any $\xextend \in \calxextend$, the left-hand side of~\eqref{eq:matchingcoefficients} reduces to $f(\xextend) - \hatf$ (due to $h_j=0$) and the right-hand side of~\eqref{eq:matchingcoefficients} is nonnegative (due to $g_k \geq 0,\MS_0,\MS_k\succeq 0$), producing a certificate that \edit{$f(\xextend) \geq \hatf$}.
 The SOS relaxation~\eqref{eq:sosfeasiblity} also uses basis reduction and it is the dual of the sparse moment relaxation~\eqref{eq:sparseMoment}~\cite{lasserre10book-momentsOpt} with the constraint that $\hatf$ is the global optimum. In SDP~\eqref{eq:sdpfeasibility}, the convex cone $\calK$ corresponds to the PSD constraints in~\eqref{eq:sosfeasiblity} and the affine subspace $\calA$ corresponds to matching coefficients in the equality constraint~\eqref{eq:matchingcoefficients}.
 The partial orthogonality of $\MA$ is a property for SDPs resulting from SOS relaxations and allows efficient \emph{projection} onto the affine subspace $\calA$~\cite{Zheng18TAC-partialOrthogonality,Bertsimas13OMS-acceleratedSOSRelaxation}.

\subsection{Douglas-Rachford Splitting}
\label{sec:drs}
In this section, we propose a first-order method based on \emph{Douglas-Rachford Splitting} (\DRS)~\cite{Combettes11book-proximalSplitting,Jegelka13NIPS-DRSreflection} to solve~\eqref{eq:sdpfeasibility} at scale. 
\DRS iteratively solves~\eqref{eq:sdpfeasibility} by starting at an arbitrary initial point $\dualVar_0$, 
and performing the following  three-step updates (at each iteration $\tau \geq 0$):
\bea
(i)\  \dualVar_\tau^{\calK} = \proj_{\calK} \left(\dualVar_{\tau} \right) ,\quad  
(ii)\ \dualVar^{\calA}_{\tau} = \proj_{\calA} \left(2 \dualVar_\tau^{\calK} - \dualVar_{\tau} \right), \quad 
(iii)\ \dualVar_{\tau+1} = \dualVar_{\tau} + \gamma_{\tau} \left( \dualVar^{\calA}_{\tau} - \dualVar_\tau^{\calK} \right), \label{eq:DRSIterates}
\eea 
where $\proj_{\calK}$ (resp. $\proj_{\calA}$) denotes the orthogonal projection onto $\calK$ (resp. $\calA$) 
and $\gamma_{\tau}$ is a parameter of the algorithm.
\optional{
Towards this goal, we first rewrite problem~\eqref{eq:sdpfeasibility} equivalently as $\min_{\dualVar} \indicator_{\calA}(\dualVar) + \indicator_{\calK} (\dualVar)$, where $\indicator_{\calA}(\dualVar)$ is the indicator function for convex set $\calA$ (\ie~$\indicator_{\calA}(\dualVar) = 0$ if $\dualVar \in \calA$ and $\indicator_{\calA}(\dualVar) = \infty$ if $\dualVar \not\in \calA$). Then we apply \DRS and obtain the following result.
}{}
The rationale behind the use of \DRS to solve the feasibility SDP~\eqref{eq:sdpfeasibility} is that, although finding $\dualVar \in \calK \cap \calA$ is expensive (requires solving a large-scale SDP), finding $\dualVar \in \calK$ and $\dualVar \in \calA$ separately (\ie~projecting onto $\calK$ and $\calA$ separately) is computationally inexpensive~\cite{Henrion12handbook-conicProjection,Bauschke96SIAM-projectionFeasibility,Higham88LA-nearestSPD}.
The following 
result shows how to certify optimality using the \DRS iterations~\eqref{eq:DRSIterates}. 
\begin{theorem}[\DRS for Optimality Certification] \label{thm:DRSOptimalityCertification} 
Consider the \DRS iterations~\eqref{eq:DRSIterates}. Then the following properties hold true: 
(i) If the SDP~\eqref{eq:sdpfeasibility} is feasible, then the sequence $\{ \dualVar_\tau \}_{\tau \geq 0}$ in~\eqref{eq:DRSIterates} converges to a solution of~\eqref{eq:sdpfeasibility} when $0 < \gamma_\tau < 2$;
(ii) Let $\subopt = (\hatf - f^\star)/\hatf$ be the relative suboptimality between $\hatf$ and the global minimum $f^\star$ of the POP~\eqref{eq:pop}, then each \DRS iteration~\eqref{eq:DRSIterates} gives a valid suboptimality upper bound $\suboptbound_\tau$, \ie  $\subopt \leq \suboptbound_{\tau}$, and $\suboptbound_{\tau}$ can be efficiently computed from $\dualVar_\tau^\calA$.
\end{theorem}
A complete proof of Theorem~\ref{thm:DRSOptimalityCertification} 
is given in the \supp.
The intuition behind Theorem~\ref{thm:DRSOptimalityCertification}(i) is that, 
by using the two projections alternatively (thus, the name ``\emph{splitting}''), the \DRS iterates~\eqref{eq:DRSIterates} \edit{converge} to a solution in $\calK \cap \calA$ if the intersection is nonempty. Moreover, even if the intersection is empty (\eg~when $\hatf$ is not the global minimum), Theorem~\ref{thm:DRSOptimalityCertification}(ii) states that each \DRS iteration is still able to \emph{assess} the suboptimality of $\hatf$, 
which enables the \emph{dual certifiers} to detect wrong candidate solutions (\cf~Section~\ref{sec:experiments:certification}). 
\DRS converges faster than the vanilla \emph{alternating projections to convex sets} used in~\cite{Yang20arxiv-teaser} (\cf~\cite{Bauschke04JAT-averagedAlternatingReflections}). 
Moreover, we further boost convergence speed by initializing \DRS with an initial point $\vd_0$ computed by solving an
inexpensive SOS program with \emph{chordal sparsity}~\cite{Waki06jopt-SOSSparsity,lasserre10book-momentsOpt} (see the \supp for 
implementation details).

\optional{
In theory, one can start \DRS~\eqref{eq:DRSIterates} at any initial condition $\dualVar_0$. However, to speed up \DRS, we compute the initial point $\vd_0$ by solving a cheap SOS program with \emph{chordal sparsity}~\cite{Waki06jopt-SOSSparsity,lasserre10book-momentsOpt}.

\subsection{Chordal Sparse Initialization}
In theory, one can start \DRS~\eqref{eq:DRSIterates} at any initial condition $\dualVar_0$. However, to speed up \DRS, we compute the initial point $\vd_0$ by solving a cheap SOS program with \emph{chordal sparsity}~\cite{Waki06jopt-SOSSparsity,lasserre10book-momentsOpt}.
\begin{proposition}[Chordal Sparse Initialization] \label{prop:chordalSparseInitialization}
Define $\monoleq{\xextend}{\rbasisset_i} \doteq [1,\vxx\tran,\theta_i, \theta_i \vxx\tran]\tran \in \Real{2n+2}$, $\monoleq{\xextend}{1i} \doteq [1,\vxx\tran,\theta_i]\tran \in \Real{n+2}$, as the sparse monomial bases only in $\vxx$ and $\theta_i,i=1,\dots,N$, and define $\sosindeg{\xextend}{2\rbasisset_i},\sosindeg{\xextend}{2i}$ as the set of SOS polynomials parametrized by the sparse bases $\monoleq{\xextend}{\rbasisset_i},\monoleq{\xextend}{1i}$,~\ie~$q \in \sosindeg{\xextend}{2\rbasisset_i}$ if and only if $q= \monoleq{\xextend}{\rbasisset_i}\tran \MQ \monoleq{\xextend}{\rbasisset_i}$ for some $\MQ \in \psd^{2n+2}$ and $q \in \sosindeg{\xextend}{2i}$ if and only if $q = \monoleq{\xextend}{1i}\tran \MQ \monoleq{\xextend}{1i}$ for some $\MQ \in \psd^{n+2}$. Then $\dualVar_0$ is the solution of the SOS program:
\bea
\max & \zeta \label{eq:soschordalsparse}\\
\subject &  f(\xextend) - \zeta - \sum_{h_j \in \vh^{\vxx}} \lambda_j^{\vxx} h_j  - \sum_{h_j \in \vh^{\vtheta} }  \lambda_j^{\vtheta} h_j = s_0 + \sum_{k=1}^l s_k g_k, \\
& \lambda_j^\vxx, \lambda_j^\vtheta \text{ satisfy eq.~}\eqref{eq:lambdaSparsity}, \\
& s_0 = \sum_{i=1}^N s_{0i}, s_{0i} \in \sosindeg{\xextend}{2\rbasisset_i}, s_k = \sum_{i=1}^N s_{ki}, s_{ki} \in \sosindeg{\xextend}{2i},k=1\dots,l. \label{eq:soschordalsparseCon}
\eea
\end{proposition}
The SOS program~\eqref{eq:soschordalsparse} is the same as the SOS program~\eqref{eq:sosfeasiblity}, except that we have replaced constraint~\eqref{eq:sosSparsity} with constraint~\eqref{eq:soschordalsparseCon}, and we maximize $\zeta$ instead of requiring $\zeta = \hatf$. The advantage of SOS program~\eqref{eq:soschordalsparse} is that it breaks the large PSD constraint of size $\dimbasis{}{\rbasisset}$ (\ie~$s_0 \in \sosindeg{\xextend}{2\rbasisset}$) into $N$ small PSD constraints of fixed size $2n+2$ (\ie~$s_{0i} \in \sosindeg{\xextend}{2\rbasisset_i}$), and can scale to much larger $N$. However, the trade-off is program~\eqref{eq:soschordalsparse} is more restrictive and in general its optimum $\zeta^\star$ cannot certify the global optimality of $\hatf$ (\ie~$\zeta^\star < p^\star_{\rbasisset} \leq f^\star \leq \hatf$). Therefore, we compute $\dualVar_0$ by solving this cheap SOS program~\eqref{eq:soschordalsparse} using IPM-based SDP solvers and refine $\dualVar_0$ by running \DRS for the more powerful SOS program~\eqref{eq:sosfeasiblity}.
}{}

%% file: experiments.tex
\shrink
\section{Experiments}
\label{sec:experiments}
\shrink

This section shows that 
(i) the sparse moment relaxation~\eqref{eq:sparseMoment} is tight and can be used to solve small problems (\eg~$N=20$);
(ii)  
our \emph{dual optimality certifiers} are effective and scale to larger problems (\eg~$N=100$); 
(iii) our algorithms allow solving realistic satellite pose estimation problems. 


{\bf Implementation.} We model the sparse moment relaxation~\eqref{eq:sparseMoment} 
using YALMIP~\cite{Lofberg04cacsd-yalmip} in Matlab and solve the resulting SDPs using MOSEK~\cite{mosek}. \DRS is implemented in Matlab using $\gamma_{\tau} = 2$.\footnote{The limiting case of $\gamma_{\tau} = 2$ for \DRS is commonly referred to as the \emph{Peaceman-Rachford Splitting} (\PRS)~\cite{Combettes11book-proximalSplitting}. Although theoretically \PRS could diverge, we found it worked well for all our applications.} 

{\bf Setup.} We test primal relaxation and dual certification on {random} problem instances of Examples~\ref{eg:singleRotationAveraging}-\ref{eg:meshRegistration}: single rotation averaging (\singlerotation), 
shape alignment (\shapealign),
point cloud registration (\pointcloud),
and mesh registration (\mesh).
At each Monte Carlo run, we randomly sample a ground truth model $\vxx$ and
generate inliers by perturbing the measurements with Gaussian noise with standard deviation $\sigma$.
We choose $\sigma = 3^\circ$ in \singlerotation, and 
$\sigma = 0.01$ in \shapealign, \pointcloud, and \mesh.
Outliers are generated as arbitrary rotations or vectors (independent on $\vxx$).
 The relative weight between point-to-plane distance and normal-to-normal distance in \mesh is set to $w_i = 1,i=1,\dots,N$. 
 The threshold in problem~\eqref{eq:generalTLS} is set to $\barc=1$ for all applications, and $\beta_i,i=1,\dots,N$, is set to 
 be proportional to the inlier noise. \edit{The interested reader can find more details about the setup in the \supp.}

\input{fig-relax_certify_converge_satellite}

{\bf Primal Relaxation.}
\label{sec:experiments:relaxation}
We first evaluate the performance of the sparse moment relaxation~\eqref{eq:sparseMoment} under increasing outlier rates, with 
 $N=20$ measurements. Fig.~\ref{fig:relax_certify_converge_satellite}(a) shows the boxplots of rotation estimation errors and relative duality gap for \singlerotation (top), \shapealign (middle), and \mesh (bottom) averaged over 30 Monte Carlo runs (results for \pointcloud are qualitatively \edit{similar to \mesh} and hence postponed to the \supp). The sparse moment relaxation is numerically tight (relative gap smaller than $10^{-3}$), 
 with a single instance exhibiting a large gap (mesh registration, $80\%$ outliers). 
 The figure also shows that the relaxation produces an accurate estimate in all tested instances. \final{In the \supp, we show our primal relaxation is tight even under \emph{adversarial outliers}.}


{\bf Dual Certification.}
\label{sec:experiments:certification}
We test our dual optimality certifiers under increasing outlier rates, with 
 $N=100$ measurements. 
In each Monte Carlo run, we first use \GNC~\cite{Yang20ral-GNC} as a fast heuristics to compute a candidate solution to the POP~\eqref{eq:pop}, and then run the proposed dual certifiers (Theorem~\ref{thm:DRSOptimalityCertification}) to compute a suboptimality gap. Fig.~\ref{fig:relax_certify_converge_satellite}(b) plots the number of runs when \GNC returns the \emph{correct} solutions (\ie~with rotation error less than $5^\circ$), and the number of runs when the solutions are \emph{certified} (\ie~have suboptimality below $1\%$). We can see that our dual certifiers can certify all correct solutions and reject all incorrect estimates (the blue and green bars always have same height, 
meaning that there are no false positives nor false negatives). Fig.~\ref{fig:relax_certify_converge_satellite}(c) plots the average convergence history of the suboptimality gap versus the number of \DRS iterations (in log-log scale). \DRS drives the suboptimality below $1\%$ within 1000 iterations (within 100 iterations for \singlerotation) if the solution is correct, while it reports a suboptimality larger than $10\%$ if the solution is incorrect. \final{In the \supp, we show our certification outperforms the statistical Kolmogorov–Smirnov test~\cite{Massey51JASA-KStest}.}

{\bf Which One is More Scalable?}
\label{sec:experiments:scalability}
Table~\ref{table:timing} compares the scalability of the sparse relaxation and the dual certification for increasing number of measurements. 
Solving the large-scale SDP quickly becomes intractable for moderate $N$, while certification using \DRS can scale to large number of measurements.

\input{table-timing}

\vspace{2mm}
{\bf Satellite Pose Estimation.}
\label{sec:experiments:satellite}
Satellite pose estimation using monocular vision is a crucial technology for many space operations~\cite{Sharma19arXiv-SPEED,Chen19ICCVW-satellitePoseEstimation}.
We use ``Shape Alignment (Example~\ref{eg:shapeAlignment})'' to perform 6D pose estimation from satellite images in the \SPEED dataset~\cite{Sharma19arXiv-SPEED} (see Fig.~\ref{fig:relax_certify_converge_satellite}(d)). Towards this goal, we first use the pre-trained network from~\cite{Chen19ICCVW-satellitePoseEstimation} to detect 11 pixel measurements corresponding to 3D keypoints of the Tango satellite model.
Because the network outputs fairly accurate detections (all inliers), we also replace $0\%,18\%,35\%,49\%,62\%$, and $73\%$ \emph{pairwise} inliers (see \supp) with random outliers to test more challenging instances. We show a correct and certified estimation with $62\%$ outliers in Fig.~\ref{fig:relax_certify_converge_satellite}(d) top panel, and an incorrect and non-certified estimation with $73\%$ outliers in Fig.~\ref{fig:relax_certify_converge_satellite}(d) middle panel. Fig.~\ref{fig:relax_certify_converge_satellite}(d) bottom panel plots the statistics of the rotation error over 20 satellite images \edit{(showing the relation between suboptimality and estimation errors)}. 
We refer the reader to the \supp for a more comprehensive description of the tests and the results.
\optional{Then we form \emph{pairwise} measurements: $\vz_{ij} = \tldvz_i - \tldvz_j$ and $\MB_{ij} = \tldMB_i - \tldMB_j$ for all $i \neq j$, which gives rise to $N=\nchoosek{11}{2} = 55$ measurements that satisfy the \shapealign model in example~\ref{eg:shapeAlignment} (the translation cancels due to the subtraction). We then use \GNC to estimate $s$ and $\MR$, followed by using the adaptive voting algorithm in~\cite{Yang19rss-teaser} to estimate a 2D translation ($s$ is approximately the inverse of the depth in weak perspective camera model).}{}

%% file: fig-relax_certify_converge_satellite.tex

\newcommand{\mpwfour}{3.5cm}
\newcommand{\myhspace}{\hspace{-4mm}}
\newcommand{\mygap}{\hspace{-3.5mm}}
\newcommand{\mygapcon}{\hspace{0.5mm}}
\newcommand{\mygapsate}{\hspace{1mm}}
\begin{figure}[t]
	\begin{center}
	\begin{minipage}{\textwidth}
	\begin{tabular}{ccccc}%
		\myhspace \mygap \mygap \rotatebox{90}{ {\smaller \singlerotation} } \hspace{-3mm}
	    & \myhspace \mygap
			\begin{minipage}{\mpwfour}%
			\centering%
			\includegraphics[width=\columnwidth]{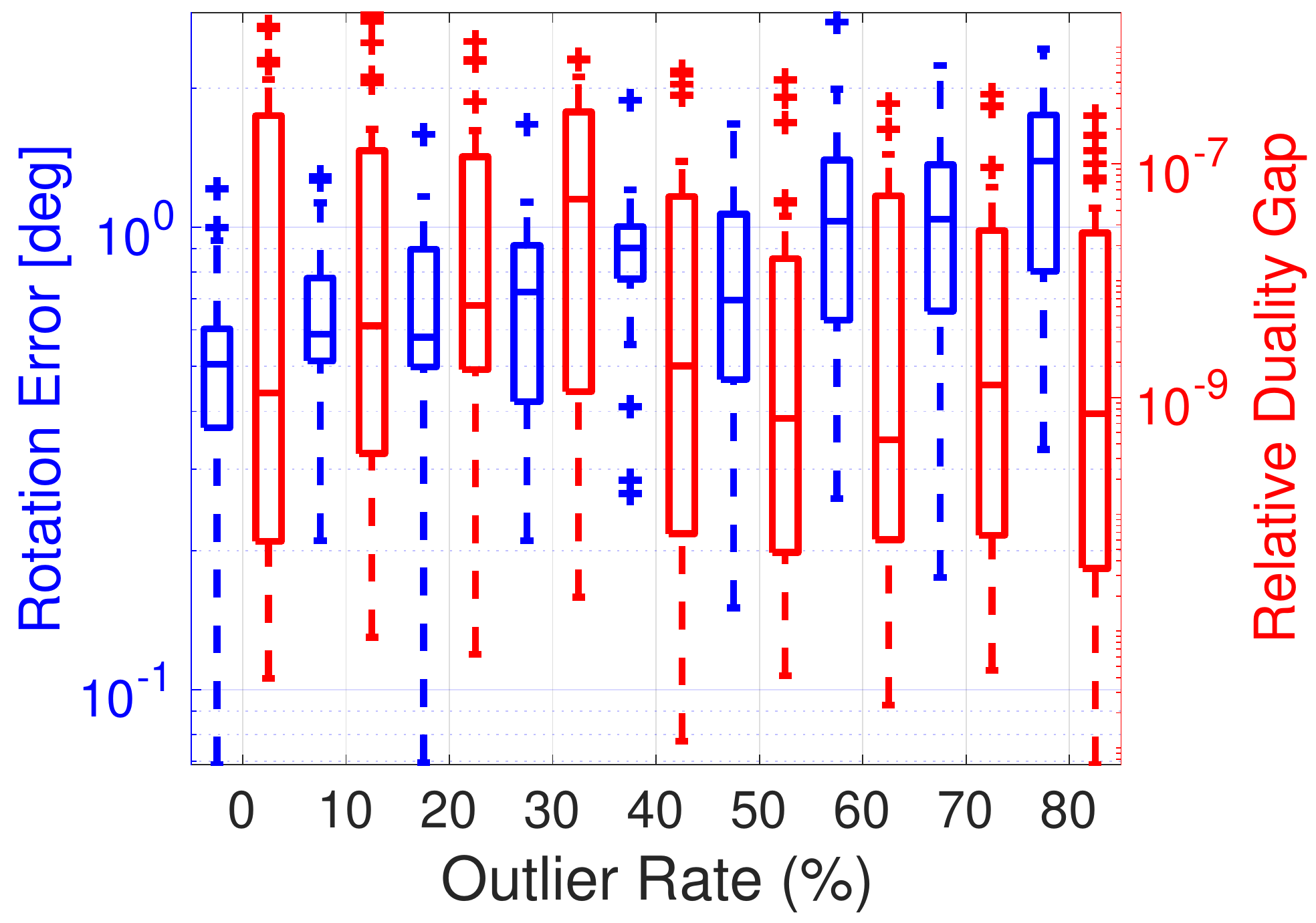}
			\end{minipage}
		&  \myhspace
			\begin{minipage}{\mpwfour}%
			\centering%
			\includegraphics[width=\columnwidth]{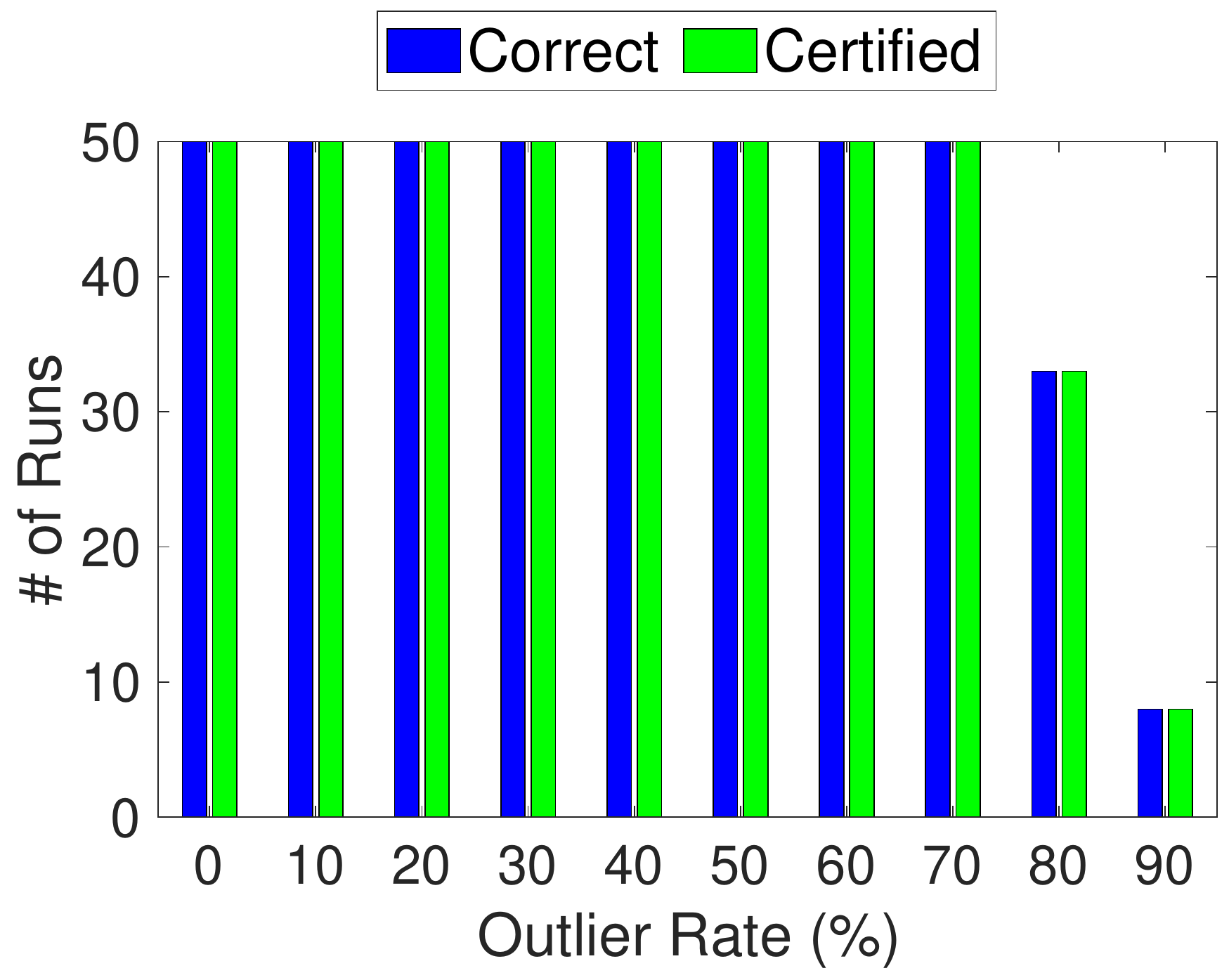}
			\end{minipage}
		&  \myhspace \mygapcon
			\begin{minipage}{\mpwfour}%
			\centering%
			\includegraphics[width=\columnwidth]{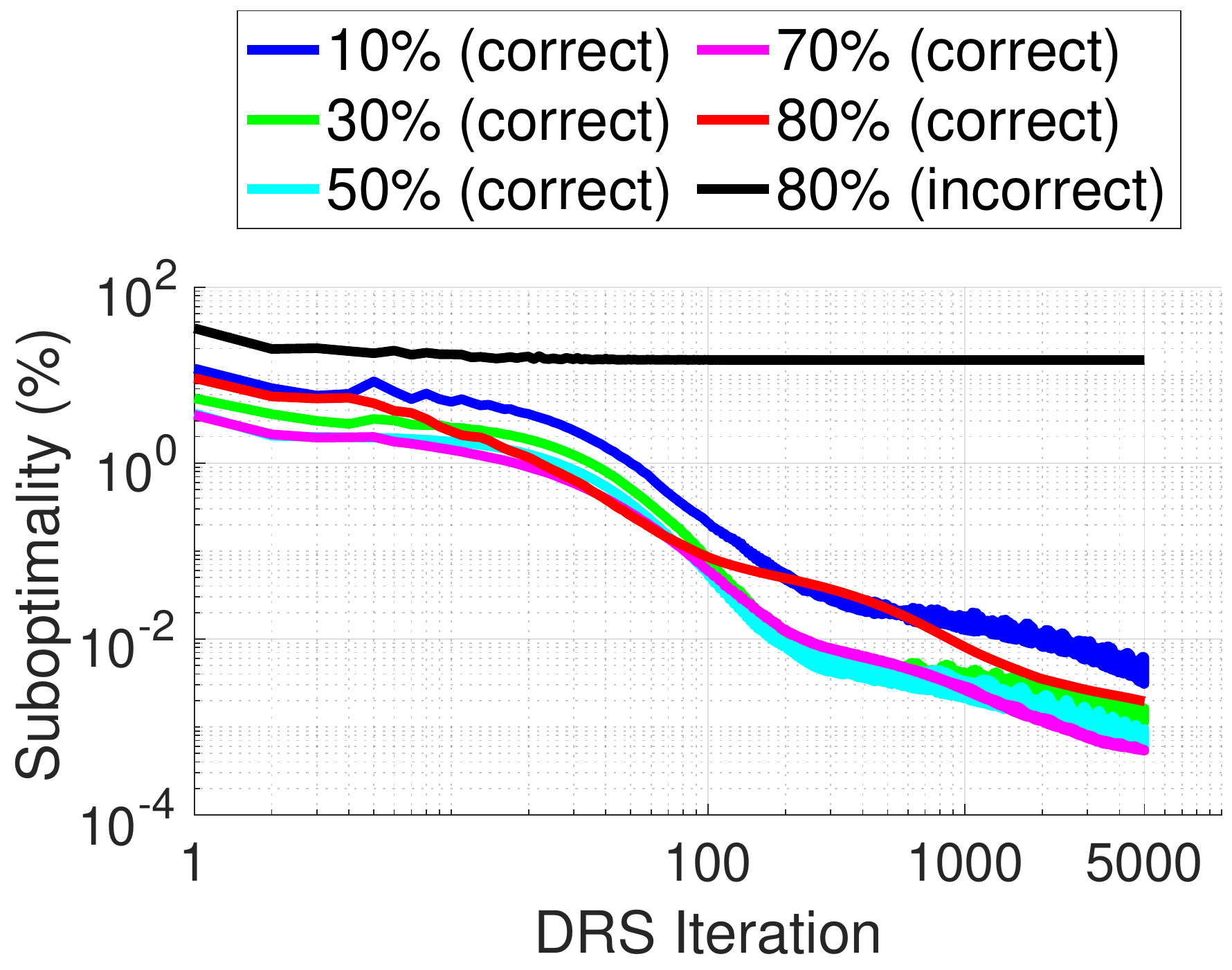}
			\end{minipage}
		&  \myhspace \mygapsate
			\begin{minipage}{\mpwfour}%
			\centering%
			\includegraphics[width=\columnwidth]{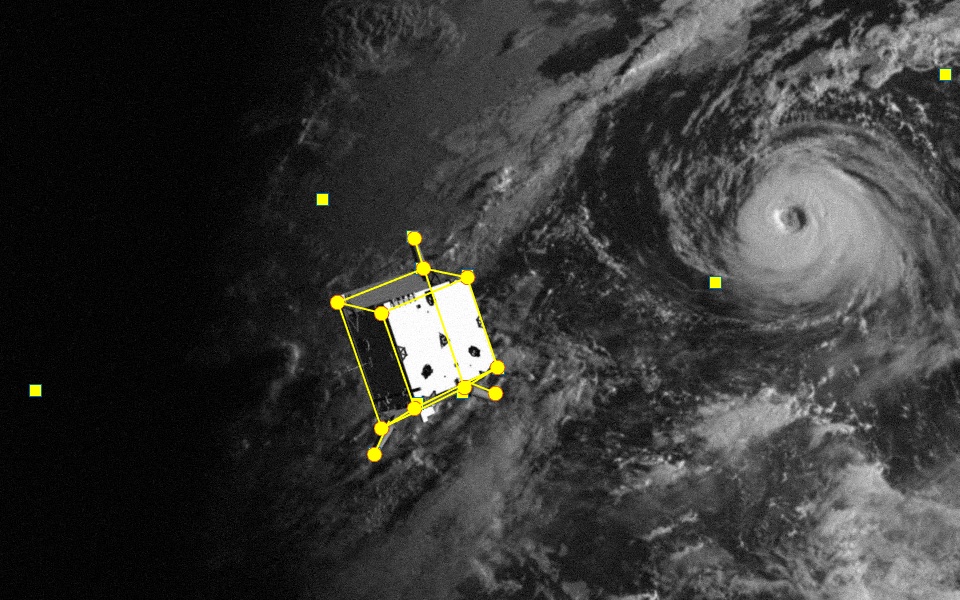} \\
			\vspace{-1mm}
			{\tiny $62\%$ outliers, correct, certified}
			\end{minipage}
		\\
		\myhspace \mygap \mygap \rotatebox{90}{ {\smaller \shapealign} } \hspace{-3mm}
		& \myhspace \mygap
			\begin{minipage}{\mpwfour}%
			\centering%
			\includegraphics[width=\columnwidth]{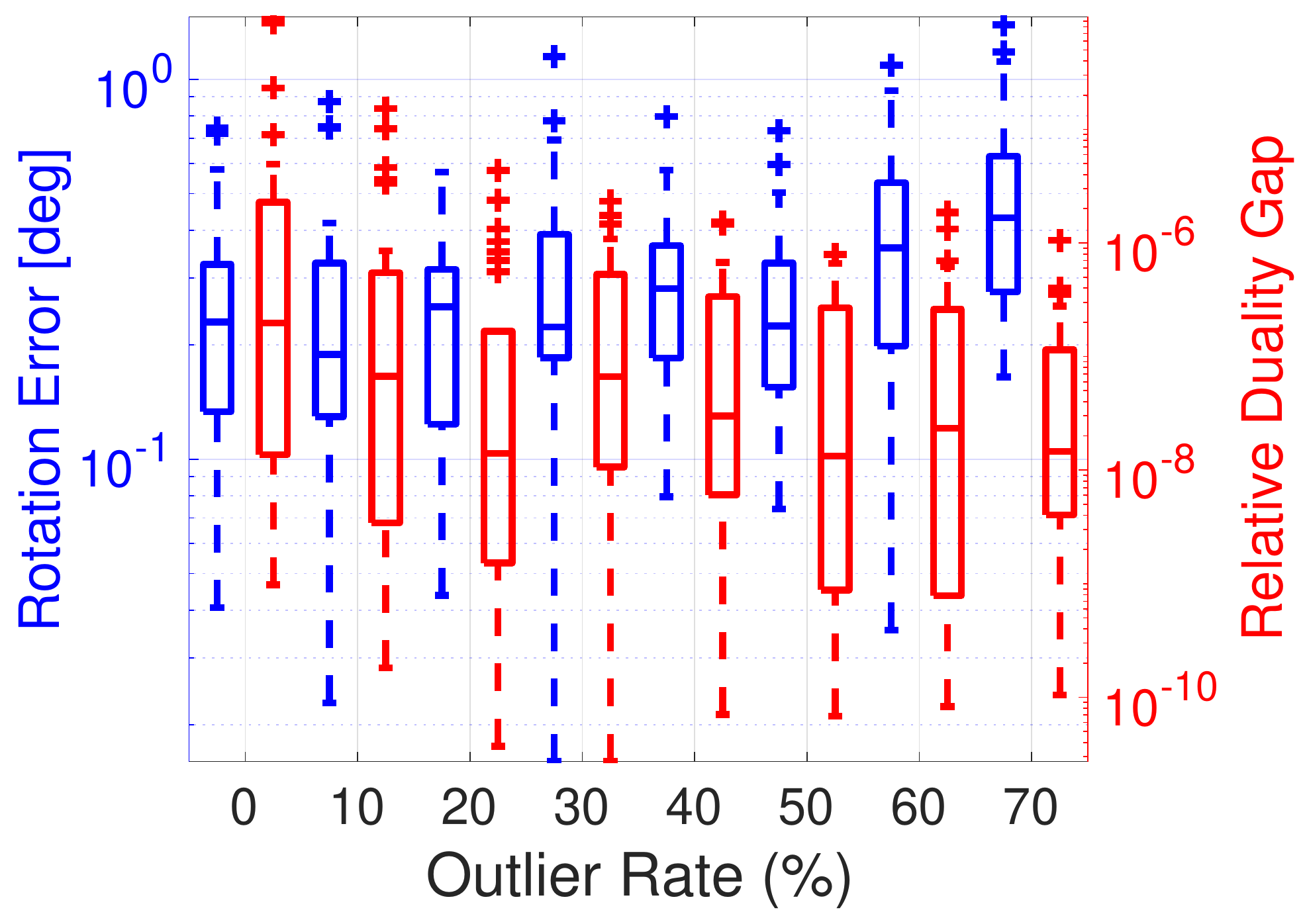}
			\end{minipage}
		&  \myhspace
			\begin{minipage}{\mpwfour}%
			\centering%
			\includegraphics[width=\columnwidth]{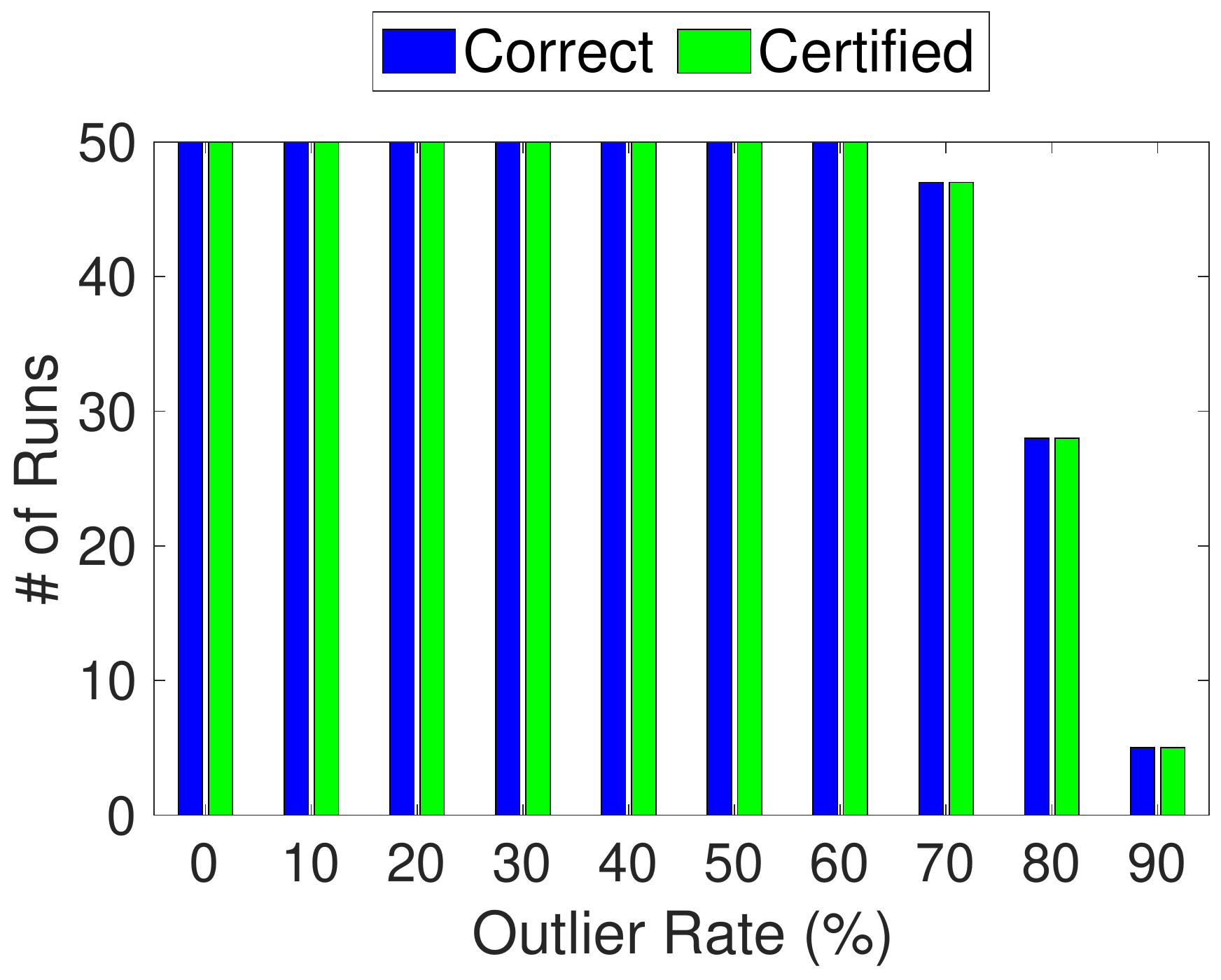}
			\end{minipage}
		&  \myhspace \mygapcon
			\begin{minipage}{\mpwfour}%
			\centering%
			\includegraphics[width=\columnwidth]{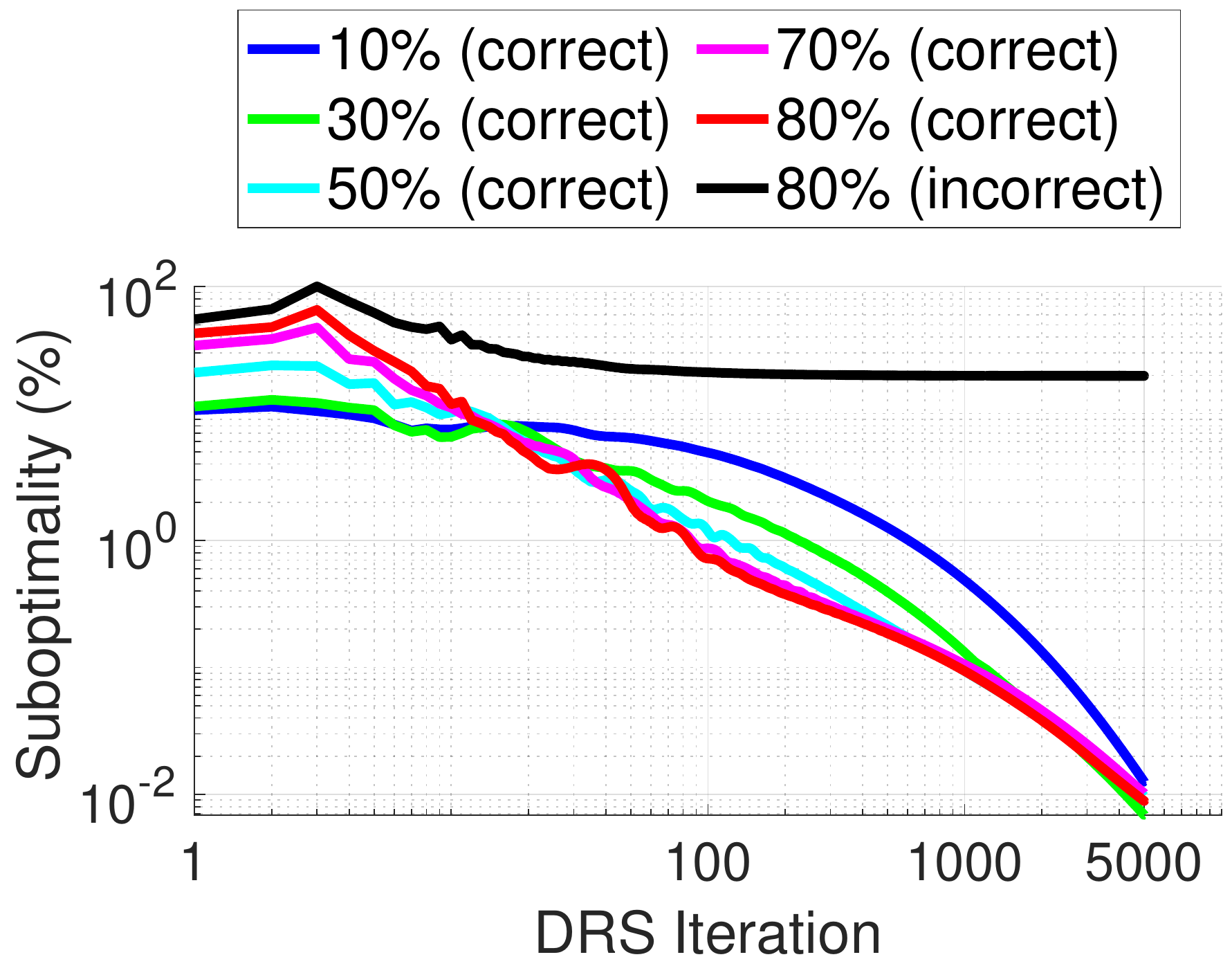}
			\end{minipage}
		&  \myhspace \mygapsate
			\begin{minipage}{\mpwfour}%
			\centering%
			\includegraphics[width=\columnwidth]{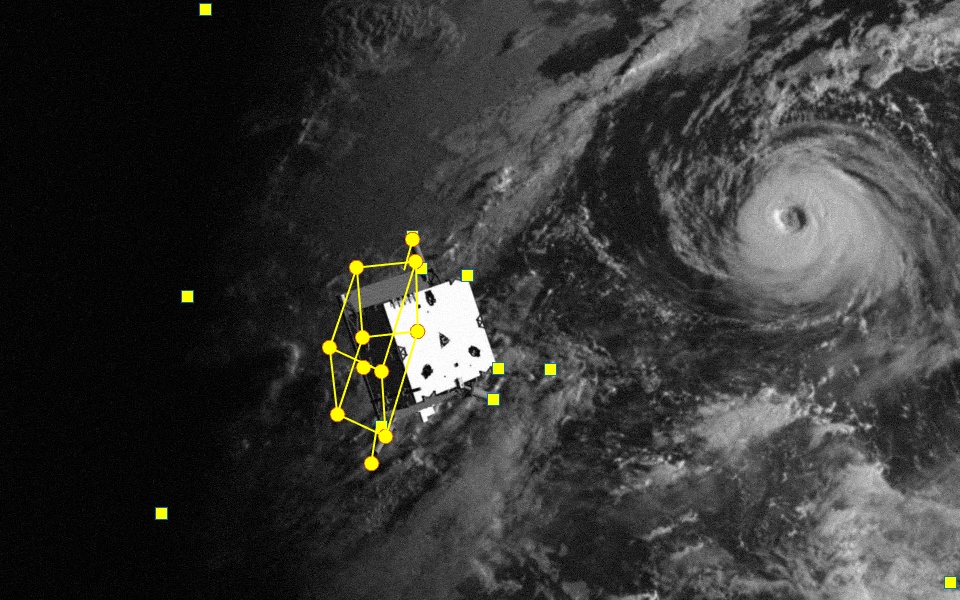} \\
			\vspace{-1mm}
			{\tiny $73\%$ outliers, incorrect, non-certified}
			\end{minipage}
		\\
		\myhspace \mygap \mygap \rotatebox{90}{ {\smaller \mesh} } \hspace{-3mm}
		& \myhspace \mygap
			\begin{minipage}{\mpwfour}%
			\centering%
			\includegraphics[width=\columnwidth]{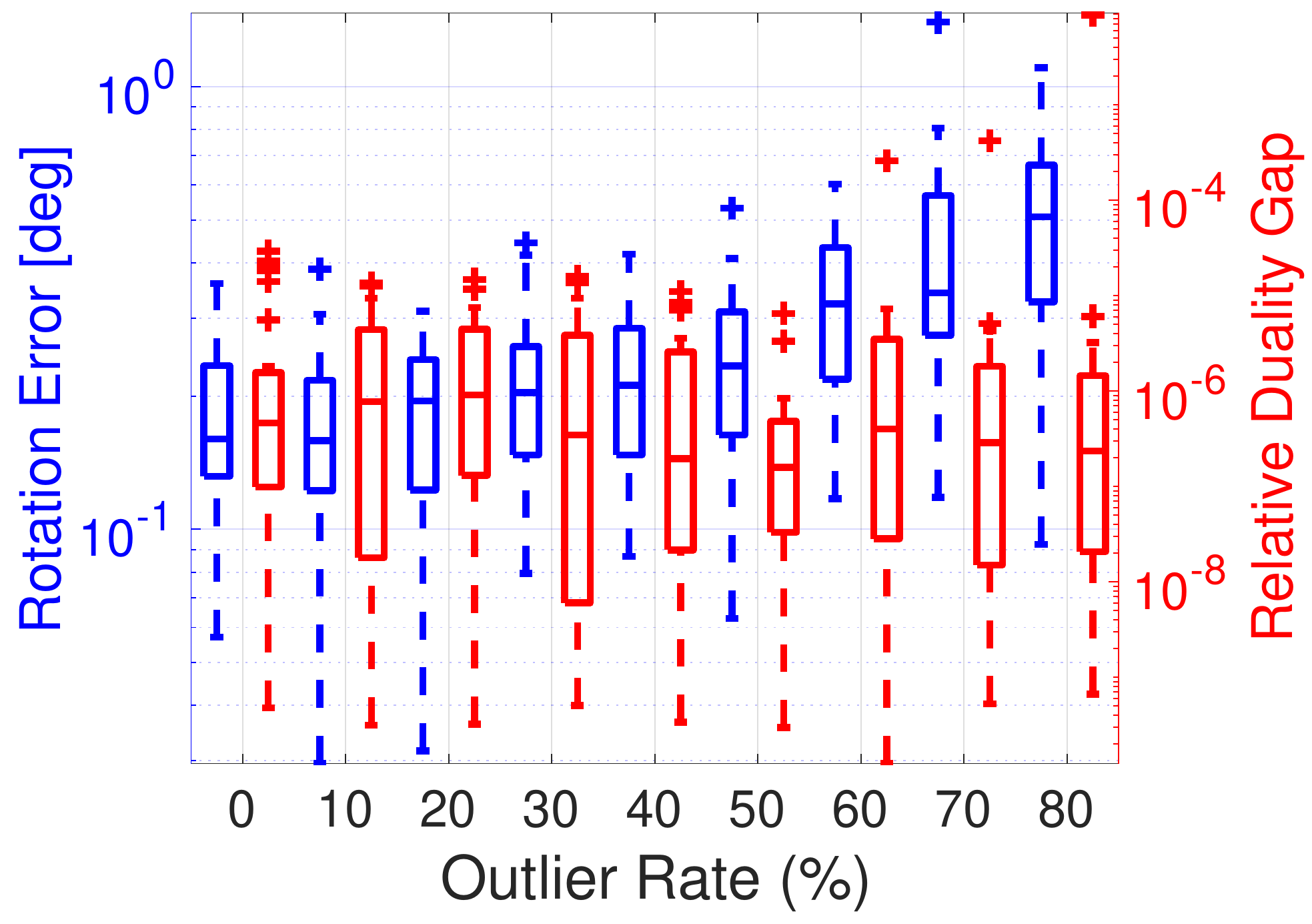}
			\end{minipage}
		&  \myhspace
			\begin{minipage}{\mpwfour}%
			\centering%
			\includegraphics[width=\columnwidth]{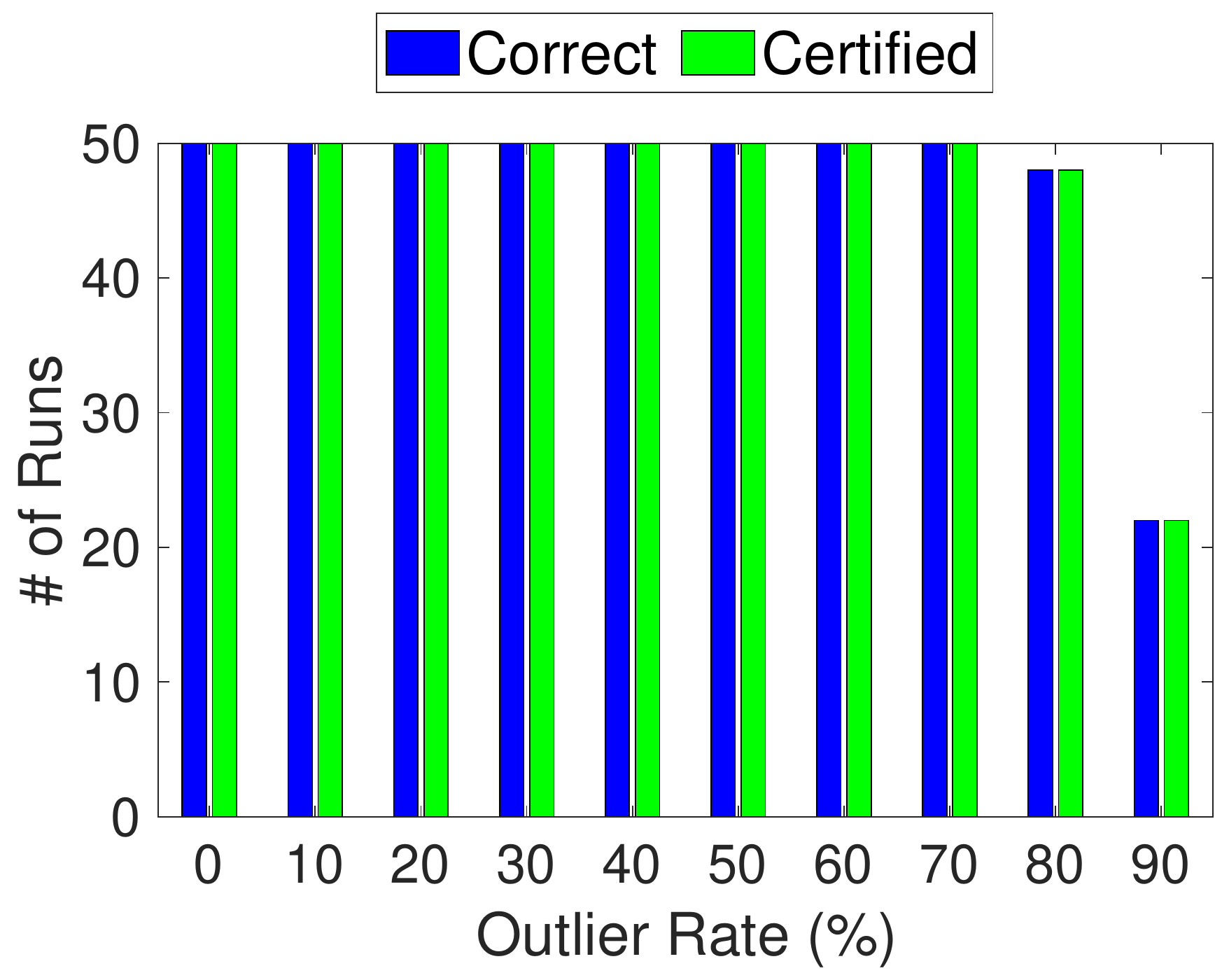}
			\end{minipage}
		&  \myhspace \mygapcon
			\begin{minipage}{\mpwfour}%
			\centering%
			\includegraphics[width=\columnwidth]{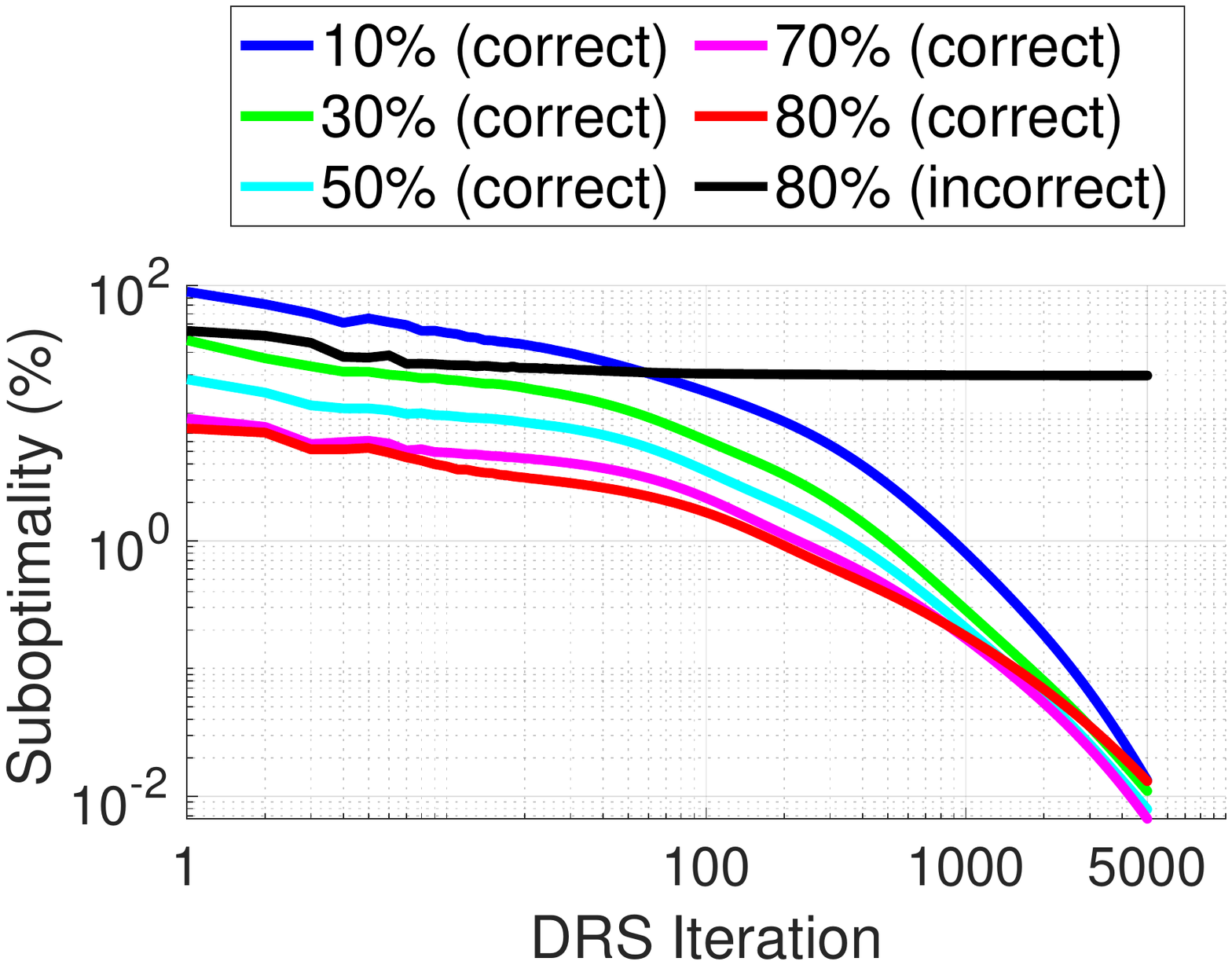}
			\end{minipage}
		&  \myhspace \mygapsate
			\begin{minipage}{\mpwfour}%
			\centering%
			\includegraphics[width=\columnwidth]{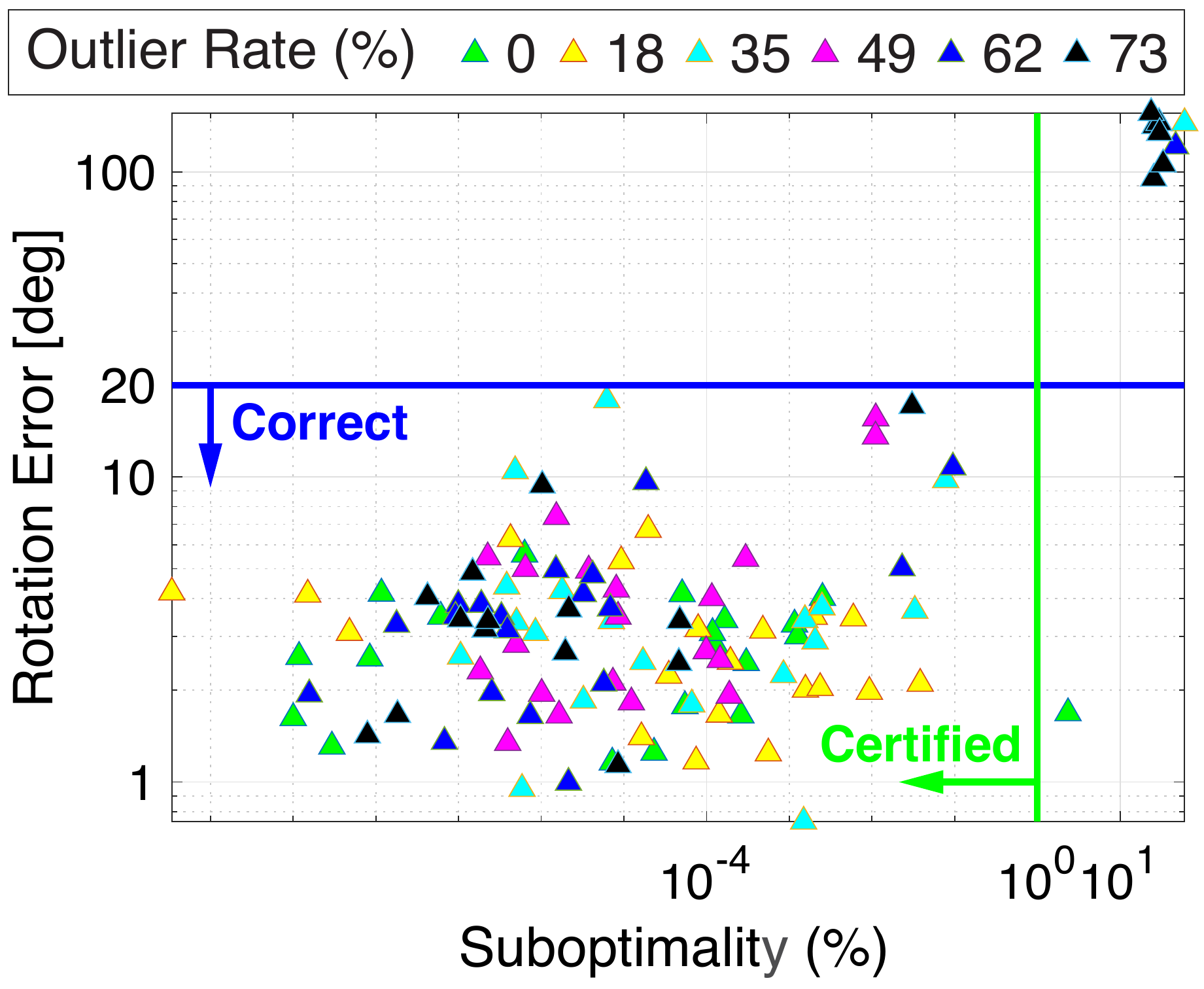} \\
			\end{minipage}
		\\
		&
		   \myhspace \mygap {\scriptsize (a) Sparse moment relaxation} 
		&  \myhspace \hspace{0mm} {\scriptsize (b) Dual optimality certification}
		&  \myhspace \hspace{0mm} {\scriptsize (c) Convergence of suboptimality}
		&  \myhspace \hspace{0mm} {\scriptsize (d) Satellite pose estimation}
	\end{tabular}
	\end{minipage} 
	\caption{\footnotesize Performance of certifiable algorithms. (a) Rotation estimation error (\blue{left axis}) and relative duality gap (\red{right axis}) of the sparse moment relaxation~\eqref{eq:sparseMoment}. (b) Number of runs when the solution of \GNC is \blue{correct} (\ie~rotation error less than $5^\circ$) and number of runs when the solution of \GNC is \green{certified} (\ie~suboptimality below $1\%$). (c) 
	Suboptimality gap versus~\DRS iterations, averaged over correct and incorrect runs. \edit{Top row: single rotation averaging (\singlerotation), middle row: shape alignment (\shapealign), bottom row: mesh registration (\mesh).} (d) Qualitative and quantitative results for satellite pose estimation on the \SPEED dataset~\cite{Sharma19arXiv-SPEED}.
	\label{fig:relax_certify_converge_satellite}} 
	\vspace{-4mm} 
	\end{center}
\end{figure}

%% file: table-timing.tex
\vspace{-3mm}
\begin{table}[h]
\hspace{-3.5mm}
\begin{minipage}{\textwidth}
\begin{tabular}{c|ccc|ccc|ccc|ccc}
\hline
 \multirow{2}{*}{ \scriptsize $N$ } & \multicolumn{3}{c|}{ {\scriptsize \singlerotation} } & \multicolumn{3}{c|}{\scriptsize \shapealign} &\multicolumn{3}{c|}{ {\scriptsize \pointcloud} } & \multicolumn{3}{c}{ \scriptsize \mesh} \\
 & {\scriptsize $\dimbasis{}{\rbasisset}$}  & {\scriptsize $\relaxtime$} & \scriptsize $\certifytime$ & \scriptsize $\dimbasis{}{\rbasisset}$ & \scriptsize $\relaxtime$ & \scriptsize $\certifytime$ & \scriptsize $\dimbasis{}{\rbasisset}$ & \scriptsize $\relaxtime$ & \scriptsize $\certifytime$ & \scriptsize $\dimbasis{}{\rbasisset}$ & \scriptsize $\relaxtime$ & \scriptsize $\certifytime$ \\
 \hline 
 \scriptsize $20$ & \scriptsize $255$  & \scriptsize $151.65$ & \scriptsize  $0.73$ & \scriptsize $168$ & \scriptsize $35.52$ & \scriptsize $2.00$ & \scriptsize $351$ & \scriptsize $763.59$ & \scriptsize $16.34$ & \scriptsize $351$ & \scriptsize $750.67$ & \scriptsize $8.43$ \\
 \hline 
 \scriptsize $50$ & \scriptsize $555$ &  \scriptsize $38866$ & \scriptsize $2.67$ & \scriptsize $378$ & \scriptsize $3287$ & \scriptsize $7.88$ & \scriptsize $741$ & \scriptsize $\mathrm{>}10^5$ & \scriptsize $84.86$ & \scriptsize $741$ & \scriptsize $\mathrm{>}10^5$ & \scriptsize $60.76$ \\
 \hline 
 \scriptsize $100$ & \scriptsize $1055$ & \scriptsize $**$ & \scriptsize $8.35$ & \scriptsize $728$ & \scriptsize $\mathrm{>}10^5$ & \scriptsize $37.44$ & \scriptsize $1391$ & \scriptsize $**$ & \scriptsize $357.48$ & \scriptsize $1391$ & \scriptsize $**$ &  \scriptsize $165.87$ \\
 \hline
\end{tabular}
\end{minipage}
\vspace{1mm}
\caption{\footnotesize Relaxation and certification time (in seconds) for increasing $N$. $\relaxtime$ is the time for solving the sparse moment relaxation~\eqref{eq:sparseMoment}. $\certifytime$ includes the time for computing the chordal initial guess and the time for \DRS to drive the suboptimality below $1\%$. ``$**$'' denotes instances where MOSEK ran out of memory.
\label{table:timing}
}
\vspace{-6mm}
\end{table}

%% file: conclusions.tex
\shrink
\section{Conclusions}
\label{sec:conclusions}
\shrink
We have proposed a general framework for designing certifiable algorithms for a broad class of robust geometric perception problems. From the primal perspective, we apply Lasserre's hierarchy of moment relaxations, together with basis reduction, to construct tight semidefinite relaxations to nonconvex robust estimation problems.
From the dual perspective, we use SOS relaxation to convert the certification of a given candidate solution to a convex feasibility SDP, and then we leverage Douglas-Rachford Splitting to solve the feasibility SDP and compute a suboptimality for the candidate solution. 
Our primal relaxation is tight, and our dual certification is correct and scalable. 

%% file: impact.tex
\shrink
\section*{Broader Impact}
\label{sec:impact}
\shrink

\Perception plays an essential role in robotics applications ranging from autonomous driving, robotic manipulation, autonomous flight, to robotic search and rescue. Occasional perception failures are almost inevitable while operating in the wild (\eg~due to sensor 
malfunction, or incorrect data association resulting from neural networks or hand-crafted feature matching techniques). These failures, if not detected and handled properly, have detrimental effects, especially in safety-critical and high-integrity applications (\eg~they may put passengers at risk in autonomous driving or damage a satellite in space applications). Existing perception algorithms (\eg~fast heuristics) can fail without notice.
On the contrary, the certifiable algorithms presented in this work perform geometric perception with optimality guarantees and act as safeguards to detect perception failures. 
The development of certifiable algorithms has the potential to enhance the robustness of perception systems, inform system certification and monitoring, and increase the 
trustworthiness of autonomous systems. 

On the negative side, the use of certifiable algorithms as an enabler for robots and autonomous systems inherits the shortcomings connected to the misuse of these technologies. The use of autonomous systems in military applications as well as the impact of robotics and automation on the (human) workforce must be carefully pondered to ensure a positive societal impact.

%% file: acknowledgements.tex
\shrink
\section*{Acknowledgments}
\shrink
The authors would like to thank Jie Wang, Victor Magron, and Jean B. Lasserre for the discussion about Lasserre's hierarchy of moment and SOS relaxations; Alp Yurtsever, Suvrit Sra, and Yang Zheng for the discussion about using first-order methods to solve large-scale SDPs; Bo Chen and Tat-Jun Chin for kindly sharing the data for satellite pose estimation; and the anonymous reviews.

\final{This work was partially funded by ARL DCIST CRA W911NF-17-2-0181, ONR RAIDER N00014-18-1-2828, and the Lincoln Laboratory “Resilient Perception in Degraded Environments” program.}

%% file: supp-notations.tex

{\bf Notation.} Besides the notation already defined in Section~\ref{sec:introduction} of the main text, we define the following extra notation. A polynomial $q \in \polyring$ is a sums-of-squares (SOS) polynomial if and only if $q$ can be written as $q = \monoleq{\vxx}{\calF}\tran \MQ \monoleq{\vxx}{\calF}$ for some monomial basis $\monoleq{\vxx}{\calF}$ and PSD matrix $\MQ \succeq 0$, in which case $q \geq 0, \forall \vxx \in \Real{n}$. We use $\sosindeg{\vxx}{2\calF}$ to denote the set of SOS polynomials parametrized by the monomial basis $\monoleq{\vxx}{\calF}$. In particular, when $\monoleq{\vxx}{\calF} = \monoleq{\vxx}{d}$ is the full standard monomial basis of degree up to $d$, we use $\sosindeg{\vxx}{2d}$ to denote the set of SOS polynomials with degree up to $2d$. Moreover, $\sosin{\vxx} \subset \polyring$ is the set of all SOS polynomials (with arbitrary degrees). For a constraint set $\calX$ defined by polynomial equality and inequality constraints $\calX \doteq \{ \vxx: h_j(\vxx) = 0,j=1,\dots,\nrEq;g_k(\vxx) \geq 0,k=1,\dots,\nrIneq\}$, the set $\calX$ is said to be Archimedean if there exist $M > 0$, $\lambda_j \in \polyring,j=1,\dots,\nrEq$, and $s_k \in \sosin{\vxx},k=1,\dots,\nrIneq$, such that $M - \twonorm{\vxx}^2 = \sum_{j=1}^\nrEq \lambda_j h_j + \sum_{k=1}^\nrIneq s_k g_k$, which immediately implies that $\twonorm{\vxx}^2 \leq M$ and the set $\calX$ is compact~\cite[Definition 3.137, p.~115]{Blekherman12Book-sdpandConvexAlgebraicGeometry}.

%% file: supp-proof-pop.tex
\section{Proof of Proposition~\ref{prop:pop} (Geometric Perception as POP)}
\label{sec:supp-proof-pop}
\begin{proof}
To tackle the non-smoothness of the inner minimization ``$\min\{\cdot,\cdot\}$'' in problem~\eqref{eq:generalTLS}, we first reformulate problem~\eqref{eq:generalTLS} as:
\bea
f^\star = \min_{\substack{\vxx \in \calX \\ \theta_i \in \{ \pm 1 \},i=1,\dots,N} } \sumallpoints \frac{1+\theta_i}{2\beta_i^2} r^2(\vxx,\measured_i) + \frac{1-\theta_i}{2}\barcsq, \label{eq:TLSBinary}
\eea
where we have used the fact that ``$\min\{a,b\}$'' is equivalent to an optimization over a binary variable: $\min\{a,b\} = \min_{\theta \in \{ \pm 1\}} \frac{1+\theta}{2} a + \frac{1-\theta}{2} b$ (where $\theta = + 1$ when $a < b$ and $\theta = -1$ when $a>b$). Intuitively, if the $i$-th measurement $\measured_i$ is an inlier (\ie~$r^2 \leq \barcsq \beta_i^2$), then $\theta_i = +1$ and the corresponding term in~\eqref{eq:TLSBinary} reduces to least squares; if $\measured_i$ is an outlier (\ie~$r^2 > \barcsq \beta_i^2$), then $\theta_i = -1$ and the corresponding term in~\eqref{eq:TLSBinary} becomes a constant $\barcsq$, whence the outlier is irrelevant to the optimization. Since we have introduced $N$ binary variables to the optimization~\eqref{eq:TLSBinary}, we denote $\xextend = \bracket{\vxx\tran, \vtheta\tran}\tran \in \Real{\dimxextend}$ as the new set of variables, where $\vtheta \doteq \bracket{\theta_1,\dots,\theta_N}\tran \in \{ \pm 1 \}^N$ is the vector of binary variables and $\dimxextend \doteq n + N$ is the number of variables. Then we make two immediate observations: (1) denote $f_i(\xextend) = \frac{1+\theta_i}{2\beta_i^2} r^2(\vxx,\measured_i) + \frac{1-\theta_i}{2}\barcsq$, then $f_i(\xextend) \in \polyringin{\vxx,\theta_i}$ is only a polynomial of $\vxx$ and $\theta_i$ and the objective function of~\eqref{eq:TLSBinary} can be written as the finite sum of $f_i$'s: $f(\xextend) = \sumallpoints f_i(\xextend)$ (\ie~claim~\ref{prop:pop-objective} in Proposition~\ref{prop:pop}). (2) The binary constraints $\theta_i \in \{\pm 1\},i=1,\dots,N$ are equivalent to quadratic polynomial equality constraints $h^{\theta_i} \doteq 1 - \theta_i^2 = 0,i=1,\dots,N$, and obviously each $h^{\theta_i} \in \polyringin{\theta_i}$ is only a polynomial in $\theta_i$. For simplicity, we denote $\vh^\theta = \{ h^{\theta_i} \}_{i=1}^N$ (\ie~claim~\ref{prop:pop-constraint} in Proposition~\ref{prop:pop}).

Next we will show that -- for Examples~\ref{eg:singleRotationAveraging}-\ref{eg:meshRegistration}-- (1) $r^2(\vxx,\measured_i)$ is a polynomial in $\vxx$ with $\deg{r^2(\vxx,\measured_i)} \leq 2$ (and hence, $\deg{f_i(\xextend)} \leq 3$), (2) the constraint set $\vxx \in \calX$ can be written as quadratic polynomial inequality and equality constraints, and (3) the feasible set is Archimedean (\ie~claim~\ref{prop:pop-archimedean} in Proposition~\ref{prop:pop}).


{\bf Example~\ref{eg:singleRotationAveraging} (Single Rotation Averaging)}. We develop the residual function:
\bea
\hspace{-6mm} r^2(\vxx,\measured_i) = \twonorm{\MR - \MR_i}_F^2 = \trace{(\MR - \MR_i)\tran(\MR - \MR_i)} = \trace{2\eye_3} - 2\trace{\MR_i\tran \MR\tran} = 6 - 2 \measured_i\tran \vr, \label{eq:residualSRA}
\eea
where we have denoted $\vr \doteq \vectorize{\MR} \in \Real{9}$ as the vectorization of the unknown rotation matrix $\MR$, and $\measured_i \doteq \vectorize{\MR_i} \in \Real{9}$ as the vectorization of the measurements $\MR_i$. From eq.~\eqref{eq:residualSRA} it is clear that $\deg{r^2(\vxx,\measured_i)} = 1$. The constraint set for single rotation averaging is $\MR \in \SOthree$, which is known to be equivalent to a set of (redundant) quadratic polynomial equality constraints~\cite{Yang20cvpr-shapeStar}.
\begin{lemma}[\bf Quadratic Constraints for $\SOthree$~\cite{Tron15rssws3D-dualityPGO3D,Yang20cvpr-shapeStar}] 
\label{lemma:quadraticConSOthree}
For any matrix $\MR \in \Real{3 \times 3}$, $\MR \in \SOthree$ is equivalent to the following set of 15 quadratic polynomial equality constraints $\vh^r = \{ h^r_i \}_{i=1}^{15}$:
\bea
\begin{cases}
\text{\small Orthonormality: } \scriptstyle h^r_1 = 1-\| \vr_1 \|^2,\ h^r_2 = 1-\| \vr_2 \|^2,\ h^r_3 = 1-\| \vr_3 \|^2,\ h^r_4 = \vr_1\tran \vr_2,\ h^r_5 = \vr_2\tran \vr_3,\ h^r_6 = \vr_3\tran \vr_1 \\
\text{\small Right-handedness: } \scriptstyle h^r_{7,8,9} = \vr_1 \times \vr_2 - \vr_3,\ h^r_{10,11,12} = \vr_2 \times \vr_3 - \vr_1,\ h^r_{13,14,15} = \vr_3 \times \vr_1 - \vr_2
\end{cases}
\eea
where $\vr_i \in \Real{3},i=1,2,3$ denotes the $i$-th column\footnote{The same set of quadratic constraints hold when $\vr_i\tran \in \Real{3},i=1,\dots,3$ denotes the $i$-th row of $\MR$.} of $\MR$ and ``$\times$'' represents vector cross product.
\end{lemma}
Therefore, we have $\vh = \vh^x \cup \vh^\theta$ with $\vh^x \doteq \vh^r$, and $\vg = \varnothing$ for single rotation averaging. To show the Archimedeanness of the feasible set $\calxextend \doteq \{ \xextend: h(\xextend)=0,\forall h \in \vh, 1\geq g(\xextend) \geq 0, \forall g \in \vg \}$, we note that:
\bea
3 + N - \twonorm{\xextend}^2 = \sum_{i=1}^3 1 \cdot h_i^r + \sum_{i=1}^N 1 \cdot h^{\theta_i} = 0,
\eea 
which implies that $\twonorm{\xextend}^2 \leq N+3$ and the feasible set $\calxextend$ is equipped with a polynomial certificate for compactness. 


{\bf Example~\ref{eg:shapeAlignment} (Shape Alignment)}. Directly developing the residual function $r^2(\vxx,\measured_i) = \twonorm{\vb_i - s \Pi \MR \MB_i}^2$ leads to a quartic polynomial (degree 4) in $s$ and $\MR$, which is not suitable for moment relaxation because it would increase the minimum relaxation order $\rorder$~\cite{lasserre10book-momentsOpt}. Therefore, we perform a change of variables and let $\SAR = s\Pi\MR$:
\bea
\SAR = s \bmat{ccc}
1 & 0 & 0 \\
0 & 1 & 0
\emat
\bmat{c}
\vr_1\tran \\
\vr_2\tran \\
\vr_3\tran 
\emat
= 
\bmat{c}
s \vr_1\tran \\
s \vr_2\tran
\emat
\doteq
\bmat{c}
\SAr_1\tran \\
\SAr_2\tran
\emat,
\eea
where $\vr_i\tran \in \Real{3}$ denotes the $i$-th row of the rotation matrix $\MR$ and we have denoted $\SAr_i = s \vr_i,i=1,2$ as the product of $s$ and $\vr_i$. Now using Lemma~\ref{lemma:quadraticConSOthree}, we can see that $s \in [0,\sub]$ and $\MR \in \SOthree$ is equivalent to the following constraints on $\SAr \doteq \vectorize{\SAR\tran} = \bracket{\SAr_1\tran,\SAr_2\tran}\tran$:
\bea
\vh^{\SAlr} = \left \{ h^r_1 = \twonorm{\SAr_1}^2 - \twonorm{\SAr_2}^2, h^r_2 = \SAr_1\tran \SAr_2 \right\},\quad \vg^{\SAlr} = \left\{ 1 - \frac{\twonorm{\SAr_1}^2 + \twonorm{\SAr_2}^2}{2 \sub^2} \right\}. \label{eq:SAQuadraticCon}
\eea
Therefore, we have $\vh = \vh^x \cup \vh^\theta$ with $\vh^x \doteq \vh^\SAlr$, and $\vg = \vg^\SAlr$ for shape alignment.\footnote{Note that due to the division by $2 \sub^2$ in eq.~\eqref{eq:SAQuadraticCon}, $0 \leq g^\SAlr \leq 1$ is satisfied.} To prove the feasible set is Archimedean, we write the following polynomial certificate for compactness:
\bea
2\sub^2 + N - \twonorm{\xextend}^2 = 2\sub^2 \cdot g^\SAlr + \sumallpoints 1 \cdot h^{\theta_i} \geq 0.
\eea


{\bf Example~\ref{eg:pointCloudRegistration} (Point Cloud Registration)}. We develop the residual function:
\bea
r^2(\vxx,\measured_i) = \twonorm{\vb_i - \MR\va_i - \vt}^2 = \twonorm{\vt}^2 - 2\vb_i\tran\vt - 2\vb_i\tran\MR\va_i + 2\vt\tran \MR\va_i + \twonorm{\va_i}^2 + \twonorm{\vb_i}^2 \nonumber \\
= \twonorm{\vt}^2 - 2\vb_i\tran\vt - 2\parentheses{\va_i\tran \kron \vb_i\tran} \vr + 2 \parentheses{\va_i\tran \kron \vt\tran} \vr + \twonorm{\va_i}^2 + \twonorm{\vb_i}^2, \label{eq:residualPointCloud}
\eea
where $\vr \doteq \vectorize{\MR} \in \Real{9}$ is the vectorization of $\MR$ and ``$\kron$'' denotes the Kronecker product. Clearly, $\deg{r^2(\vxx,\measured_i)} = 2$ from eq.~\eqref{eq:residualPointCloud}. For the constraint set of $\parentheses{\MR,\vt}$, we have the 15 quadratic equality constraints from Lemma~\ref{lemma:quadraticConSOthree} for $\MR \in \SOthree$, and we have $g^t = 1 - \frac{\twonorm{\vt}^2}{T^2}$ for $\vt$ (the translation is bounded by a known value $T$). Therefore, for point cloud registration, we have $\vh = \vh^x \cup \vh^\theta$ with $\vh^x \doteq \vh^r$ and $\vg = \{ g^t \}$. The Archimedeanness of the constraint set can be seen from the following inequality:
\bea
T^2 + N - \twonorm{\xextend}^2 = T^2 \cdot g^t + \sumallpoints 1 \cdot h^{\theta_i} \geq 0. \label{eq:archimedeanPCR}
\eea


{\bf Example~\ref{eg:meshRegistration} (Mesh Registration)}. To make the residual function $r^2(\vxx,\measured_i)$ a quadratic polynomial, we perform the following change of variables and develop the residual function:
\bea
r^2(\vxx,\measured_i) = \twonorm{\parentheses{\MR\vu_i}\tran\parentheses{\vb_i - \MR\va_i - \vt}}^2 + w_i \twonorm{\vv_i - \MR\vu_i}^2 \\
 = \twonorm{\vu_i\tran \parentheses{\MR\tran\vb_i - \va_i - \MR\tran\vt}}^2 + w_i \twonorm{\vv_i - \MR\vu_i}^2 \\
 \overset{\MRR \doteq \MR\tran, \MRt \doteq \MR\tran\vt}{=} \twonorm{\vu_i\tran \MRR \vb_i - \vu_i\tran\va_i - \vu_i\tran\MRt}^2 + w_i \twonorm{\vv_i - \MRR\tran \vu_i}^2 \\
 = \MRt\tran \parentheses{\vu_i \kron \vu_i\tran} \MRt + \MRr\tran \parentheses{\vb_i\vb_i\tran \kron \vu_i\vu_i\tran} \MRr - 2 \vectorize{\vu_i\va_i\tran\vu_i\vb_i\tran}\tran \MRr + 2 \vu_i\tran\va_i\vu_i\tran \MRt \nonumber  \\
 - 2 \MRr\tran \parentheses{\vb_i \kron \vu_i\vu_i\tran}\MRt  - 2w_i\vectorize{\vu_i\vv_i\tran}\tran\MRr  + \parentheses{\vu_i\tran\va_i}^2 + w_i \parentheses{\twonorm{\vv_i}^2 + \twonorm{\vu_i}^2 },
\eea
where $\MRR \doteq \MR\tran \in \SOthree$ and $\MRt \doteq \MR\tran \vt \in \Real{3}$ is the new set of unknown rotation and translation. In addition, $\twonorm{\vt} \leq T$ if and only if $\twonorm{\MRt} \leq T$ because the rotation matrix preserves the norm of $\vt$. The original $\parentheses{\MR,\vt}$ can be recovered from $\parentheses{\MRR,\MRt}$ by:
\bea
\MR = \MRR\tran, \quad \vt = \MRR\tran \MRt.
\eea
The constraints for $\parentheses{\MRR,\MRt}$ is the same as what we developed for point cloud registration: $\vh = \vh^x \cup \vh^\theta$ with $\vh^x \doteq \vh^\MRlr \doteq \vh^r$, and $\vg \doteq \vg^\MRlt \doteq \{g^t\}$. Therefore, the Archimedeanness of the feasible set follows from eq.~\eqref{eq:archimedeanPCR}. This concludes the proof for Proposition~\ref{prop:pop}.
\end{proof}

%% file: supp-proof-denseRelaxation.tex
\newcommand{\matIdColor}[1]{\red{#1}}
\section{Explanation and Example for Theorem~\ref{thm:denseMoment} (Dense Moment Relaxation)}
\label{sec:supp-proof-denseRelaxation}
In this section, we provide a brief but self-contained explanation to shed light on Lasserre's hierarchy of dense moment relaxations in Theorem~\ref{thm:denseMoment} (Section~\ref{sec:explanation}). We also give an accessible example to demonstrate the application of the hierarchy to a simple but illustrative problem, namely 2D single rotation averaging (Section~\ref{sec:2DSRA}). 

\subsection{Explanation}
\label{sec:explanation}
Our explanation of Lasserre's hierarchy is adapted from~\cite{lasserre10book-momentsOpt,Lasserre01siopt-LasserreHierarchy}. Let $\mu\parentheses{\xextend}$ be a probability measure supported on the feasible set $\calxextend$ of the POP~\eqref{eq:pop}, and let $\probin{\calxextend}$ be the set of all possible probability measures on $\calxextend$. Then the POP~\eqref{eq:pop} can be rewritten as a generalized moment problem.
\begin{theorem}[\bf POP as the Moment Problem~{\cite[Proposition 2.1]{Lasserre01siopt-LasserreHierarchy}}]
\label{thm:popMoment}
Let the feasible set of the POP~\eqref{eq:pop} be $\calxextend$, then the POP is equivalent to the following optimization:
\bea \label{eq:momentProblem}
f^\star_\mu = \min_{\mu \in \probin{\calxextend}} \int f(\xextend) d\mu,
\eea
in the sense that:
\begin{enumerate}[label=(\roman*)]
\item $f^\star_\mu = f^\star$;
\item if $\xextend^\star$ is a (potentially not unique) global minimizer of the POP~\eqref{eq:pop}, then $\mu^\star = \delta_{\xextend^\star}$ is a global minimizer of the moment problem~\eqref{eq:momentProblem}, where $\delta_{\xextend^\star}$ is the Dirac measure at $\xextend^\star$;
\item assuming the POP~\eqref{eq:pop} has a (potentially not unique) global minimizer with global minimum $f^\star$, then for every optimal solution $\mu^\star$ of the moment problem~\eqref{eq:momentProblem}, $f(\xextend) = f^\star$, $\mu^\star$-almost everywhere (\ie~$\mu^\star\parentheses{\cbrace{\xextend: f(\xextend) \neq f^\star}} = 0$ and $\mu^\star$ is supported only on the global minimizers of the POP);
\item if $\xextend^\star$ is the unique global minimizer of the POP~\eqref{eq:pop}, then $\mu^\star = \delta_{\xextend^\star}$ is the unique global minimizer of the moment problem~\eqref{eq:momentProblem}.
\end{enumerate}
\end{theorem}
Although the moment problem~\eqref{eq:momentProblem} is convex~\cite{lasserre10book-momentsOpt}, it is infinite-dimensional and still intractable. Therefore, the crux of making the optimization tractable is to relax the infinite-dimensional problem into a finite-dimensional one. Towards this goal, we introduce the notion of moments, moment matrices and localizing matrices.
\begin{definition}[\bf Moments, Moment Matrices, Localizing Matrices~{\cite[Chapter 3]{lasserre10book-momentsOpt}}] 
\label{def:moments}
Given a probability measure $\mu$ supported on $\calxextend \subset \Real{\dimxextend}$, its moment of order $\valpha \in \nnint^\dimxextend$ is the scalar $\moment_{\valpha} \doteq \int_\calxextend \xextend^\valpha d\mu = \expect{\mu}{\xextend^\valpha} \in \Real{}$,~\ie~the integral (expected value) of the monomial $\xextend^\valpha$ over the set $\calxextend$~\wrt~$\mu$. In particular, if $\valpha = \zero$, then $\xextend^\valpha = p_1^{\alpha_1}\cdots p_{\dimxextend}^{\alpha_\dimxextend} = 1$, and $\moment_\zero = 1$. Now let $\moments = (\moment_\valpha)$ be an infinite sequence of moments (the order $\valpha$ can be unbounded), we define the linear functional $L_\moments: \polyringin{\xextend} \rightarrow \Real{}$:
\bea
f(\xextend) = \sum_{\valpha \in \calF} c(\valpha) \xextend^\valpha \mapstochar\rightarrow L_\moments(f) = \sum_{\valpha \in \calF} c(\valpha) \moment_\valpha,
\eea
that maps a polynomial $f$ to a real number $L_\moments(f)$ by replacing the monomials of $f$ with corresponding moments. With this linear functional, the moment sequence of degree up to $2\rorder$ is simply:
\bea
\moments_{2\rorder} \doteq L_\moments \parentheses{\monoleq{\xextend}{2\rorder}} \in \Real{\dimbasis{\dimxextend}{2\rorder}},\label{eq:defMoment}
\eea
where the linear functional $L_\moments$ applies component-wise to the vector of monomials $\monoleq{\xextend}{2\rorder}$, and the \emph{moment matrix} of degree $\rorder$ is:
\bea
\MM_{\rorder}(\moments_{2\rorder}) \doteq L_\moments \parentheses{\monoleq{\xextend}{\rorder}\monoleq{\xextend}{\rorder}\tran} \in \sym^{\dimbasis{\dimxextend}{\rorder}},
\eea
where $L_\moments$ also applies component-wise to the monomial matrix $\monoleq{\xextend}{\rorder}\monoleq{\xextend}{\rorder}\tran$, and $\MM_{\rorder}(\moments_{2\rorder})$ essentially assembles the vector of moments $\moments_{2\rorder}$ into a symmetric matrix. Finally, given a polynomial $h \in \polyringin{\xextend}$, we definite the \emph{localizing matrix} of order $\rorder$ with respect to $\moments$ and $h$ to be:
\bea
\MM_\rorder\parentheses{h \moments_{2\rorder}} \doteq L_\moments \parentheses{h \cdot \parentheses{\monoleq{\xextend}{\rorder}\monoleq{\xextend}{\rorder}\tran} } \in \sym^{\dimbasis{\dimxextend}{\rorder}},
\eea
where $L_\moments$ applies component-wise, and $h \cdot \parentheses{\monoleq{\xextend}{\rorder}\monoleq{\xextend}{\rorder}\tran}$ means multiplying $h$ with each entry of the monomial matrix $\monoleq{\xextend}{\rorder}\monoleq{\xextend}{\rorder}\tran$.
\end{definition}
With this definition, we can see that the cost function of the moment problem~\eqref{eq:momentProblem} is a linear function of the moments:
\bea
\int f(\xextend) d\mu = \int \sum_{\valpha \in \calF} c(\valpha) \xextend^\valpha d\mu = \sum_{\valpha \in \calF} c(\valpha) \int \xextend^\valpha d\mu = \sum_{\alpha \in \calF} c(\valpha) \moment_\valpha.
\eea
Therefore, instead of finding the probability measure $\mu$ directly in the infinite-dimensional space $\probin{\calxextend}$ as written in eq.~\eqref{eq:momentProblem}, we can equivalently search for the sequence of (possibly finite number of) moments $\moments$ and then recover the measure $\mu$ from the moments $\moments$. However, not every sequence of moments has a representing measure. In fact, in order to have a representing measure, the moment sequence has to satisfy the following conditions.

\begin{theorem}[\bf Necessary and Sufficient Condition for Representing Measure{~\cite[Theorem 3.8(b), p.~63]{lasserre10book-momentsOpt}}] Let $\moments = (\moment_\valpha)$ be a given infinite sequence of moments, and let $\calxextend$ be an Archimedean constraint set defined by the polynomial equality and inequality constraints in the POP~\eqref{eq:pop}. Then, the sequence $\moments$ has a representing measuring on $\calxextend$ if and only if:
\bea
\forall \rorder \in \mathbb{N}: \ \MM_{\rorder}(\moments_{2\rorder}) \succeq 0; \ \MM_{\rorder}(h\moments_{2\rorder}) = \zero, \forall h \in \vh; \ \MM_{\rorder}(g\moments_{2\rorder}) \succeq 0, \forall g \in \vg. \label{eq:representMeasure}
\eea
\end{theorem}
Enforcing the PSD constraints in~\eqref{eq:representMeasure} for every $\rorder \in \natural$ (potentially unbounded) is intractable due to the infinite moment sequence $\moments$. Therefore, a natural strategy is to enforce the constraints for a fix order $\rorder$ (called the relaxation order), which is precisely the optimization~\eqref{eq:denseMoment} in Theorem~\ref{thm:denseMoment}.\footnote{Because the constraint polynomials $\vh$ and $\vg$ have degree 2, the localizing matrices of order $\rorder-1$ is used to make sure every moment appearing in the localizing matrices also appears in the moment matrix $\MM_\rorder(\moments_{2\rorder})$. In a more general setting where the constraint polynomials $h_i$ (or $g_i$) have degree $2v_i$ or $2v_i-1$, then the localizing matrices of degree $\rorder - v_i$ should be used. } Problem~\eqref{eq:denseMoment} is a \emph{relaxation} of the moment problem~\eqref{eq:momentProblem} because the constraints at a fixed $\rorder$ only provide a \emph{necessary} condition for the existence of a representing measure, which in turn implies that the global minimum of the relaxation, $p^\star_\rorder$, is a \emph{lower bound} of the global minimum of the moment problem~\eqref{eq:momentProblem} (and thus the POP~\eqref{eq:pop}): 
\bea
p^\star_\rorder \leq f^\star_\mu = f^\star. \label{eq:lowerBoundfixedK}
\eea 
Lasserre's hierarchy is a hierarchy of moment relaxations with increasing relaxation orders $\rorder_1 < \rorder_2 < \dots$ (and increasing lower bounds $p^\star_{\rorder_1} \leq p^\star_{\rorder_2} < \dots$ ) until the relaxation is tight (\ie~$p^\star_{\rorder} = f^\star_\mu$). In general, Lasserre's hierarchy may achieve tightness only \emph{asymptotically} (\ie~$p^\star_{\rorder} \rightarrow f^\star_\mu$ as $\rorder \rightarrow \infty$). However, when the feasible set $\calxextend$ is Archimedean, Nie~\cite{Nie14mp-finiteConvergenceLassere} showed that the hierarchy terminates at a finite relaxation order, which is the case for our POP~\eqref{eq:pop} arising from a broad class of \perception problems (\cf claim~\ref{prop:pop-archimedean} in Proposition~\ref{prop:pop}).

Now a natural question is, how can one determine when the relaxation is tight (and terminate the hierarchy) without knowing the true global minimum $f^\star$? In other words, how to compute an \emph{optimality certificate}, and possibly recover the global minimizers of the POP~\eqref{eq:pop} from the solution of the moment relaxation~\eqref{eq:denseMoment}? Both questions boil down to checking if the solution of the moment relaxation, $\moments^\star_{2\rorder}$, has a representing measure on $\calxextend$, which is known as the \emph{truncated $\mathbb{K}$-moment problem}~\cite{Curto00TAMS-truncatedKmoment}. The following theorem states a sufficient condition.
\begin{theorem}[\bf Sufficient Condition for Truncated $\mathbb{K}$-Moment Problem~{\cite[Theorem 3.11, p.~66]{lasserre10book-momentsOpt}}] \label{thm:sufficientTruncatedK}
Let $\calxextend$ be the feasible set of the POP~\eqref{eq:pop}, where both $h_j$ and $g_k$ are quadratic polynomials. Let $\moments^\star_{2\rorder}$ be the solution of the moment relaxation~\eqref{eq:denseMoment}. Then $\moments^\star_{2\rorder}$ admits an $r$-atomic representing measure supported on $\calxextend$, with $r = \rank{\MM_{\rorder-1}\parentheses{\moments^\star_{2\rorder-2}}}$, if:
\bea
\rank{\MM_{\rorder-1}\parentheses{\moments^\star_{2\rorder-2}}} = \rank{\MM_{\rorder}\parentheses{\moments^\star_{2\rorder}}}.\label{eq:sufficientKMoment}
\eea
\end{theorem}
Theorem~\ref{thm:sufficientTruncatedK} is a special case of Theorem 3.11 in~\cite{lasserre10book-momentsOpt}, where we have used the fact that $\calxextend$ are defined by quadratic polynomials.\footnote{In the general case, suppose $\calxextend$ are defined by polynomials with degree $2v_i$ or $2v_i-1, i=1,\dots,\nrEq+\nrIneq$, then denote $v = \max_{i} v_i$, the sufficient condition becomes $\rank{\MM_{\rorder-v}\parentheses{\moments_{2\rorder-2v}}} = \rank{\MM_{\rorder}\parentheses{\moments_{2\rorder}}}$.} In particular, for the POP arising from \perception problems, at the minimum relaxation order $\rorder =2 $, we usually have $\rank{\MM_\rorder\parentheses{\moments^\star_{2\rorder}}} = 1$, which immediately implies that $r = \rank{\MM_{\rorder-1}\parentheses{\moments^\star_{2\rorder-2}}}=\rank{\MM_\rorder\parentheses{\moments^\star_{2\rorder}}} = 1$ (because $\MM_{\rorder-1}\parentheses{\moments^\star_{2\rorder-2}}$ is a nonzero principal sub-matrix of $\MM_\rorder\parentheses{\moments^\star_{2\rorder}}$), and $\moments^\star_{2\rorder}$ admits a $1$-atomic representing measure $\mu = \delta_{\xextend^\star}$. Therefore, from Theorem~\ref{thm:popMoment}, we have $\xextend^\star$ is the unique global minimizer of the POP~\eqref{eq:pop}.\footnote{The uniqueness of the solution comes from the fact that Interior Point Methods solvers (\eg~SeDuMi) output an optimal solution of \emph{maximum rank} if the SDP has more than one optimal solutions~\cite{Wolkowicz12book-handbookSDP}. Therefore, if $\xextend^\star$ is not unique, then the SDP will have multiple optimal solutions and the rank of the solution will be larger than 1 (\cf Theorem 6.18 in~\cite{Laurent09eaag-SOSMomentOptimization}).} Additionally, for $\mu = \delta_{\xextend^\star}$, it is straightforward to verify that $\moments^\star_{2\rorder} = \monoleq{\xextend^\star}{2\rorder}$ from eq.~\eqref{eq:defMoment}, and $\xextend^\star$ can be directly read off from the moments. 

{\bf Rounding and Relative Duality Gap}. When the sufficient condition eq.~\eqref{eq:sufficientKMoment} does not hold and the moment matrix $\MM_\rorder\parentheses{\moments^\star_{2\rorder}}$ has rank larger than 1, we can first perform spectral decomposition on $\MM_\rorder\parentheses{\moments^\star_{2\rorder}}$, and extract its eigenvector corresponding to the largest eigenvalue, denoted as $\vv_\rorder \in \Real{\dimbasis{\dimxextend}{\rorder}}$. Then we normalize $\vv_\rorder$'s first entry, $\vv_\rorder(1)$, to be 1 by: $\vv_\rorder \leftarrow \frac{\vv_\rorder}{\vv_\rorder(1)}$. If the relaxation were tight, then $\vv_\rorder = \monoleq{\xextend^\star}{\rorder}$ is a vector of moments up to degree $\rorder$ and $\xextend^\star$ can be directly read off from $\vv_\rorder$. However, since the relaxation is not tight, we obtain a feasible point $\hatxextend$ by: 
\bea
\hatxextend = \proj_{\calxextend}\parentheses{\vv_\rorder\parentheses{\xextend}}, \label{eq:projection}
\eea 
where $\vv_\rorder\parentheses{\xextend}$ denotes the entries of $\vv_\rorder$ at indices corresponding to the locations of $\xextend$ in $\monoleq{\xextend}{\rorder}$, and $\proj_{\calxextend}$ performs the projection onto the feasible set $\calxextend$ (see Section~\ref{sec:projectionP} for details). Let $\hatf = f(\hatxextend)$, we have the following inequality:
\bea
p^\star_\rorder \leq f^\star \leq \hatf,
\eea
where the first inequality follows from eq.~\eqref{eq:lowerBoundfixedK}, and the second inequality holds because $f^\star$ is the global minimum of the POP~\eqref{eq:pop}. The relative duality gap can be computed as:
\bea
\eta_\rorder = \frac{\hatf - p^\star_\rorder}{\hatf}. \label{eq:relDualityGapDense}
\eea
A smaller $\eta_\rorder$ implies a tighter relaxation and $\eta_\rorder = 0$ if and only if the relaxation is tight.

\subsection{Example: 2D Single Rotation Averaging} 
\label{sec:2DSRA}
To make our explanation of Lasserre's hierarchy in Section~\ref{sec:explanation} more accessible, we show an application of the hierarchy to 2D single rotation averaging, because the dimension of $\vxx$ is small and the constraint set $\calX$ is simple to characterize. 2D single rotation averaging follows the definition of 3D single rotation averaging in Example~\ref{eg:singleRotationAveraging}, except that the measurements $\MR_i,i=1,\dots,N$ and the unknown geometric model $\MR$ are 2D rotation matrices,~\ie~$\MR \in \SOtwo$. In this case, we describe a 2D rotation matrix using:
\bea
\MR \in \SOtwo \Longleftrightarrow \MR = \bmat{cc} x_1 & -x_2 \\ x_2 &  x_1 \emat, \quad \subject \quad h^x = 1-x_1^2 - x_2^2 = 0.
\eea
We then choose $N=2$, leading to two binary variables $\theta_1$ and $\theta_2$. Denote $\vxx = [x_1,x_2]\tran$, $\vtheta = [\theta_1,\theta_2]\tran$, and $\xextend = [\vxx\tran,\vtheta\tran]\tran$, the POP~\eqref{eq:pop} for 2D single rotation averaging with $N=2$ is:
\bea
\min_{\xextend = [x_1,x_2,\theta_1,\theta_2]\tran \in \Real{4}} & f(\xextend) \\
\subject & h^x = 1-x_1^2 - x_2^2 = 0, \label{eq:2DSRAhx} \\
  & h^{\theta_1} = 1 - \theta_1^2 = 0, \label{eq:2DSRAhtheta1} \\
  & h^{\theta_2} = 1 - \theta_2^2 = 0, \label{eq:2DSRAhtheta2}
\eea
where the objective function can be computed from eq.~\eqref{eq:TLSBinary}.

{\bf Moment matrices}. To describe the dense moment relaxation~\eqref{eq:denseMoment} at $\rorder = 2$, we first form the moment matrices $\MM_1\parentheses{\moments_2}$ and $\MM_2\parentheses{\moments_4}$. Towards this goal, let us write the vector of monomials $\monoleq{\xextend}{1}$ and $\monoleq{\xextend}{2}$:
\bea
& \monoleq{\xextend}{1} = [1,x_1,x_2,\theta_1,\theta_2]\tran \in \Real{5}, \label{eq:vecMono1} \\
& \monoleq{\xextend}{2} = [1,x_1,x_2,\theta_1,\theta_2,x_1^2,x_1 x_2, x_1 \theta_1,x_1 \theta_2,x_2^2,x_2 \theta_1,x_2\theta_2, \theta_1^2, \theta_1 \theta_2, \theta_2^2]\tran \in \Real{15}. \label{eq:vecMono2}
\eea
For notation simplicity, we use $\moment_{\xextend^\valpha}$, instead of $\moment_{\valpha}$, to denote the moment of order $\valpha$. For example, $\moment_{x_1 x_2} \doteq \int_{\calxextend} x_1x_2 d\mu$ for some probability measure $\mu$ supported on $\calxextend$. Then the vector of moments $\moments_1$ and $\moments_2$ directly follow from the vector of monomials in eq.~\eqref{eq:vecMono1} and~\eqref{eq:vecMono2}:
\bea
& \moments_1 = \bracket{\moment_1,\moment_{x_1},\moment_{x_2},\moment_{\theta_1},\moment_{\theta_2}}\tran \in \Real{5}, \\
& \hspace{-10mm} \moments_2 = \bracket{\moment_1,\moment_{x_1},\moment_{x_2},\moment_{\theta_1},\moment_{\theta_2},\moment_{x_1^2},\moment_{x_1 x_2}, \moment_{x_1 \theta_1},\moment_{x_1 \theta_2},\moment_{x_2^2},\moment_{x_2 \theta_1},\moment_{x_2\theta_2}, \moment_{\theta_1^2}, \moment_{\theta_1 \theta_2}, \moment_{\theta_2^2}} \in \Real{15}. \label{eq:momentVec2}
\eea
The vector of moments of degree up to 4, $\moments_4 \in \Real{\dimbasis{4}{4}=70}$ can be written in a similar way. We omit its full expression here, because it will soon appear in the moment matrix $\MM_2\parentheses{\moments_4}$. Then we are ready to form the moment matrix of order 1, with rows and columns indexed by $\monoleq{\xextend}{1}$:
\bea
& \MM_1\parentheses{\moments_2} = L_\moments\parentheses{\monoleq{\xextend}{1} \monoleq{\xextend}{1}\tran} =  
\begin{blockarray}{cccccc}
& \matIdColor{1} & \matIdColor{x_1} & \matIdColor{x_2} & \matIdColor{\theta_1} & \matIdColor{\theta_2} \\
\begin{block}{c[ccccc]}
\matIdColor{1}   	 & \moment_1 & \moment_{x_1} & \moment_{x_2} & \moment_{\theta_1} & \moment_{\theta_2}\\
\matIdColor{x_1} 	 & \star & \moment_{x_1^2} & \moment_{x_1 x_2} & \moment_{x_1 \theta_1} & \moment_{x_1 \theta_2} \\
\matIdColor{x_2} 	 & \star & \star & \moment_{x_2^2} & \moment_{x_2 \theta_1} & \moment_{x_2 \theta_2} \\
\matIdColor{\theta_1} & \star & \star & \star & \moment_{\theta_1^2} & \moment_{\theta_1 \theta_2} \\
\matIdColor{\theta_2} & \star & \star & \star & \star & \moment_{\theta_2^2} \\
\end{block}
\end{blockarray}, \label{eq:momentMat1}
\eea
where we see that the moments appearing in $\MM_1\parentheses{\moments_2}$ are exactly $\moments_2$ (compare upper triangular entries in~\eqref{eq:momentMat1} with~\eqref{eq:momentVec2}, thus the expression $\MM_1\parentheses{\moments_2}$). Similarly, we form the moment matrix of order 2, with rows and columns indexed by $\monoleq{\xextend}{2}$:
\tiny
\input{eq-largeMomentMat}
\normalsize
where the upper triangular entries are exactly $\moments_4$, the vector of moments up to degree $4$ (thus the expression $\MM_2\parentheses{\moments_4}$). Moreover, the moment matrix is called a \emph{generalized Hankel matrix} because a moment of order $\valpha$ can appear multiple times in the matrix. For example, the moment $\moment_{x_1x_2\theta_1\theta_2}$ (highlighted in blue) appears three times in the upper triangular part of $\MM_2\parentheses{\moments_4}$. 

{\bf Localizing matrices.} Using the moment matrix of order 1, $\MM_1\parentheses{\moments_2}$, the localizing matrix for $h^x = 1-x_1^2-x_2^2 = 0$ (eq.~\eqref{eq:2DSRAhx}) is:
\input{eq-largeLocalizeMat_x}
where the columns are indexed by $\monoleq{\xextend}{1}$, and the rows are indexed by $h^x\cdot \monoleq{\xextend}{1}$. The localizing matrix for $h^{\theta_1} = 1-\theta_1^2=0$ (eq.~\eqref{eq:2DSRAhtheta1}) is:
\input{eq-largeLocalizeMat_theta1}
where the columns are indexed by $\monoleq{\xextend}{1}$, and the rows are indexed by $h^{\theta_1}\cdot \monoleq{\xextend}{1}$. The localizing matrix for $h^{\theta_2} = 1-\theta_2^2=0$ (eq.~\eqref{eq:2DSRAhtheta2}) is:
\input{eq-largeLocalizeMat_theta2}
where the columns are indexed by $\monoleq{\xextend}{1}$, and the rows are indexed by $h^{\theta_2}\cdot \monoleq{\xextend}{1}$. 

{\bf Dense Moment Relaxation.} With the expressions of the moment matrices and localizing matrices, the dense moment relaxation at $\rorder=2$ for 2D single rotation averaging is:
\bea \label{eq:SDP2DSRA}
p^\star_{2} = \min_{\moments_4 \in \Real{70}} & \sum_{\valpha \in \calF} c(\valpha) \moment_{\xextend^\valpha} \\
\subject & \MM_2\parentheses{\moments_4} \succeq 0 \nonumber \quad \text{(\cf~eq.~\eqref{eq:momentMat_4})} ,\\
		& \MM_1\parentheses{h^x \moments_2} = \zero \nonumber  \quad \text{(\cf~eq.~\eqref{eq:localizeMoment_x})} ,\\
		& \MM_1\parentheses{h^{\theta_1} \moments_2} = \zero \quad \text{(\cf~eq.~\eqref{eq:localizeMoment_theta_1})} ,\nonumber \\
		& \MM_2\parentheses{h^{\theta_2} \moments_2} = \zero \quad \text{(\cf~eq.~\eqref{eq:localizeMoment_theta_2})} .\nonumber
\eea
Now it is clearly that problem~\eqref{eq:SDP2DSRA} is an SDP because the entries of the moment matrix $\MM_2\parentheses{\moments_4}$, and the localizing matrices $\MM_1\parentheses{h^x \moments_2},\MM_1\parentheses{h^{\theta_1} \moments_2},\MM_1\parentheses{h^{\theta_2} \moments_2}$ depend \emph{linearly} on the optimization variables $\moments_4$, and the objective function is also a linear function of $\moments_4$.
\begin{remark}[\bf Moment Relaxation for POP \myversus SDP Relaxation for QCQP]
The expert reader may now see connections between the moment relaxation for POPs and Shor's SDP relaxation for quadratically constrained quadratic programs (QCQP): the moment relaxation can be seen as first performing a change of variables so that the POP becomes a QCQP (\ie~using $\monoleq{\xextend}{2}$ as the new set of variables, the POP can be seen as a QCQP because $\monoleq{\xextend}{2}$ contains monomials of degree higher than 1), and then apply standard SDP relaxations with \emph{redundant constraints}. The redundant constraints come from (i) the new set of variables $\monoleq{\xextend}{2}$ are not mutually independent,~\eg~$x_1x_2 = x_1 \cdot x_2$, and hence the moment matrix $\MM_2\parentheses{\moments_4}$ possess \emph{linear equality} constraints,~\eg~the term $\moment_{x_1x_2\theta_1\theta_2}$ appears multiple times; (ii) combinations of equality constraints,~\eg~$h^x = 0$ and $h^{\theta_1} = 0$ implies $h^x\cdot h^{\theta_1} =0$. Therefore, converting a POP to a QCQP and then applying SDP relaxation (see~\cite{Briales18cvpr-global2view,Yang19iccv-QUASAR} for two examples) has two drawbacks: first, carefully listing the complete set of redundant constraints can be time-consuming; second, it is challenging to handle \emph{inequality} constraints. On the contrary, the localizing matrices in moment relaxation provide a systematic way to include all redundant equality and inequality constraints.  
\end{remark}
\subsection{Projection onto $\calxextend$}
\label{sec:projectionP}
Here we discuss how to project an estimate to the feasible set of the (POP): this is required to implement the rounding procedure
described in eq.~\eqref{eq:projection}.
Denote $\xextend_v = [\vxx_v\tran,\vtheta_v\tran]\tran = \vv_\rorder\parentheses{\xextend}$ as the entries of $\vv_\rorder$ at indices corresponding to the locations of $\xextend$ in $\monoleq{\xextend}{\rorder}$ (recall that $\vv_\rorder$ is obtained from the spectral decomposition of a moment matrix $\MM_\rorder\parentheses{\moments_4}$ with rank larger than 1, and in general $\xextend_v \not\in \calxextend$). To project $\xextend_v$ onto $\calxextend$ (eq.~\eqref{eq:projection}), we need to project $\vxx_v$ onto $\calX$ and project $\vtheta_v$ onto $\cbrace{\pm 1}_{i=1}^N$. 

{\bf Project $\vtheta_v$ onto $\cbrace{\pm 1}_{i=1}^N$}. The projection of $\vtheta_v$ onto the set of binary variables $\cbrace{\pm 1}_{i=1}^N$, denoted $\hatvtheta$, is straightforward:
\bea
\bracket{\hatvtheta}_i = \proj_{\cbrace{\pm 1}} \parentheses{\bracket{\vtheta_v}_i} = \sgn{\bracket{\vtheta_v}_i}, i = 1,\dots,N,
\eea
where $\bracket{\cdot}_i$ denotes the $i$-th entry of a vector and $\sgn{\cdot}$ denotes the sign function. 

{\bf Project $\vxx_v$ onto $\calX$}. Because Examples~\ref{eg:singleRotationAveraging}-\ref{eg:meshRegistration} have different feasible sets $\calX$ for the geometric models, the projections are different.

{\bf Example~\ref{eg:singleRotationAveraging} (Single Rotation Averaging)}. $\vxx = \MR \in \SOthree$, so the projection is:
\bea
\hatMR = \proj_{\SOthree} \parentheses{\vxx_v} = \MU \diag{1,1,\det\parentheses{\MU}\cdot\det\parentheses{\MV}} \MV\tran,
\eea
where $\vxx_v = \MU \MS \MV\tran, \MU,\MV \in \Othree$ is the singular value decomposition for $\vxx_v$ (first reshape $\vxx_v \in \Real{9}$ into a $3\times 3$ matrix)~\cite{Hartley13ijcv}.  

{\bf Example~\ref{eg:shapeAlignment} (Shape Alignment)}. $\vxx = s\Pi\MR, s \in [0,\sub], \MR \in \SOthree$, so the projection is:
\bea
\hats = \min\cbrace{\sub, \frac{\sigma_1 + \sigma_2}{2} },
\hatMR = \bmat{c} \vr_1\tran \\ \vr_2\tran \\ \parentheses{\vr_1 \times \vr_2}\tran \emat,
\eea
where $\sigma_1,\sigma_2,\vr_1,\vr_2$ come from the singular value decomposition of $\vxx_v$ (first reshape $\vxx_v$ into a $2\times 3$ matrix):
\bea
\vxx_v = \MU \bmat{ccc} \sigma_1 & 0 & 0 \\ 0 & \sigma_2 & 0 \emat \MV\tran, \MU \in \Otwo, \MV \in \Othree,\quad  \bmat{c} \vr_1\tran \\ \vr_2\tran \emat = \MU \bmat{ccc} 1 & 0 & 0 \\ 0 & 1 & 0 \emat \MV\tran.
\eea

{\bf Example~\ref{eg:pointCloudRegistration}-\ref{eg:meshRegistration} (Point Cloud Registration and Mesh Registration)}. $\vxx = (\MR,\vt), \MR \in \SOthree, \| \vt \| \leq T$, so the projection is:
\bea
\hatMR = \proj_{\SOthree}\parentheses{\vxx_v^r}, \hatvt = \min\cbrace{\twonorm{\vxx_v^t},T} \frac{\vxx_v^t}{\twonorm{\vxx_v^t}},
\eea
where $\vxx_v^r$ denotes the entries of $\vxx_v$ corresponding to $\MR$, and $\vxx_v^t$ denotes the entries of $\vxx_v$ corresponding to $\vt$.

%% file: eq-largeMomentMat.tex

\bea
 &  \hspace{-40mm} \MM_2\parentheses{\moments_4} = L_\moments\parentheses{\monoleq{\xextend}{2}\monoleq{\xextend}{2}\tran} = \\
&  \hspace{-40mm}\begin{blockarray}{cccccccccccccccc}
 &\scriptstyle \matIdColor{1} &\scriptstyle \matIdColor{x_1} &\scriptstyle \matIdColor{x_2} &\scriptstyle \matIdColor{\theta_1} &\scriptstyle \matIdColor{\theta_2} &\scriptstyle \matIdColor{x_1^2} &\scriptstyle \matIdColor{x_1 x_2} &\scriptstyle  \matIdColor{x_1 \theta_1} &\scriptstyle \matIdColor{x_1 \theta_2} &\scriptstyle \matIdColor{x_2^2} &\scriptstyle \matIdColor{x_2 \theta_1} &\scriptstyle \matIdColor{x_2\theta_2} &\scriptstyle \matIdColor{\theta_1^2} &\scriptstyle \matIdColor{\theta_1 \theta_2} &\scriptstyle \matIdColor{\theta_2^2} \\
\begin{block}{c[ccccccccccccccc]}
\scriptstyle \matIdColor{1} &\scriptstyle \moment_{1} &\scriptstyle \moment_{x_1} &\scriptstyle \moment_{x_2} &\scriptstyle \moment_{\theta_1} &\scriptstyle \moment_{\theta_2} &\scriptstyle \moment_{x_1^2} &\scriptstyle \moment_{x_1x_2} &\scriptstyle \moment_{x_1\theta_1} &\scriptstyle \moment_{x_1\theta_2} &\scriptstyle \moment_{x_2^2} &\scriptstyle \moment_{x_2\theta_1} &\scriptstyle \moment_{x_2\theta_2} &\scriptstyle \moment_{\theta_1^2} &\scriptstyle \moment_{\theta_1\theta_2} &\scriptstyle \moment_{\theta_2^2} \\
\scriptstyle \matIdColor{x_1} &\scriptstyle \star &\scriptstyle \moment_{x_1^2} &\scriptstyle \moment_{x_1x_2} &\scriptstyle \moment_{x_1\theta_1} &\scriptstyle \moment_{x_1\theta_2} &\scriptstyle \moment_{x_1^3} &\scriptstyle \moment_{x_1^2x_2} &\scriptstyle \moment_{x_1^2\theta_1} &\scriptstyle \moment_{x_1^2\theta_2} &\scriptstyle \moment_{x_1x_2^2} &\scriptstyle \moment_{x_1x_2\theta_1} &\scriptstyle \moment_{x_1x_2\theta_2} &\scriptstyle \moment_{x_1\theta_1^2} &\scriptstyle \moment_{x_1\theta_1\theta_2} &\scriptstyle \moment_{x_1\theta_2^2} \\
\scriptstyle \matIdColor{x_2} & \scriptstyle \star & \scriptstyle \star &\scriptstyle \moment_{x_2^2} & \scriptstyle \moment_{x_2\theta_1} & \scriptstyle \moment_{x_2\theta_2} & \scriptstyle \moment_{x_1^2x_2} & \scriptstyle \moment_{x_1x_2^2} & \scriptstyle \moment_{x_1x_2\theta_1} & \scriptstyle \moment_{x_1x_2\theta_2} & \scriptstyle \moment_{x_2^3} & \scriptstyle \moment_{x_2^2\theta_1} & \scriptstyle \moment_{x_2^2\theta_2} & \scriptstyle\moment_{x_2\theta_1^2} &\scriptstyle\moment_{x_2\theta_1\theta_2} &\scriptstyle\moment_{x_2\theta_2^2} \\
\scriptstyle\matIdColor{\theta_1} &\scriptstyle\star &\scriptstyle\star &\scriptstyle\star & \scriptstyle\moment_{\theta_1^2} & \scriptstyle\moment_{\theta_1\theta_2} &\scriptstyle\moment_{x_1^2\theta_1} &\scriptstyle\moment_{x_1x_2\theta_1} &\scriptstyle\moment_{x_1\theta_1^2} &\scriptstyle\moment_{x_1\theta_1\theta_2} &\scriptstyle\moment_{\theta_1x_2^2} &\scriptstyle\moment_{x_2\theta_1^2} &\scriptstyle\moment_{x_2\theta_1\theta_2} &\scriptstyle\moment_{\theta_1^3} &\scriptstyle\moment_{\theta_1^2\theta_2} &\scriptstyle\moment_{\theta_1\theta_2^2} \\
\scriptstyle\matIdColor{\theta_2} &\scriptstyle\star &\scriptstyle\star &\scriptstyle\star &\scriptstyle\star &\scriptstyle\moment_{\theta_2^2} &\scriptstyle\moment_{x_1^2\theta_2} &\scriptstyle\moment_{x_1x_2\theta_2} &\scriptstyle\moment_{x_1\theta_1\theta_2} &\scriptstyle\moment_{x_1\theta_2^2} &\scriptstyle\moment_{x_2^2\theta_2} &\scriptstyle\moment_{x_2\theta_1\theta_2} &\scriptstyle\moment_{x_2\theta_2^2} &\scriptstyle\moment_{\theta_1^2\theta_2} &\scriptstyle\moment_{\theta_1\theta_2^2} &\scriptstyle\moment_{\theta_2^3} \\
\scriptstyle \matIdColor{x_1^2} &\scriptstyle\star &\scriptstyle\star &\scriptstyle\star &\scriptstyle\star &\scriptstyle\star &\scriptstyle\moment_{x_1^4} &\scriptstyle\moment_{x_1^3x_2} &\scriptstyle\moment_{x_1^3\theta_1} &\scriptstyle\moment_{x_1^3\theta_2} &\scriptstyle\moment_{x_1^2x_2^2} &\scriptstyle\moment_{x_1^2x_2\theta_1} &\scriptstyle\moment_{x_1^2x_2\theta_2} &\scriptstyle\moment_{x_1^2\theta_1^2} &\scriptstyle\moment_{x_1^2\theta_1\theta_2} &\scriptstyle\moment_{x_1^2\theta_2^2} \\
\scriptstyle \matIdColor{x_1x_2} &\scriptstyle\star &\scriptstyle\star &\scriptstyle\star &\scriptstyle\star &\scriptstyle\star & \scriptstyle\star &\scriptstyle\moment_{x_1^2x_2^2} &\scriptstyle\moment_{x_1^2x_2\theta_1} &\scriptstyle\moment_{x_1^2x_2\theta_2} &\scriptstyle\moment_{x_1x_2^3} &\scriptstyle\moment_{x_1x_2^2\theta_1} &\scriptstyle\moment_{x_1x_2^2\theta_2} &\scriptstyle\moment_{x_1x_2\theta_1^2} &\scriptstyle\blue{\moment_{x_1x_2\theta_1\theta_2}} &\scriptstyle\moment_{x_1x_2\theta_2^2} \\
\scriptstyle \matIdColor{x_1\theta_1} &\scriptstyle\star &\scriptstyle\star &\scriptstyle\star &\scriptstyle\star &\scriptstyle\star &\scriptstyle\star &\scriptstyle\star &\scriptstyle\moment_{x_1^2\theta_1^2} &\scriptstyle\moment_{x_1^2\theta_1\theta_2} &\scriptstyle\moment_{x_1x_2^2\theta_1} &\scriptstyle\moment_{x_1x_2\theta_1^2} &\scriptstyle\blue{\moment_{x_1x_2\theta_1\theta_2}} &\scriptstyle\moment_{x_1\theta_1^3} &\scriptstyle\moment_{x_1\theta_1^2\theta_2} &\scriptstyle\moment_{x_1\theta_1\theta_2^2}\\
\scriptstyle \matIdColor{x_1\theta_2} &\scriptstyle\star &\scriptstyle\star &\scriptstyle\star &\scriptstyle\star &\scriptstyle\star &\scriptstyle\star &\scriptstyle\star &\scriptstyle\star &\scriptstyle\moment_{x_1^2\theta_2^2} &\scriptstyle\moment_{x_1x_2^2\theta_2} &\scriptstyle\blue{\moment_{x_1x_2\theta_1\theta_2}} &\scriptstyle\moment_{x_1x_2\theta_2^2} &\scriptstyle\moment_{x_1\theta_1^2\theta_2} &\scriptstyle\moment_{x_1\theta_1\theta_2^2} & \scriptstyle\moment_{x_1\theta_2^3}\\
\scriptstyle \matIdColor{x_2^2} &\scriptstyle\star &\scriptstyle\star &\scriptstyle\star &\scriptstyle\star &\scriptstyle\star &\scriptstyle\star &\scriptstyle\star &\scriptstyle\star &\scriptstyle\star &\scriptstyle\moment_{x_2^4} &\scriptstyle\moment_{x_2^3\theta_1} &\scriptstyle\moment_{x_2^3\theta_2} &\scriptstyle\moment_{x_2^2\theta_1^2} &\scriptstyle\moment_{x_2^2\theta_1\theta_2} &\scriptstyle\moment_{x_2^2\theta_2^2}\\
\scriptstyle \matIdColor{x_2\theta_1}  &\scriptstyle\star &\scriptstyle\star &\scriptstyle\star &\scriptstyle\star &\scriptstyle\star &\scriptstyle\star &\scriptstyle\star &\scriptstyle\star &\scriptstyle\star &\scriptstyle\star &\scriptstyle\moment_{x_2^2\theta_1^2} &\scriptstyle\moment_{x_2^2\theta_1\theta_2} &\scriptstyle\moment_{x_2\theta_1^3} &\scriptstyle\moment_{x_2\theta_1^2\theta_2} &\scriptstyle\moment_{x_2\theta_1\theta_2^2}\\
\scriptstyle \matIdColor{x_2\theta_2} &\scriptstyle\star &\scriptstyle\star &\scriptstyle\star &\scriptstyle\star &\scriptstyle\star &\scriptstyle\star &\scriptstyle\star &\scriptstyle\star &\scriptstyle\star &\scriptstyle\star &\scriptstyle\star &\scriptstyle\moment_{x_2^2\theta_2^2} &\scriptstyle\moment_{x_2\theta_1^2\theta_2} &\scriptstyle\moment_{x_2\theta_1\theta_2^2} &\scriptstyle\moment_{x_2\theta_2^3}\\
\scriptstyle \matIdColor{\theta_1^2} &\scriptstyle\star &\scriptstyle\star &\scriptstyle\star &\scriptstyle\star &\scriptstyle\star &\scriptstyle\star &\scriptstyle\star &\scriptstyle\star &\scriptstyle\star &\scriptstyle\star &\scriptstyle\star &\scriptstyle\star &\scriptstyle\moment_{\theta_1^4} &\scriptstyle\moment_{\theta_1^3\theta_2} &\scriptstyle\moment_{\theta_1^2\theta_2^2}\\
\scriptstyle \matIdColor{\theta_1\theta_2} &\scriptstyle\star &\scriptstyle\star &\scriptstyle\star &\scriptstyle\star &\scriptstyle\star &\scriptstyle\star &\scriptstyle\star &\scriptstyle\star &\scriptstyle\star &\scriptstyle\star &\scriptstyle\star &\scriptstyle\star &\scriptstyle\star &\scriptstyle\moment_{\theta_1^2\theta_2^2} &\scriptstyle\moment_{\theta_1\theta_2^3}\\
\scriptstyle \matIdColor{\theta_2^2} &\scriptstyle\star &\scriptstyle\star &\scriptstyle\star &\scriptstyle\star &\scriptstyle\star &\scriptstyle\star &\scriptstyle\star &\scriptstyle\star &\scriptstyle\star &\scriptstyle\star &\scriptstyle\star &\scriptstyle\star &\scriptstyle\star &\scriptstyle\star &\scriptstyle\moment_{\theta_2^4}\\
\end{block}
\end{blockarray}\!\!,\ \ \ \label{eq:momentMat_4}
\eea

%% file: eq-largeLocalizeMat_x.tex
\bea
& \MM_1\parentheses{h^x \moments_2} = L_\moments\parentheses{h^x \monoleq{\xextend}{1}\monoleq{\xextend}{1}\tran} = \\
&\scriptstyle \hspace{-25mm} \begin{blockarray}{cccccc}
&\scriptstyle \matIdColor{1-x_1^2-x_2^2} &\scriptstyle \matIdColor{x_1-x_1^3-x_1x_2^2} &\scriptstyle \matIdColor{x_2 - x_1^2x_2 - x_2^3} &\scriptstyle \matIdColor{\theta_1 - x_1^2\theta_1 - x_2^2\theta_1} &\scriptstyle \matIdColor{\theta_2 - x_1^2\theta_2 - x_2^2\theta_2} \\
\begin{block}{c[ccccc]}
\scriptstyle \matIdColor{1}  	 &\scriptstyle \moment_1 - \moment_{x_1^2} - \moment_{x_2^2} &\scriptstyle \moment_{x_1} - \moment_{x_1^3} - \moment_{x_1x_2^2} &\scriptstyle \moment_{x_2} - \moment_{x_1^2x_2} - \moment_{x_2^3} &\scriptstyle \moment_{\theta_1} - \moment_{x_1^2\theta_1} - \moment_{x_2^2\theta_1} &\scriptstyle \moment_{\theta_2} - \moment_{x_1^2\theta_2} - \moment_{x_2^2\theta_2} \\
\scriptstyle \matIdColor{x_1} 	 &\scriptstyle \star &\scriptstyle \moment_{x_1^2} - \moment_{x_1^4} - \moment_{x_1^2x_2^2} &\scriptstyle \moment_{x_1 x_2} -\moment_{x_1^3x_2}-\moment_{x_1x_2^3} &\scriptstyle \moment_{x_1 \theta_1} - \moment_{x_1^3\theta_1} - \moment_{x_1x_2^2\theta_1} &\scriptstyle \moment_{x_1 \theta_2} - \moment_{x_1^3\theta_2} - \moment_{x_1x_2^2\theta_2} \\
\scriptstyle \matIdColor{x_2} 	 &\scriptstyle \star &\scriptstyle \star &\scriptstyle \moment_{x_2^2} - \moment_{x_1^2x_2^2} - \moment_{x_2^4} &\scriptstyle \moment_{x_2 \theta_1}-\moment_{x_1^2x_2\theta_1}-\moment_{x_2^3\theta_1}&\scriptstyle \moment_{x_2 \theta_2}-\moment_{x_1^2x_2\theta_2}-\moment_{x_2^3\theta_2} \\
\scriptstyle \matIdColor{\theta_1} &\scriptstyle \star &\scriptstyle \star &\scriptstyle \star &\scriptstyle \moment_{\theta_1^2}-\moment_{x_1^2\theta_1^2}-\moment_{x_2^2\theta_1^2} &\scriptstyle \moment_{\theta_1 \theta_2}-\moment_{x_1^2\theta_1\theta_2}-\moment_{x_2^2\theta_1\theta_2} \\
\scriptstyle \matIdColor{\theta_2} &\scriptstyle \star &\scriptstyle \star &\scriptstyle \star &\scriptstyle \star &\scriptstyle \moment_{\theta_2^2}-\moment_{x_1^2\theta_2^2}-\moment_{x_2^2\theta_2^2} \\
\end{block}
\end{blockarray}, \label{eq:localizeMoment_x}
\eea

%% file: eq-largeLocalizeMat_theta1.tex
\bea
& \MM_1\parentheses{h^{\theta_1} \moments_2} = L_\moments\parentheses{h^{\theta_1} \monoleq{\xextend}{1}\monoleq{\xextend}{1}\tran} = \\
&\scriptstyle  \begin{blockarray}{cccccc}
&\scriptstyle \matIdColor{1-\theta_1^2} &\scriptstyle \matIdColor{x_1-x_1\theta_1^2} &\scriptstyle \matIdColor{x_2 - x_2\theta_1^2} &\scriptstyle \matIdColor{\theta_1 - \theta_1^3} &\scriptstyle \matIdColor{\theta_2 - \theta_1^2\theta_2} \\
\begin{block}{c[ccccc]}
\scriptstyle \matIdColor{1 } 	 &\scriptstyle \moment_1 - \moment_{\theta_1^2} &\scriptstyle \moment_{x_1} - \moment_{x_1\theta_1^2}  &\scriptstyle \moment_{x_2} - \moment_{x_2\theta_1^2} &\scriptstyle \moment_{\theta_1} - \moment_{\theta_1^3} &\scriptstyle \moment_{\theta_2} - \moment_{\theta_1^2\theta_2} \\
\scriptstyle \matIdColor{x_1 }	 &\scriptstyle \star &\scriptstyle \moment_{x_1^2} - \moment_{x_1^2\theta_1^2} &\scriptstyle \moment_{x_1 x_2} -\moment_{x_1x_2\theta_1^2} &\scriptstyle \moment_{x_1 \theta_1} - \moment_{x_1\theta_1^3} &\scriptstyle \moment_{x_1 \theta_2} - \moment_{x_1\theta_1^2\theta_2} \\
\scriptstyle \matIdColor{x_2} 	 &\scriptstyle \star &\scriptstyle \star &\scriptstyle \moment_{x_2^2} - \moment_{x_2^2\theta_1^2} &\scriptstyle \moment_{x_2 \theta_1}-\moment_{x_2\theta_1^3} &\scriptstyle \moment_{x_2 \theta_2}-\moment_{x_2\theta_1^2\theta_2} \\
\scriptstyle \matIdColor{\theta_1} &\scriptstyle \star &\scriptstyle \star &\scriptstyle \star &\scriptstyle \moment_{\theta_1^2}-\moment_{\theta_1^4} &\scriptstyle \moment_{\theta_1 \theta_2}-\moment_{\theta_1^3\theta_2} \\
\scriptstyle \matIdColor{\theta_2} &\scriptstyle \star &\scriptstyle \star &\scriptstyle \star &\scriptstyle \star &\scriptstyle \moment_{\theta_2^2}-\moment_{\theta_1^2\theta_2^2}\\
\end{block}
\end{blockarray}, \label{eq:localizeMoment_theta_1}
\eea

%% file: eq-largeLocalizeMat_theta2.tex
\bea
& \MM_1\parentheses{h^{\theta_2} \moments_2} = L_\moments\parentheses{h^{\theta_2} \monoleq{\xextend}{1}\monoleq{\xextend}{1}\tran} = \\
&\scriptstyle  \begin{blockarray}{cccccc}
&\scriptstyle \matIdColor{1-\theta_2^2} &\scriptstyle \matIdColor{x_1-x_1\theta_2^2} &\scriptstyle \matIdColor{x_2 - x_2\theta_2^2} &\scriptstyle \matIdColor{\theta_1 - \theta_1\theta_2^2} &\scriptstyle \matIdColor{\theta_2 - \theta_2^3} \\
\begin{block}{c[ccccc]}
\scriptstyle \matIdColor{1} 	 &\scriptstyle \moment_1 - \moment_{\theta_2^2} &\scriptstyle \moment_{x_1} - \moment_{x_1\theta_2^2}  &\scriptstyle \moment_{x_2} - \moment_{x_2\theta_2^2} &\scriptstyle \moment_{\theta_1} - \moment_{\theta_1\theta_2^2} &\scriptstyle \moment_{\theta_2} - \moment_{\theta_2^3} \\
\scriptstyle \matIdColor{x_1}	 &\scriptstyle \star &\scriptstyle \moment_{x_1^2} - \moment_{x_1^2\theta_2^2} &\scriptstyle \moment_{x_1 x_2} -\moment_{x_1x_2\theta_2^2} &\scriptstyle \moment_{x_1 \theta_1} - \moment_{x_1\theta_1\theta_2^2} &\scriptstyle \moment_{x_1 \theta_2} - \moment_{x_1\theta_2^3} \\
\scriptstyle \matIdColor{x_2} 	 &\scriptstyle \star &\scriptstyle \star &\scriptstyle \moment_{x_2^2} - \moment_{x_2^2\theta_2^2} &\scriptstyle \moment_{x_2\theta_1}-\moment_{x_2\theta_1\theta_2^2} &\scriptstyle \moment_{x_2\theta_2}-\moment_{x_2\theta_2^3} \\
\scriptstyle \matIdColor{\theta_1} &\scriptstyle \star &\scriptstyle \star &\scriptstyle \star &\scriptstyle \moment_{\theta_1^2}-\moment_{\theta_1^2\theta_2^2} &\scriptstyle \moment_{\theta_1 \theta_2}-\moment_{\theta_1\theta_2^3} \\
\scriptstyle \matIdColor{\theta_2} &\scriptstyle \star &\scriptstyle \star &\scriptstyle \star &\scriptstyle \star &\scriptstyle \moment_{\theta_2^2}-\moment_{\theta_2^4}\\
\end{block}
\end{blockarray}, \label{eq:localizeMoment_theta_2}
\eea

%% file: supp-proof-sparseRelaxation.tex
\section{Proof of Theorem~\ref{thm:sparseMoment} (Sparse Moment Relaxation)}
\label{sec:supp-proof-sparseRelaxation}
\begin{proof}
Let us first show that the sparse moment relaxation~\eqref{eq:sparseMoment} is indeed a valid relaxation,~\ie~(a) the sparse set of monomials $\monoleq{\xextend}{2\rbasisset}$ contains all the monomials in the objective function $f(\xextend)$ of the POP~\eqref{eq:pop} (otherwise, the objective function of the relaxation~\eqref{eq:sparseMoment} is not equivalent to the objective function of the POP~\eqref{eq:pop}); and (b) the sparse set of moments $\moments_{2\rbasisset}$ contains all the moments appearing in the localizing matrices of~\eqref{eq:sparseMoment} (otherwise, the optimization contains undefined variables). To see (a), from the sparsity of the objective function $f$ (\cf~property~\ref{prop:pop-objective} of Proposition~\ref{prop:pop}), we know that $f = \sumallpoints f_i$ and each $f_i$ at most contains monomials of type $\monoleq{\vxx}{2}$, $\theta_i$ and $\theta_i \cdot \monoleq{\vxx}{2}$ (\cf~the expression of $f_i$ in eq.~\eqref{eq:TLSBinary}), all of which are included in the sparse set of monomials $\monoleq{\xextend}{2\rbasisset}$ (recall $\monoleq{\xextend}{\rbasisset} = [1,\vxx\tran,\vtheta\tran,\monoindeg{\vxx}{2}\tran,\vtheta\tran \kron \vxx\tran]\tran$). To see (b), we observe that $h^x$ and $g$ only contain monomials $\monoleq{\vxx}{2}$, and $\moments_2$ only contain monomials $\monoleq{\xextend}{2}$, so the product $\monoleq{\vxx}{2} \kron \monoleq{\xextend}{2}$ is included in $\monoleq{\xextend}{2\rbasisset}$. Hence, the moments in the localizing matrices $\MM_1\parentheses{h^x\moments_2}$ and $\MM_1\parentheses{g\moments_2}$ are included in the moment vector $\moments_{2\rbasisset}$. Similarly, $h^\theta$ only contains monomials $\monoleq{\vtheta}{2}$, and $\moments_{2\rbasisset_x}$ only contains monomials $\monoleq{\vxx}{2}$, so the product $\monoleq{\vtheta}{2} \kron \monoleq{\vxx}{2}$ is included in $\monoleq{\xextend}{2\rbasisset}$. Hence, the moments in the localizing matrices $\MM_{\rbasisset_x}\parentheses{h^\theta \moments_{2\rbasisset_x}}$ are also included in the moment vector $\moments_{2\rbasisset}$.

{\bf Lower Bound}. Then we prove that $p^\star_{\rbasisset} \leq p^\star_2$,~\ie~the optimal value of the sparse relaxation~\eqref{eq:sparseMoment} is a lower bound of the optimal value of the dense relaxation~\eqref{eq:denseMoment} at order $\rorder=2$. We prove this by showing that the feasible set of the dense relaxation is contained in the feasible set of the sparse relaxation. To see this, consider $\moments_4$ as an arbitrary point in the feasible set of the dense relaxation~\eqref{eq:denseMoment},~\ie~$\moments_4$ satisfies $\moment_\zero = 1$, $\MM_2\parentheses{\moments_4} \succeq 0$, $\MM_1\parentheses{h \moments_2} = \zero,\forall h \in \vh$, and $\MM_1\parentheses{g \moments_2} \succeq 0,\forall g \in \vg$. Then $\moments_{2\rbasisset} \subset \moments_4$, the sub-vector of $\moments_4$ corresponding to the sparse set of monomials $\monoleq{\xextend}{2\rbasisset}$, must be feasible for the sparse relaxation~\eqref{eq:sparseMoment}. This is because $\MM_\rbasisset\parentheses{\moments_{2\rbasisset}} \succeq 0$ must hold as $\MM_\rbasisset\parentheses{\moments_{2\rbasisset}}$ is a principal sub-matrix of the full moment matrix $\MM_2\parentheses{\moments_4}$; $\MM_{\rbasisset_x}\parentheses{h \moments_{2\rbasisset_x}} = \zero$ must hold as $\MM_{\rbasisset_x}\parentheses{h \moments_{2\rbasisset_x}}$ is a principal sub-matrix of the full localizing matrix $\MM_1\parentheses{h \moments_2}, \forall h \in \vh^\theta$; and $\MM_1\parentheses{h\moments_2}$ and $\MM_1\parentheses{g\moments_2}$ are the same as the localizing matrices in the dense relaxation.

{\bf Rounding and Relative Duality Gap}. Property~\ref{prop:dense-gap} of the dense relaxation still holds for the sparse relaxation. Because $p^\star_{\rbasisset} \leq p^\star_2$ and $p^\star_2 \leq f^\star$, we have $p^\star_{\rbasisset} \leq f^\star$ is also a valid lower bound for $f^\star$, the true optimal objective value of the POP. Additionally, since the sparse set of monomials $\monoleq{\xextend}{\rbasisset}$ still contains $\monoleq{\xextend}{1}$, the degree-1 monomials of $\vxx$ and $\vtheta$, one can use the same rounding method (\ie~spectral decomposition and the projection onto the feasible set $\calxextend$ in~\eqref{eq:projection}) to obtain a feasible solution $\hatxextend$, which gives a value $\hatf = f(\hatxextend)$ that satisfies $p^\star_{\rbasisset} \leq f^\star \leq \hatf$. The relative duality gap can then be calculated similar to eq.~\eqref{eq:relDualityGapDense} as:
\bea
\eta_\rbasisset = \frac{\hatf - p^\star_\rbasisset}{\hatf},
\eea
where a smaller $\eta_\rbasisset$ implies a tighter relaxation and $\eta_\rbasisset = 0$ if and only if the relaxation is tight.

{\bf Optimality Certificate}. Showing property~\ref{prop:dense-certificate} of the dense relaxation also holds for the sparse relaxation is non-trivial because Theorem~\ref{thm:sufficientTruncatedK} does not hold for an arbitrary sparse moment matrix (\ie~a moment matrix with rows and columns indexed by a sparse set of monomials $\monoleq{\xextend}{\rbasisset} \subset \monoleq{\xextend}{2}$). Therefore, we will show that the sparse moment matrix $\MM_\rbasisset\parentheses{\moments_{2\rbasisset}}$ can be \emph{extended} to a full moment matrix $\MM_2\parentheses{\moments_4}$ when $\rank{\MM_\rbasisset\parentheses{\moments_{2\rbasisset}}} = 1$. Let us first introduce the notion of a flat extension.

\begin{definition}[\bf Flat Extension]
\label{def:flatExtension}
Given two moment matrices $\MM_{\rbasisset}\parentheses{\moments_{2\rbasisset}}$ and $\MM_{\calA}\parentheses{\moments_{2\calA}}$, with $\rbasisset \subset \calA$ (recall that the rows and columns of $\MM_{\rbasisset}\parentheses{\moments_{2\rbasisset}}$ (resp. $\MM_{\calA}\parentheses{\moments_{2\calA}}$) are indexed by monomials $\monoleq{\xextend}{\rbasisset}$ (resp. $\monoleq{\xextend}{\calA}$)), then $\MM_{\calA}\parentheses{\moments_{2\calA}}$ is said to be a flat extension of $\MM_{\rbasisset}\parentheses{\moments_{2\rbasisset}}$ if $\MM_{\rbasisset}\parentheses{\moments_{2\rbasisset}}$ coincides with the sub-matrix of $\MM_{\calA}\parentheses{\moments_{2\calA}}$ indexed by $\monoleq{\xextend}{\rbasisset}$ and $\rank{\MM_{\rbasisset}\parentheses{\moments_{2\rbasisset}}} = \rank{\MM_{\calA}\parentheses{\moments_{2\calA}}}$.
\end{definition}

Our goal is to show that $\MM_{\rbasisset}\parentheses{\moments_{2\rbasisset}}$ admits a flat extension $\MM_2\parentheses{\moments_4}$ when $\rank{\MM_{\rbasisset}\parentheses{\moments_{2\rbasisset}}} = 1$, so that $\rank{\MM_2\parentheses{\moments_4}}=1$ is also true, in which case we recover the dense moment relaxation and obtain an optimality certificate. To do so, we will show that the sparse moment matrix $\MM_{\rbasisset}\parentheses{\moments_{2\rbasisset}}$ satisfies the \emph{generalized flat extension theorem} in~\cite{Laurent09AdM-generalFlatExtension}, stated below.

\begin{theorem}[\bf Generalized Flat Extension~{\cite[Theorem 1.4]{Laurent09AdM-generalFlatExtension}}]
\label{thm:generalFlatExtension}
Given a monomial basis $\monoleq{\xextend}{\calC}$, define its \emph{closure} to be the set:
\bea
\monoleq{\xextend}{\calC^+} \doteq \monoleq{\xextend}{\calC} \cup \left(  \cup_{i=1}^\dimxextend p_i \monoleq{\xextend}{\calC} \right)  \doteq \left\{ \xextend^\valpha,p_1\xextend^\valpha,\dots,p_{\dimxextend}\xextend^\valpha | \valpha \in \calC \right\}.
\eea
For example, let $\dimxextend = 3$, and $\monoleq{\xextend}{\calC} = [p_1]$, then $\monoleq{\xextend}{\calC^+} = [p_1,p_1^2,p_1p_2,p_1p_3]$. In addition, the monomial set $\monoleq{\xextend}{\calC}$ is said to be \emph{connected to 1} if $1 \in \monoleq{\xextend}{\calC}$ and every monomial $\xextend^\valpha$ can be written as $\xextend^\valpha = p_{i_1}p_{i_2}\cdots p_{i_k}$ with $p_{i_1},p_{i_1}p_{i_2},...,p_{i_1}p_{i_2}\cdots p_{i_k} \in \monoleq{\xextend}{\calC}$. For example $[1,p_1,p_1p_2]$ is connected to 1, but $[1,p_1p_2]$ is not. Then the generalized flat extension theorem states:

If $\monoleq{\xextend}{\calC}$ is connected to 1, and $\MM_{\calC^+}\parentheses{\moments_{2\calC^+}}$ is a flat extension of $\MM_{\calC}\parentheses{\moments_{2\calC}}$, then $\MM_{\calC^+}\parentheses{\moments_{2\calC^+}}$ admits a unique flat extension $\MM_\rorder\parentheses{\moments_{2\rorder}}$ for any $\rorder \geq \alpha_\max$, where $\alpha_\max \doteq \max\{|\valpha|: \valpha \in \calC^+\}$ denotes the maximum degree of the monomials in $\monoleq{\xextend}{\calC^+}$. 
\end{theorem}

Using Theorem~\ref{thm:generalFlatExtension}, let $\monoleq{\xextend}{\calC} \doteq \monoleq{\xextend}{\rbasisset_x} = [1,\vxx\tran]\tran$, then obviously $\monoleq{\xextend}{\calC}$ is connected to 1. The closure of $\monoleq{\xextend}{\calC}$ is $\monoleq{\xextend}{\calC^+} = \monoleq{\xextend}{\rbasisset} = [1,\vxx\tran,\vtheta\tran,\monoindeg{\vxx}{2},\vtheta\tran \kron \vxx\tran]
\tran$. If $\rank{\MM_\rbasisset\parentheses{\moments_{2\rbasisset}}} = 1$, then $\rank{\MM_\rbasisset\parentheses{\moments_{2\rbasisset}}} = \rank{\MM_{\rbasisset_x}\parentheses{\moments_{2\rbasisset_x}}} = 1$ and $\MM_\rbasisset\parentheses{\moments_{2\rbasisset}}$ is a flat extension of $\MM_{\rbasisset_x}\parentheses{\moments_{2\rbasisset_x}}$. Therefore, $\MM_\rbasisset\parentheses{\moments_{2\rbasisset}}$ admits a flat extension $\MM_\rorder\parentheses{\moments_{2\rorder}}$ for any $\rorder \geq 2$. In particular, setting $\rorder = 2$, we recover a dense moment matrix $\MM_2\parentheses{\moments_4}$ with $\rank{\MM_2\parentheses{\moments_4}}=1$. It remains to show that the moments $\moments_4$ (obtained from the flat extension) satisfy the constraints on the localizing matrices in the dense relaxation~\eqref{eq:denseMoment}. The only different constraint between the dense relaxation~\eqref{eq:denseMoment} at $\rorder=2$ and the sparse relaxation~\eqref{eq:sparseMoment} is that, the constraint $\MM_1\parentheses{h \moments_2} = \zero$ has been relaxed to $\MM_1\parentheses{h \moments_{2\rbasisset_x}} = \zero$ for $h \in \vh^\theta$. However, we observe that the top-left entry of $\MM_1\parentheses{h \moments_{2\rbasisset_x}}$ is $\moment_1 - \moment_{\theta_i^2}$ when $h = h^{\theta_i} = 1-\theta_i^2$, and $\moment_1 - \moment_{\theta_i^2} = 0$ implies $\moment_{\theta_i^2} = \moment_{\theta_i}^2 = 1$ (due to $\rank{\MM_2\parentheses{\moments_4}} =1 $), which implies $\moment_{\theta_i} = \pm 1$ and the solution is indeed binary and supported on $\calxextend$ and must satisfy $\MM_1\parentheses{h \moments_2} = \zero$.
\end{proof}

%% file: supp-proof-sufficientOptimality.tex
\section{Proof of Theorem~\ref{thm:sufficientOptimalityCondition} (Sufficient Condition for Global Optimality)}
\label{sec:supp-proof-sufficientOptimality}
\begin{proof}
If problem~\eqref{eq:sosfeasiblity} is feasible, then for any $\xextend \in \calxextend$, we have:
\bea
\monoleq{\xextend}{\rbasisset}\tran \MS_0 \monoleq{\xextend}{\rbasisset} \geq 0, \quad \monoleq{\xextend}{1}\tran \MS_k \monoleq{\xextend}{1} \geq 0, \forall k = 1,\dots,\nrIneq,
\eea
because $\MS_0$ and $\MS_k,k=1,\dots,\nrIneq,$ are PSD matrices (\ie~$\monoleq{\xextend}{\rbasisset}\tran \MS_0 \monoleq{\xextend}{\rbasisset}$ and $ \monoleq{\xextend}{1}\tran \MS_k \monoleq{\xextend}{1}$ are SOS polynomials). In addition, $g_k(\xextend) \geq 0, \forall \xextend \in \calxextend$ by construction of the inequality constraints of the POP~\eqref{eq:pop}. Therefore, the right-hand side of eq.~\eqref{eq:matchingcoefficients} is nonnegative. On the other hand, the left-hand side of eq.~\eqref{eq:matchingcoefficients} reduces to $f(\xextend) - \hatf$ for any $\xextend \in \calxextend$ due to the equality constraints $h_j(\xextend) = 0,\forall h_j \in \vh^x$ and $h_j(\xextend) = 0,\forall h_j \in \vh^\theta$ of the POP~\eqref{eq:pop}. Combining these two observations, we have $f(\xextend) - \hatf \geq 0$ for any $\xextend \in \calxextend$, which implies that $\hatf$ is the global minimum of the POP and $\hatxextend$ is the corresponding global minimizer.

Next we show how to convert problem~\eqref{eq:sosfeasiblity} into the compact SDP formulation in~\eqref{eq:sdpfeasibility}, by writing every polynomial in~\eqref{eq:matchingcoefficients} as a sum of products between coefficients (parametrized by the unknowns $\vlambda_j^x,\vlambda_j^\theta,\MS_0$ and $\MS_k$) and the monomial basis $\monoleq{\xextend}{2\rbasisset}$. 

{\bf Right-hand Side of~\eqref{eq:matchingcoefficients}}. We start from the right-hand side of~\eqref{eq:matchingcoefficients}. To do so, we first write the monomial outer product $\monoleq{\xextend}{\rbasisset}\monoleq{\xextend}{\rbasisset}\tran$ as:
\bea
\monoleq{\xextend}{\rbasisset}\monoleq{\xextend}{\rbasisset}\tran = \sum_{\valpha \in 2\rbasisset} \monoindicator^0 \xextend^{\valpha}, \label{eq:monoouterproduct}
\eea
where $\monoindicator^0 \in \sym^{\dimbasis{}{\rbasisset}}$ is the ``$0/1$'' \emph{monomial indicator matrix} with rows and columns indexed by $\monoleq{\xextend}{\rbasisset}$, whose entries are defined as:
\bea
\bracket{\monoindicator^0}_{\xextend^{\valpha_1},\xextend^{\valpha_2}} = \begin{cases}
1 & \text{if } \valpha_1 + \valpha_2 = \valpha \\
0 & \text{otherwise}
\end{cases}. \label{eq:monoindicators0}
\eea
Using the expression in~\eqref{eq:monoouterproduct}, we can write the SOS polynomial $s_0 \doteq \monoleq{\xextend}{\rbasisset}\tran \MS_0 \monoleq{\xextend}{\rbasisset}$ as:
\bea
s_0 \doteq \monoleq{\xextend}{\rbasisset}\tran \MS_0 \monoleq{\xextend}{\rbasisset} = \inner{\MS_0}{\monoleq{\xextend}{\rbasisset}\monoleq{\xextend}{\rbasisset}\tran} = \inner{\MS_0}{\sum_{\valpha \in 2\rbasisset} \monoindicator^0 \xextend^{\valpha}} = \sum_{\valpha \in 2\rbasisset} \inner{\MS_0}{\monoindicator^0} \xextend^\valpha, \label{eq:coeffBasiss0}
\eea
where $\inner{\MA}{\MB} \doteq \trace{\MA \MB}$ denotes the inner product between two symmetric matrices $\MA,\MB \in \sym^n$. Similarly, for each outer product $g_k \monoleq{\xextend}{1}\monoleq{\xextend}{1}\tran$, we write them as:
\bea
g_k \monoleq{\xextend}{1}\monoleq{\xextend}{1}\tran = \sum_{\valpha \in 2\rbasisset} \monoindicator^k \xextend^\valpha, k=1,\dots,\nrIneq, \label{eq:monogkouterproduct}
\eea
where $\monoindicator^k \in \sym^{\dimbasis{\dimxextend}{1}}$ is the \emph{monomial coefficient matrix} with rows and columns indexed by $\monoleq{\xextend}{1}$, whose entries are defined as:
\bea
\bracket{\monoindicator^k}_{\xextend^{\valpha_1},\xextend^{\valpha_2}} = c_k\parentheses{\valpha - \valpha_1 - \valpha_2},
\eea
where $c_k\parentheses{\valpha - \valpha_1 - \valpha_2}$ denotes the coefficient of $g_k$ corresponding to the monomial $\xextend^{\valpha - \valpha_1 - \valpha_2}$. Note that $\monoindicator^k$ is not an ``$0/1$'' matrix due to the multiplication of $g_k$ with the monomial outer product $\monoleq{\xextend}{1}\monoleq{\xextend}{1}\tran$. Using the expression in~\eqref{eq:monogkouterproduct}, we can write the nonnegative polynomials $s_k \doteq g_k \monoleq{\xextend}{1}\tran \MS_k \monoleq{\xextend}{1},k=1,\dots,\nrIneq$, as:
\bea
s_k \doteq g_k \monoleq{\xextend}{1}\tran \MS_k \monoleq{\xextend}{1} = \inner{\MS_k}{g_k \monoleq{\xextend}{1}\monoleq{\xextend}{1}\tran} = \inner{\MS_k}{\sum_{\valpha \in 2\rbasisset} \monoindicator^k \xextend^\valpha} = \sum_{\valpha \in 2\rbasisset}\inner{\MS_k}{\monoindicator^k}\xextend^\valpha. \label{eq:coeffBasissk}
\eea
Eq.~\eqref{eq:coeffBasiss0} and~\eqref{eq:coeffBasissk} have written the right-hand side of~\eqref{eq:matchingcoefficients} as a sum of products, where each product is between a monomial $\xextend^\valpha$ and a coefficient, $\inner{\MS_0}{\monoindicator^0}$ or $\inner{\MS_k}{\monoindicator^k}$, parametrized by the unknown PSD matrices $\MS_0$ and $\MS_k,k=1,\dots,\nrIneq$.

{\bf Left-hand Side of~\eqref{eq:matchingcoefficients}}. We now perform similar algebra for the left-hand side of~\eqref{eq:matchingcoefficients}. We write $f(\xextend) - \hatf$ as:
\bea
f(\xextend) - \hatf = \sum_{\valpha \in 2\rbasisset} c_f(\valpha) \xextend^\valpha, \label{eq:LHSf}
\eea 
where $c_f(\valpha)$ denotes the coefficient of $f(\xextend) - \hatf$ corresponding to the monomial $\xextend^\valpha$. We write $h_j \monoleq{\xextend}{2}, h_j \in \vh^x$ as:
\bea
h_j \monoleq{\xextend}{2} = \sum_{\valpha \in 2\rbasisset} \monoindicatorVec^{x_j} \xextend^\valpha, \label{eq:monoseparationhx}
\eea
where $\monoindicatorVec^{x_j} \in \Real{\dimbasis{\dimxextend}{2}}$ is a vector of coefficients indexed by $\monoleq{\xextend}{2}$, whose entries are defined as:
\bea
\bracket{\monoindicatorVec^{x_j}}_{\xextend^{\valpha_1}} = c_{h_j^x}\parentheses{\valpha - \valpha_1},
\eea
where $c_{h_j^x}\parentheses{\valpha - \valpha_1}$ is the coefficient of $h_j \in \vh^x$ corresponding to the monomial $\xextend^{\valpha - \valpha_1}$. Using the expression in~\eqref{eq:monoseparationhx}, we can write $q_j^x \doteq h_j \monoleq{\xextend}{2}\tran \vlambda_j^x, \forall h_j \in \vh^x$, as:
\bea
q_j^x \doteq h_j \monoleq{\xextend}{2}\tran \vlambda_j^x = \inner{\vlambda_j^x}{h_j \monoleq{\xextend}{2}} = \inner{\vlambda_j^x}{\sum_{\valpha \in 2\rbasisset} \monoindicatorVec^{x_j} \xextend^\valpha} = \sum_{\valpha \in 2\rbasisset} \inner{\vlambda_j^x}{\monoindicatorVec^{x_j}} \xextend^\valpha, \label{eq:LHSqjx}
\eea
where $\inner{\va}{\vb} \doteq \va\tran \vb$ denotes the inner product between two vectors $\va,\vb \in \Real{n}$. Similarly, we write $h_j \monoleq{\vxx}{2}, h_j \in \vh^\theta$ as:
\bea
h_j \monoleq{\vxx}{2} = \sum_{\valpha \in 2\rbasisset} \monoindicatorVec^{\theta_j} \xextend^\valpha, \label{eq:monoseparationhtheta}
\eea
where $\monoindicatorVec^{\theta_j} \in \Real{\dimbasis{n}{2}}$ is a vector of coefficients indexed by $\monoleq{\vxx}{2}$, whose entries are defined as:
\bea
\bracket{\monoindicatorVec^{\theta_j}}_{\xextend^{\valpha_1}} = c_{h_j^\theta}\parentheses{\valpha - \valpha_1},
\eea
where $c_{h_j^\theta}\parentheses{\valpha - \valpha_1}$ is the coefficient of $h_j \in \vh^\theta$ corresponding to the monomial $\xextend^{\valpha - \valpha_1}$. Using the expression in~\eqref{eq:monoseparationhtheta}, we can write $q_j^\theta \doteq h_j \monoleq{\vxx}{2}\tran \vlambda_j^\theta$ as:
\bea
q_j^\theta \doteq h_j \monoleq{\vxx}{2}\tran \vlambda_j^\theta = \inner{\vlambda_j^\theta}{h_j \monoleq{\vxx}{2}} = \inner{\vlambda_j^\theta}{\sum_{\valpha \in 2\rbasisset} \monoindicatorVec^{\theta_j} \xextend^\valpha} = \sum_{\valpha \in 2\rbasisset} \inner{\vlambda_j^\theta}{\monoindicatorVec^{\theta_j}} \xextend^\valpha. \label{eq:LHSqjtheta}
\eea
Eq.~\eqref{eq:LHSqjx} and~\eqref{eq:LHSqjtheta} have written the left-hand side of~\eqref{eq:matchingcoefficients} as a sum of products, where each product is between a monomial $\xextend^\valpha$ and a coefficient, $\inner{\vlambda_j^x}{\monoindicatorVec^{x_j}}$ or $\inner{\vlambda_j^\theta}{\monoindicatorVec^{\theta_j}}$, parametrized by the unknown vectors $\vlambda_j^x, j = 1,\dots,|\vh^x|$, and $\vlambda_j^\theta, j=1,\dots,|\vh^\theta|$, where $|\vh^x|$ and $|\vh^\theta|$ denotes the cardinality of the sets $\vh^x$ and $\vh^\theta$, respectively.

{\bf Obtaining the Compact SDP~\eqref{eq:sdpfeasibility}}. Combining the left-hand side (eq.~\eqref{eq:LHSf},~\eqref{eq:LHSqjx} and~\eqref{eq:LHSqjtheta}) and the right-hand side (eq.~\eqref{eq:coeffBasiss0} and~\eqref{eq:coeffBasissk}) of eq.~\eqref{eq:matchingcoefficients}, we are ready to write down the final expression for the compact SDP~\eqref{eq:sdpfeasibility}. To do so, we first concatenate all the independent unknown variables into a single vector, called the dual variable: 
\bea
\hspace{-12mm} \dualVar = [(\vlambda_1^x)\tran,\dots,(\vlambda^x_{|\vh^x|})\tran,(\vlambda_1^\theta)\tran,\dots,(\vlambda_{|\vh^\theta|}^\theta)\tran,\svec{\MS_1}\tran,\dots,\svec{\MS_\nrIneq}\tran,\svec{\MS_0}\tran]\tran \in \Real{m_d}, \label{eq:dualVar} 
\eea
whose dimension is:
\bea
m_d = \underbrace{|\vh^x| \cdot \dimbasis{\dimxextend}{2} + |\vh^\theta| \cdot \dimbasis{n}{2}}_{m_{d_1}} + \underbrace{\nrIneq \cdot \frac{\dimbasis{\dimxextend}{1}\parentheses{\dimbasis{\dimxextend}{1}+1}}{2}}_{m_{d_2}} + \underbrace{\frac{\dimbasis{}{\rbasisset} \parentheses{\dimbasis{}{\rbasisset}+1}}{2}}_{m_{d_3}}, \label{eq:dualVarDim}
\eea
where $m_{d_1}$ is the dimension of the free variables $\vlambda^x$ and $\vlambda^\theta$, $m_{d_2}$ is the dimension of the PSD variables $\MS_k,k=1,\dots,\nrIneq$, $m_{d_3}$ is the dimension of the PSD variable $\MS_0$, and we use symmetric vectorization to save storage. Then it is obvious that the dual variable $\dualVar$ lives in a convex cone $\calK$ defined by:
\bea
\calK \doteq \Real{m_{d_1}} \times \underbrace{\psd^{\dimbasis{\dimxextend}{1}} \times \dots \times \psd^{\dimbasis{\dimxextend}{1}}}_{\nrIneq} \times \psd^{\dimbasis{}{\rbasisset}}. \label{eq:dualCone}
\eea
Additionally, the dual variable $\dualVar$ must satisfy the equality constraint in~\eqref{eq:matchingcoefficients}:
\bea
f(\xextend) - \hatf - \sum_{j = 1}^{|\vh^x|} q_j^x -  \sum_{j = 1}^{|\vh^\theta|} q_j^\theta = s_0 + \sum_{k=1}^\nrIneq s_k, \forall \xextend.
\eea 
Now using the expressions in eq.~\eqref{eq:LHSf},~\eqref{eq:LHSqjx},~\eqref{eq:LHSqjtheta},~\eqref{eq:coeffBasiss0}, and~\eqref{eq:coeffBasissk}, we obtain the following linear constraints for each monomial $\xextend^\valpha$:
\small
\bea
\hspace{-6mm} c_f(\valpha)  = \sum_{j=1}^{|\vh^x|} \inner{\vlambda_j^x}{\monoindicatorVec^{x_j}} + \sum_{j=1}^{|\vh^\theta|} \inner{\vlambda_j^\theta}{\monoindicatorVec^{\theta_j}} + \sum_{k=1}^\nrIneq \inner{\svec{\MS_k}}{\svec{\monoindicator^k}} + \inner{\svec{\MS_0}}{\svec{\monoindicator^0}}, \label{eq:affineSpaceeach}
\eea
\normalsize
where we have used the fact that $\inner{\MA}{\MB} = \inner{\svec{\MA}}{\svec{\MB}}$ for any two symmetric matrices $\MA,\MB \in \sym^n$. The linear constraint~\eqref{eq:affineSpaceeach} can be written compactly as:
\bea
{\va_{\valpha}\tran} \dualVar = c_f(\valpha),
\eea
where $\va_{\valpha} \in \Real{m_d}$ is a vector of constants that is only related to the equality and inequality constraints $h_j$ and $g_k$ of the POP~\eqref{eq:pop}:
\smaller
\bea
\hspace{-6mm} \va_{\valpha} = \bracket{\parentheses{\monoindicatorVec^{x_1}}\tran,\dots,\parentheses{\monoindicatorVec^{x_{|\vh^x|}}}\tran,\parentheses{\monoindicatorVec^{\theta_1}}\tran,\dots,\parentheses{\monoindicatorVec^{\theta_{|\vh^\theta|}}}\tran,\svec{\monoindicator^1}\tran,\dots,\svec{\monoindicator^\nrIneq}\tran,\svec{\monoindicator^0}\tran }\tran. \label{eq:rowofA}
\eea
\normalsize
All the linear constraints, one for each $\xextend^\valpha, \valpha \in 2\rbasisset$, assembled together, define an affine subspace:
\bea
\calA \doteq \cbrace{\dualVar: \underbrace{\bmat{c} \vdots \\ \va_{\valpha}\tran \\ \vdots \emat}_{\MA \in \Real{\dimbasis{}{2\rbasisset}\times m_d} } \dualVar = \underbrace{ \bmat{c} \vdots \\ c_f(\valpha) \\ \vdots \emat}_{\vb \in \Real{\dimbasis{}{2\rbasisset}}}  }. \label{eq:dualAffine}
\eea
Therefore, problem~\eqref{eq:sosfeasiblity} is equivalent to:
\bea
\text{find } \dualVar \in \Real{m_d},\quad \subject \quad \dualVar \in \calK \cap \calA,
\eea
with the convex cone $\calK$ defined in~\eqref{eq:dualCone} and the affine subspace defined in~\eqref{eq:dualAffine}.

{\bf Partial Orthogonality}. Finally, we state a property, namely \emph{partial orthogonality}~\cite{Zheng18TAC-partialOrthogonality}, of the matrix $\MA \in \Real{\dimbasis{}{2\rbasisset}\times m_d}$ that defines the affine subspace $\calA$ in~\eqref{eq:dualAffine}.
\begin{theorem}[\bf Partial Orthogonality of $\MA$]
\label{thm:partialorthogonality}
Let $\MA = [\MA_1,\MA_2,\MA_3]$ be the column-wise partition of $\MA$ according to the dimension defined in~\eqref{eq:dualVarDim},~\ie~$\MA_1 \in \Real{\dimbasis{}{2\rbasisset}\times m_{d_1}}$, $\MA_2 \in \Real{\dimbasis{}{2\rbasisset}\times m_{d_2}}$ and $\MA_3 \in \Real{\dimbasis{}{2\rbasisset}\times m_{d_3}}$, then $\MA_3 \MA_3\tran$ is an invertible diagonal matrix.
\end{theorem}
\begin{proof}
From the partition, we know that $\MA_3$ corresponds to the columns of $\MA$ indexed by $\svec{\MS_0}$. Therefore, according to~\eqref{eq:rowofA}, which shows the row of $\MA$ corresponding to a monomial $\xextend^\valpha$, we can write $\MA_3$ as:
\bea
\MA_3 = \bmat{c}
\vdots \\
\svec{\monoindicator^0}\tran \\
\vdots
\emat.
\eea 
Now we can compute the $(i,j)$-th entry of $\MA_3\MA_3\tran$ for $i \neq j$:
\bea
\bracket{\MA_3 \MA_3\tran}_{i,j} = \svec{\monoindicatorr_{\valpha_i}^0}\tran \svec{\monoindicatorr_{\valpha_j}^0} = \inner{\monoindicatorr^0_{\valpha_i}}{\monoindicatorr^0_{\valpha_j}} = 0,
\eea
where $\inner{\monoindicatorr^0_{\valpha_i}}{\monoindicatorr^0_{\valpha_j}} = 0$ holds due to the definition of the indicator matrix in~\eqref{eq:monoindicators0} (if $\valpha_1 + \valpha_2 = \valpha_i$, then $\valpha_1 + \valpha_2 \neq \valpha_j$ when $\valpha_i \neq \valpha_j$). The diagonal entries of $\MA_3 \MA_3\tran$ are nonzero because:
\bea
\bracket{\MA_3 \MA_3\tran}_{i,i} = \svec{\monoindicatorr_{\valpha_i}^0}\tran \svec{\monoindicatorr_{\valpha_i}^0} = \inner{\monoindicatorr^0_{\valpha_i}}{\monoindicatorr^0_{\valpha_i}} \geq 1.
\eea 
Since $\bracket{\MA_3 \MA_3\tran}_{i,j} = 0$ for any $i \neq j$, and $\bracket{\MA_3 \MA_3\tran}_{i,i} \geq 1$, $\MA_3\MA_3\tran$ is diagonal and invertible.
\end{proof}
In Section~\ref{sec:supp-proof-DRS}, we will see the partial orthogonality property of $\MA$ allows efficient computation of the projection map onto the affine subspace $\calA$.
\end{proof}

%% file: supp-proof-DRS.tex
\section{Proof of Theorem~\ref{thm:DRSOptimalityCertification} (\DRS for Optimality Certification)}
\label{sec:supp-proof-DRS}
\begin{proof}
We will first prove that \DRS iterates converge to a solution of the feasibility SDP~\eqref{eq:sdpfeasibility} if it is feasible. We will then show how to compute a suboptimality bound from each iteration of the \DRS update. Finally, we will discuss how to implement the projection maps in the \DRS iterates. 

{\bf Convergence}. We first prove (i),~\ie~the \DRS iterates~\eqref{eq:DRSIterates}, with $0 < \gamma_\tau < 2$, converge to a solution of the SDP~\eqref{eq:sdpfeasibility} if it is feasible. To do so, we write the SDP~\eqref{eq:sdpfeasibility} equivalently as:
\bea \label{eq:sdpIndicator}
\min_{\dualVar} \quad \indicator_{\calK} \parentheses{\dualVar} + \indicator_{\calA} \parentheses{\dualVar},
\eea 
where $\indicator_{\calK} \parentheses{\dualVar}$ (resp. $\indicator_{\calA} \parentheses{\dualVar}$) is the indicator function of the set $\calK$ (resp. the set $\calA$),~\ie~$\indicator_\calK\parentheses{\dualVar} = 0$ if $\dualVar \in \calK$ and $\indicator_\calK\parentheses{\dualVar} = \infty$ if $\dualVar \not\in \calK$. It is clear that if the SDP~\eqref{eq:sdpfeasibility} is feasible, then the optimal cost of problem~\eqref{eq:sdpIndicator} is $0$; while if the SDP~\eqref{eq:sdpfeasibility} is infeasible, then the optimal cost of problem~\eqref{eq:sdpIndicator} is $\infty$. Douglas-Rachford Splitting is designed to solve problems of the following type:
\bea
\min_{\dualVar} \quad f(\dualVar) + g(\dualVar),
\eea 
where $f$ and $g$ are convex functions of $\dualVar$. Now let $f = \indicator_{\calK} \parentheses{\dualVar}$, $g = \indicator_{\calA} \parentheses{\dualVar}$, and observe that the proximal operator for an indicator function $\indicator_\calK$ is exactly the projection onto the set $\calK$,\footnote{The proximal operator of a function $f$ is defined as: $\prox_f\parentheses{\vxx_0} \doteq \argmin_{\vxx} \frac{1}{2} \twonorm{\vxx - \vxx_0}^2 + f(\vxx)$. When $f = \indicator_{\calK}$ is an indicator function, then $\prox_f\parentheses{\vxx_0} = \argmin_{\vxx} \frac{1}{2} \twonorm{\vxx - \vxx_0}^2 + \indicator_\calK(\vxx) := \proj_\calK \parentheses{\vxx_0}$. } we obtain the \DRS iterates of~\eqref{eq:DRSIterates} from~\cite[Algorithm 4.2]{Combettes11book-proximalSplitting}. In addition,~\cite[Proposition 4.3]{Combettes11book-proximalSplitting} tells us the \DRS iterates converge to a solution of~\eqref{eq:sdpIndicator}. This implies the sequence $\cbrace{\dualVar_\tau}_{\tau \geq 0}$ converges to a point inside $\calK \cap \calA$ when the intersection is nonempty.\footnote{In fact, more generally, even if the intersection is empty, it is known that, if $\gamma_\tau = 1$, then the sequences $\cbrace{\dualVar_\tau^\calK}_{\tau \geq 0}$ and $\cbrace{\dualVar_\tau^\calA}_{\tau \geq 0}$ converge to a solution of the optimization: $\min_{\dualVar_1 \in \calK, \dualVar_2 \in \calA} \twonorm{\dualVar_1 - \dualVar_2}$,~\ie~a pair of points $\dualVar_1 \in \calK$ and $\dualVar_2 \in \calA$ that achieves the minimum distance between set $\calK$ and set $\calA$~\cite{Jegelka13NIPS-DRSreflection,Bauschke04JAT-averagedAlternatingReflections}.} 

{\bf Suboptimality Bound}. We then prove (ii),~\ie~the \DRS iterates provide valid suboptimality bounds $\suboptbound_\tau$ at each iteration. In particular, this suboptimality bound can be computed from $\dualVar_\tau^\calA$. To show this, we note that any $\dualVar_\tau^\calA$ satisfies the equality constraint in~\eqref{eq:matchingcoefficients} because $\dualVar_\tau^\calA \in \calA$. Therefore, let $\svec{\MS_0^\tau}$ and $\svec{\MS_k^\tau},k=1,\dots,\nrIneq$ be the sub-vectors in $\dualVar_\tau^\calA$, then for any $\xextend \in \calxextend$, eq.~\eqref{eq:matchingcoefficients} tells us:
\bea
\hspace{-4mm} f\parentheses{\xextend} - \hatf = \monoleq{\xextend}{\rbasisset}\tran \MS_0^\tau \monoleq{\xextend}{\rbasisset}\tran + \sum_{k=1}^\nrIneq g_k \monoleq{\xextend}{1}\tran \MS_k^\tau \monoleq{\xextend}{1} \geq \mineig{\MS_0^\tau} M_0^2 + \sum_{k=1}^\nrIneq \min\cbrace{0,\mineig{\MS_k^\tau}} M_1^2, \label{eq:absSuboptBound}
\eea
where $M_0$ and $M_1$ are upper bounds on the $\ell_2$-norm of the monomial bases $\monoleq{\xextend}{\rbasisset}$ and $\monoleq{\xextend}{1}$:
\bea
\twonorm{\monoleq{\xextend}{\rbasisset}} \leq M_0,\quad \twonorm{\monoleq{\xextend}{1}} \leq M_1, \quad \forall \xextend \in \calxextend.
\eea
In~\eqref{eq:absSuboptBound}, we have used the fact that $g_k(\xextend) \leq 1$ for any $\xextend \in \calxextend$ from the POP~\eqref{eq:pop}. Now to obtain the suboptimality bound $\suboptbound_\tau$, let $\xextend = \xextend^\star$ be the global minimizer in~\eqref{eq:absSuboptBound}, we have $f(\xextend^\star) = f^\star$ and:
\bea
f^\star - \hatf \geq \mineig{\MS_0^\tau} M_0^2 + \sum_{k=1}^\nrIneq \min\cbrace{0,\mineig{\MS_k^\tau}} M_1^2 \\
\Longrightarrow \frac{\hatf - f^\star}{\hatf} \leq \frac{-\mineig{\MS_0^\tau} M_0^2 - \sum_{k=1}^\nrIneq \min\cbrace{0,\mineig{\MS_k^\tau}} M_1^2}{\hatf} := \suboptbound_\tau.
\eea
We now give the expressions for the upper bounds $M_0$ and $M_1$ for Examples~\ref{eg:singleRotationAveraging}-\ref{eg:meshRegistration}.

{\bf Example~\ref{eg:singleRotationAveraging} (Single Rotation Averaging)}. Recall $\vxx = \vr = \vectorize{\MR}$ with $\twonorm{\vr}^2 = 3$, so:
\bea
\twonorm{\monoleq{\xextend}{1}}^2 = 1 + \twonorm{\vr}^2 + \twonorm{\vtheta}^2 = 4+N := M_1^2, \\
\twonorm{\monoleq{\xextend}{\rbasisset}}^2 = 1+ \twonorm{\vr}^2 + \twonorm{\vtheta}^2 + \twonorm{\monoindeg{\vr}{2}}^2 + \twonorm{\vtheta \kron \vr}^2  = 4N+13 := M_0^2.
\eea

{\bf Example~\ref{eg:shapeAlignment} (Shape Alignment)}. Recall $\vxx = \SAr = \vectorize{s\Pi\MR}$ with $\twonorm{\SAr}^2 \leq 2\sub^2$, so:
\bea
\twonorm{\monoleq{\xextend}{1}}^2 = 1 + \twonorm{\SAr}^2 + \twonorm{\vtheta}^2 \leq 1+2\sub^2 +N := M_1^2, \\
\hspace{-5mm} \twonorm{\monoleq{\xextend}{\rbasisset}}^2 = 1+ \twonorm{\SAr}^2 + \twonorm{\vtheta}^2 + \twonorm{\monoindeg{\SAr}{2}}^2 + \twonorm{\vtheta \kron \vr}^2  \leq (1+2\sub^2)(1+N) + 4\sub^4 := M_0^2.
\eea

{\bf Example~\ref{eg:pointCloudRegistration} (Point Cloud Registration)}. Recall $\vxx = [\vr\tran,\vt\tran]\tran = [\vectorize{\MR}\tran,\vt\tran]\tran$ with $\twonorm{\vr}^2 = 3,\twonorm{\vt}^2 \leq T^2$, so:
\bea
\twonorm{\monoleq{\xextend}{1}}^2 = 1+\twonorm{\vr}^2 + \twonorm{\vt}^2 + \twonorm{\vtheta}^2 \leq 4+T^2+N := M_1^2, \\
\hspace{-8mm} \twonorm{\monoleq{\xextend}{\rbasisset}}^2 =
1+\twonorm{\vr}^2 + \twonorm{\vt}^2 + \twonorm{\vtheta}^2 + \twonorm{\monoindeg{\vr}{2}}^2 + \twonorm{\monoindeg{\vt}{2}}^2 + \twonorm{\vr \kron \vt}^2 + \twonorm{\vtheta \kron \vr}^2 + \twonorm{\vtheta \kron \vt}^2 \\
\hspace{-5mm} \leq 13 + 4N + 4T^2 + T^4 + NT^2 := M_0.
\eea 

{\bf Example~\ref{eg:meshRegistration} (Mesh Registration)}. Same as point cloud registration.

{\bf Projection Maps}. To carry out the \DRS iterates~\eqref{eq:DRSIterates}, we need to implement the projection onto the convex cone $\calK$, $\proj_{\calK}$, and the projection onto the affine subspace $\calA$, $\proj_{\calA}$. The projection onto the PSD cone has a closed-form solution, due to Higham~\cite{Higham88LA-nearestSPD}.
\begin{lemma}[\bf Projection onto $\psd^n$~\cite{Higham88LA-nearestSPD}]
Given any matrix $\MS \in \sym^n$, let $\MS = \MU \diag{\lambda_1,\dots,\lambda_n}\MU\tran$ be its spectral decomposition, then the projection of $\MS$ onto the PSD cone $\psd^n$ is:
\bea
\proj_{\psd^n}\parentheses{\MS} = \MU \diag{\max\parentheses{0,\lambda_1},\dots,\max\parentheses{0,\lambda_n}}\MU\tran.
\eea
\end{lemma}
Using this Lemma and the expression of the convex cone $\calK$ in eq.~\eqref{eq:dualCone}, the projection of $\dualVar$ onto $\calK$ can be performed component-wise:
\bea
\hspace{-7mm} \proj_\calK\parentheses{\dualVar} = \bracket{\vlambda\tran,\svec{\proj_{\psd^{\dimbasis{\dimxextend}{1}}}\parentheses{\MS_1}},\dots, \svec{\proj_{\psd^{\dimbasis{\dimxextend}{1}}}\parentheses{\MS_\nrIneq}}, \svec{\proj_{\psd^{\dimbasis{}{\rbasisset}}}\parentheses{\MS_0}}}\tran,
\eea
where $\vlambda \in \Real{m_{d_1}}$ are the unconstrained variables in $\dualVar$ (\cf~eq.~\eqref{eq:dualVar}). 

For the affine subspace $\calA = \cbrace{\dualVar: \MA \dualVar = \vb}$, the projection onto $\calA$ is~\cite{Henrion12handbook-conicProjection}:
\bea
\proj_\calA\parentheses{\dualVar} = \dualVar - \MA\tran\parentheses{\MA\MA\tran}\inv \parentheses{\MA \dualVar - \vb}. \label{eq:projAffine}
\eea
The next theorem states that the inverse $\parentheses{\MA\MA\tran}\inv$ can be computed efficiently using the Matrix Inversion Lemma~\cite{Hager89SIAM-matrixInverse}.
\begin{theorem}[\bf Efficient Matrix Inversion]
Let $\MA = [\MA_1,\MA_2,\MA_3]$ be the partition of $\MA$ as in Theorem~\ref{thm:partialorthogonality}. Denote $\MA_{12} = [\MA_1,\MA_2]$, and $\MD = \MA_3\MA_3\tran$ as the invertible diagonal matrix. Then the inverse of $\MA\MA\tran$ is:
\bea
\parentheses{\MA\MA\tran}\inv = \MD\inv - \MD\inv \MA_{12} \parentheses{\eye_{m_{d_1} + m_{d_2}} + \MA_{12}\tran \MD\inv \MA_{12}}\inv \MA_{12}\tran \MD\inv. \label{eq:inverseAAtran}
\eea
\end{theorem}
\begin{proof}
Write $\MA\MA\tran =  \MD + \MA_{12}\MA_{12}\tran$, and invoke the Matrix Inversion Lemma. 
\end{proof}
The computational benefit brought by the partial orthogonality property of $\MA$ is that, in eq.~\eqref{eq:inverseAAtran}, only a matrix of size $m_{d_1} + m_{d_2}$ needs to be inverted (the inversion of the diagonal matrix $\MD$ is cheap), although the matrix $\MA\MA\tran$ has size $\dimbasis{}{2\rbasisset}$, which is typically much larger. Partial orthogonality has been exploited in several works to design scalable first-order solvers for solving SOS programs~\cite{Zheng18TAC-partialOrthogonality,Bertsimas13OMS-acceleratedSOSRelaxation}.

Another computational advantage is, we can rewrite $\proj_\calA\parentheses{\dualVar}$ in eq.~\eqref{eq:projAffine} as:
\bea
\proj_\calA \parentheses{\dualVar} = \parentheses{\eye_{m_d} - \MA\tran\parentheses{\MA\MA\tran}\inv \MA} \dualVar + \MA\tran\parentheses{\MA\MA\tran}\inv \vb.
\eea
Because the matrix $\MA$ only depends on the constraints of the POP~\eqref{eq:pop} and is unrelated to the visual measurements $\measured_i$, both $\eye_{m_d} - \MA\tran\parentheses{\MA\MA\tran}\inv \MA$ and $\MA\tran\parentheses{\MA\MA\tran}\inv$ can be computed offline. Hence, during online optimality certification, only matrix-vector multiplications are required to perform the projection onto the affine subspace.
\end{proof}

%% file: supp-chordalSparse.tex
\section{Chordal Sparse Initialization}
\label{sec:supp-chordalSparse}
In theory, one can start \DRS~\eqref{eq:DRSIterates} at any initial condition $\dualVar_0$. However, to speed up \DRS, we compute the initial point $\vd_0$ by solving a cheap SOS program with \emph{chordal sparsity}~\cite{Waki06jopt-SOSSparsity,lasserre10book-momentsOpt}.
\begin{proposition}[\bf Chordal Sparse Initialization] \label{prop:chordalSparseInitialization}
Define $\monoleq{\xextend}{\rbasisset_i} \doteq [1,\vxx\tran,\theta_i, \theta_i \vxx\tran]\tran \in \Real{2n+2}$, $\monoleq{\xextend}{1i} \doteq [1,\vxx\tran,\theta_i]\tran \in \Real{n+2}$, as the sparse monomial bases only in $\vxx$ and $\theta_i,i=1,\dots,N$. 
Let the solution of the following SOS program (SDP):
\bea
\max & \zeta \in \Real{}   \label{eq:soschordalsparse}\\
 \subject &  \hspace{-3mm} \displaystyle f(\xextend)\! - \! \zeta \! - \!\!\!\! \sum_{h_j \in \vh^x} \!\! h_j  \cdot \! \left(\monoleq{\xextend}{2}\tran\vlambda_j^x \right) \! - \!\!\!\! \sum_{h_j \in \vh^{\theta} } \!\! h_j \cdot \! \left(\monoleq{\vxx}{2}\tran \vlambda_j^\theta \right) \! = \nonumber \\
  & \displaystyle \! \sumallpoints \monoleq{\xextend}{\rbasisset_i}\tran \MS_{0i} \monoleq{\xextend}{\rbasisset_i} \! + \!\! \sum_{k=1}^{\nrIneq} g_k \cdot \parentheses{ \sumallpoints \!\monoleq{\xextend}{1i}\tran \MS_{ki} \monoleq{\xextend}{1i} } \! ,\edit{ \forall \xextend}, \label{eq:matchingcoefficientschordalsparse} \\
  & \vlambda_j^x \in \Real{\dimbasis{\dimxextend}{2}}, \vlambda_j^\theta \in \Real{\dimbasis{n}{2}}, \MS_{0i} \in \psd^{2n+2}, \MS_{ki} \in \psd^{n+2},
\eea
be $\vlambda_j^{x\star}$, $\vlambda_j^{\theta\star}$, $\MS_{0i}^\star$ and $\MS_{ki}^\star$, then $\dualVar_0$ can be constructed as:
\bea
\hspace{-8mm} \dualVar_0 = [(\vlambda_1^{x\star})\tran,\dots,(\vlambda^{x\star}_{|\vh^x|})\tran,(\vlambda_1^{\theta\star})\tran,\dots,(\vlambda_{|\vh^\theta|}^{\theta\star})\tran,\svec{\barMS_1^\star}\tran,\dots,\svec{\barMS_\nrIneq^\star}\tran,\svec{\barMS^\star_0}\tran]\tran, \label{eq:dualVarzero}
\eea
where $\barMS^\star_k \in \sym^{\Real{\dimbasis{\dimxextend}{1}}},k=1,\dots,\nrIneq$, and $\barMS^\star_0 \in \sym^{\dimbasis{}{\rbasisset}}$ satisfy:
\bea
& \monoleq{\xextend}{1}\tran \barMS^\star_k  \monoleq{\xextend}{1} = \sumallpoints \monoleq{\xextend}{1i}\tran \MS^\star_{ki}  \monoleq{\xextend}{1i},\ \forall \xextend, \label{eq:chordalSparse1}\\
& \monoleq{\xextend}{\rbasisset}\tran \barMS^\star_0  \monoleq{\xextend}{\rbasisset} = \sumallpoints \monoleq{\xextend}{\rbasisset_i}\tran \MS^\star_{0i} \monoleq{\xextend}{\rbasisset_i},\ \forall \xextend. \label{eq:chordalSparse2}
\eea
\end{proposition}
The chordal sparse SDP~\eqref{eq:soschordalsparse} is different from the SDP~\eqref{eq:sosfeasiblity} in two aspects. First, we have relaxed the large PSD constraints into multiple much smaller PSD constraints with fixed sizes (independent of the number of measurements $N$). For example, $\MS_0 \in \psd^{\dimbasis{}{\rbasisset}}$ has been divided into $\MS_{0i} \in \psd^{2n+2},i=1,\dots,N$, where each $\MS_{0i}$ is associated with a sparse monomial basis $\monoleq{\xextend}{\rbasisset_i}$ of fixed size. Similarly, each $\MS_{k} \in \psd^{\dimbasis{\dimxextend}{1}}$ has been divided into $N$ smaller PSD constraints $\MS_{ki} \in \psd^{n+2},i=1,\dots,N$. Second, instead of trying to certify $\hatf$ is the global minimum of $f(\xextend)$, we turn to maximize a lower bound $\zeta$ of $f(\xextend)$. The reason is, by relaxing the large PSD constraints into multiple smaller constraints (\ie~by requiring the SOS polynomials in the feasibility SDP~\eqref{eq:sosfeasiblity} to admit chordal sparse decompositions as in~\eqref{eq:chordalSparse1} and~\eqref{eq:chordalSparse2}), problem~\eqref{eq:soschordalsparse} is more restrictive than problem~\eqref{eq:sosfeasiblity} and in general its optimum $\zeta^\star$ cannot certify the global optimality of $\hatf$ (\ie~$\zeta^\star < p^\star_{\rbasisset} \leq f^\star \leq \hatf$).\footnote{From a different perspective, $\MS_k \succeq 0$ and $\MS_0\succeq 0$ imply that there must exist smaller PSD decompositions $\MS_{ki} \succeq 0$ and $\MS_{0i} \succeq 0$. However, $\MS_{ki} \succeq 0$ and $\MS_{0i} \succeq 0$ do not necessarily mean $\MS_k \succeq 0$ and $\MS_0\succeq 0$.} However, the chordal sparse SOS program~\eqref{eq:soschordalsparse} scales to large $N$. Therefore, we compute $\dualVar_0$ by solving this cheap SOS program~\eqref{eq:soschordalsparse} using an IPM-based SDP solver and then refine $\dualVar_0$ by running \DRS for the more powerful (but more expensive) SOS program~\eqref{eq:sosfeasiblity}.

%% file: supp-experiments.tex
\section{Details of Experiments}
\label{sec:supp-experiments}

\subsection{Details of Experimental Setup}
We test primal relaxation and dual certification on \emph{random} problem instances of Examples~\ref{eg:singleRotationAveraging}-\ref{eg:meshRegistration}. At each Monte Carlo run, we generate inliers and outliers as follows. In single rotation averaging (\singlerotation), we first randomly generate a ground-truth 3D rotation $\Rgt$, then inliers are generated by $\MR\inlier = \Rgt \MR_{\epsilon}$, where $\MR_{\epsilon}$ is generated by randomly sampling a unit-norm rotation axis $\vPsi \in \Real{3}$ and a rotation angle $\phi\sim \calN(0,\sigma^2)$ with $\sigma = 3^\circ$; outliers are arbitrary random rotations. In shape alignment (\shapealign), we first randomly generate a 3D shape $\{\MB_i\}_{i=1}^N$, where each $\MB_i \sim \calN(\zero,\eye_3)$, and then scale the shape such that its diameter (\ie~maximum distance between two points) is $4$. We then generate a random ground-truth scale $\sgt \in [0.5,2]$ and a random ground-truth rotation $\Rgt$. Inlier 2D measurements are generated by $\vb\inlier = \sgt \Pi \Rgt \MB + \vepsilon$, where $\vepsilon \sim \calN(\zero,\sigma^2 \eye_2)$ with $\sigma = 0.01$, and outliers are arbitrary 2D vectors $\vb\outlier \sim \calN(\zero,\eye_2)$. In point cloud registration (\pointcloud), we first generate $\{\va_i\}_{i=1}^N$ in the same way as generating $\{\MB_i\}_{i=1}^N$ in \shapealign. Then we sample a random rotation $\Rgt$ and a random translation $\tgt$ with $\| \tgt \|\leq T = 1$. Inlier 3D points are generated by $\vb\inlier = \Rgt \va + \tgt + \vepsilon$, where $\vepsilon \sim \calN(\zero,\sigma^2 \eye_3)$ with $\sigma = 0.01$; and outliers are arbitrary random vectors $\vb\outlier \in \calN(\zero,\eye_3)$. In mesh registration (\mesh), we first generate a random mesh by sampling unit normals $\{\vu_i \}_{i=1}^N$ and points $\{\va_i\}_{i=1}^N$ the same way as in \shapealign. Then we generate a random rotation $\Rgt$ and translation $\tgt,\| \tgt \|\leq T = 1$. Inlier normals are generated by $\vv\inlier = (\Rgt \vu + \vepsilon) / \| \Rgt \vu + \vepsilon \|$, where $\vepsilon \sim \calN(\zero,\sigma^2 \eye_3)$ with $\sigma = 0.01$. Inlier points are generated by $\vb\inlier = \Rgt(\va + \vu \times \vPhi) + \tgt + \vepsilon$, where $\vPhi \sim \calN(\zero,\eye_3)$ and $\vepsilon \sim \calN(\zero,\sigma^2 \eye_3)$ with $\sigma=0.01$ (note that $\va + \vu \times \vPhi$ generates a random point on the face defined by $(\va,\vu)$). Outlier normals are randomly generated unit-norm 3D vectors $\vv\outlier$ and outlier points are randomly generated $\vb\outlier \sim \calN(\zero,\eye_3)$. The relative weight between point-to-plane distance and normal-to-normal distance is set to be $w_i = 1,i=1,\dots,N$. In problem~\eqref{eq:generalTLS}, $\barc=1$ for all applications, and $\beta_i^2 = \beta^2,i=1,\dots,N$, is set as follows. In \singlerotation, $\beta^2 = (2\sqrt{2}\sin (3 \sigma/2))^2$. In \shapealign, $\beta^2 = \sigma^2 \cdot \chiinv(2,0.99)$. In \pointcloud, $\beta^2 = \sigma^2 \cdot \chiinv(3,0.99)$. In \mesh, $\beta^2 = 2\sigma^2 \cdot \chiinv(3,0.99)$, where $\chiinv(d,p)$ computes the quantile of the $\chi^2$ distribution with $d$ degrees of freedom and lower tail probability equal to $p$ (see~\cite{Yang19iccv-QUASAR} for a probabilistic interpretation).

\subsection{Dense vs. Sparse Moment Relaxation}
We compare the performance of the dense moment relaxation~\eqref{eq:denseMoment} and the sparse moment relaxation~\eqref{eq:sparseMoment} with $N=10$ measurements, because the dense relaxation is too large to be solved by IPM solvers at $N=20$. Fig.~\ref{fig:supp_dense_sparse} boxplots the rotation estimation error (\blue{left axis}) and the relative duality gap (\red{right axis}) averaged over 30 Monte Carlo runs for the four Examples~\ref{eg:singleRotationAveraging}-\ref{eg:meshRegistration}. For single rotation averaging (Fig.~\ref{fig:supp_dense_sparse}(a)), both the dense and spare relaxations are tight up to $80\%$ outlier measurements (relative duality gap always below $10^{-5}$), and both of them return accurate rotation estimations (rotation error always below $5$ degrees). For shape alignment and point cloud registration (Fig.~\ref{fig:supp_dense_sparse}(b)(c)) , both the dense and sparse relaxations produce occasional non-tight solutions (especially at high-outlier regime). However, we see that the rotation estimations are still quite accurate. We observed that the relaxation becomes tighter for increasing $N$. Indeed, the results in the paper shows improved performance for $N=20$. Hence, we conjecture that, when $N$ is small, the estimation problem is more ``ambiguous'' for the relaxations, in the sense that inliers do not form a dominating consensus set as strong as when $N$ is large. This is similar to human perception: we recognize the patterns more easily when we see dense visual measurements (\eg~a dense point cloud vs. a sparse point cloud of only a few points). For mesh registration (Fig.~\ref{fig:supp_dense_sparse}(d)), the relaxations are always tight, and significantly better than the case of point cloud registration. This echoes our previous conjecture: adding surface normals to the visual measurements provides more cues and makes the estimation less ``ambiguous''. Finally, it is also interesting to see that at $80\%$ outlier rate (there are only 2 inliers), there are two runs where the relaxations produce the globally optimal solutions (because the relative duality gap is below $10^{-5}$), but the globally optimal solutions are far away from the ground-truth solutions (the rotation errors are $90$ and $180$ degrees). We suspect the reason is the possible symmetry in the randomly generated problems, as also observed in~\cite{Yang20arxiv-teaser}.
\input{supp-fig-denseSparse}

\subsection{Results for Point Cloud Registration}
Fig.~\ref{fig:supp_pcr} shows the performance of primal relaxation and dual certification on point cloud registration, and the results look qualitatively the same as the results for mesh registration in the main text.
\input{supp-fig-pcr}

\subsection{Details of Satellite Pose Estimation}
The neural network in~\cite{Chen19ICCVW-satellitePoseEstimation} learns a 3D model of the Tango satellite consisting of 11 keypoints $\cbrace{\MB_i}_{i=1}^{11}$, shown in Fig.~\ref{fig:supp_satellite}(a). It can also output $11$ 2D landmark detections for a given 2D image, $\cbrace{\vb_i}_{i=1}^{11}$, shown in Fig.~\ref{fig:supp_satellite}(b). We assume a weak perspective camera model\footnote{Weak perspective camera model is a good approximation of the full perspective camera model when the object is far away from the camera center~\cite{Zhou17pami-shapeEstimationConvex,Yang20cvpr-shapeStar}.} and the inlier 3D keypoints and 2D landmarks satisfy the following generative model:
\bea
\vb_i = s\Pi\MR \MB_i + \vt + \vepsilon_i,
\eea
where $\vt \in \Real{2}$ is a 2D translation and $\vepsilon_i$ models an unknown but bounded additive noise that satisfies $\twonorm{\vepsilon_i} \leq \delta_i$. Then the \emph{pairwise relative} 3D keypoints and 2D landmarks will satisfy the shape alignment model used in Example~\ref{eg:shapeAlignment}:
\bea
\underbrace{\vb_i - \vb_j}_{\TIMb_{ij}} = s\Pi\MR \underbrace{\parentheses{\MB_i - \MB_j}}_{\TIMB_{ij}} + \underbrace{\parentheses{\vepsilon_i - \vepsilon_j}}_{\TIMeps_{ij}},
\eea
because the translation $\vt$ cancels out due to the subtraction, and $\twonorm{\TIMeps_{ij}} \leq \delta_i + \delta_j$ models the updated noise. Because there are 11 keypoints and landmarks, we have $K = \nchoosek{11}{2} = 55$ pairwise measurements $\cbrace{\TIMB_k}_{k=1}^K$ and $\cbrace{\TIMb_k}_{k=1}^K$. Using the $K$ pairwise measurements, we can first estimate $s$ and $\MR$ using the certifiable algorithms discussed in the main text, and then estimate the translation using the adaptive voting method in~\cite{Yang19rss-teaser}. The full 6D camera pose can be recovered as:
\bea
\MR^{3D} = \MR, \quad \vt^{3D} = \frac{\bracket{\vt\tran,1}\tran}{s}.
\eea

When spoiling outliers, we replace $l$ landmarks $\vb_i$'s with random 2D pixels, which implies that the outlier rate should be computed as:
\bea
1 - \frac{\nchoosek{11-l}{2}}{55}, \label{eq:satelliteOutlierRate}
\eea
where $\nchoosek{11-l}{2}$ is the number of inlier pairwise relative measurements (a pairwise measurement $\TIMb_{ij}$ is an inlier if and only if both $\vb_i$ and $\vb_j$ are inliers). Using the formula in eq.~\eqref{eq:satelliteOutlierRate}, the outlier rates are $0\%$, $18\%$, $35\%$, $49\%$, $62\%$ and $73\%$ for $l=0,1,2,3,4,5$. 

Extra results and visualizations are provided in Fig.~\ref{fig:supp_satellite}. These results were certified as correct by the dual optimality certifiers presented in the main text.

\input{supp-fig-satellite}

\subsection{Comparison to Primal Baselines}

\final{Fig.~\ref{fig:baselines}(a) compares the performance of our primal solver (SDP: Basis Reduction) versus two state-of-the-art baselines: (1) \gnc (best heuristics, no optimality guarantees)~\cite{Yang20ral-GNC} and (2) SDP: Chordal Sparse (an efficient SDP relaxation that exploits correlative sparsity)~\cite{Waki06jopt-SOSSparsity} on single rotation averaging. Our primal relaxation is significantly tighter than chordal sparse relaxation, and the accuracy and robustness of our
estimates \finalLC{dominate} both baselines. }

\input{fig-response}

\subsection{Comparison to Certification Baselines}

\final{Our \DRS approach is the first mathematically rigorous approach for verifying solution correctness. We compare it with a heuristic method that performs Kolmogorov–Smirnov (KS) test~\cite{Massey51JASA-KStest} on the
squared residuals with a $\chi^2$ distribution (i.e., tests normality of the residuals classified as inliers). Fig.~\ref{fig:baselines}(b) shows that
KS test has many false positives/negatives, while ours has zero, for single rotation averaging.}

\subsection{Adversarial Outliers}
\final{ We performed tests with an adversarial outlier model (where outliers follow a different model and are consistent with each other) and test our algorithm (SDP: Basis Reduction) against two state-of-the-art baselines. Fig.~\ref{fig:baselines}(c) shows our method dominates both baselines, is insensitive to adversarial outliers until the maximum breakdown point 50\%. Note that our relaxation is still tight at 50\% outlier rate, certifying that the globally optimal solution is obtained. However, due to the presence of adversarial outliers, the globally optimal solution may not be the ground-truth solution (if we assume the ground-truth solution has a larger set of consistent measurements).}

%% file: supp-fig-denseSparse.tex

\newcommand{\mpwtwo}{7cm}
\renewcommand{\myhspace}{\hspace{-6mm}}
\newcommand{\myvspace}{\vspace{-1mm}}
\begin{figure}[h!]
	\begin{center}
	\begin{minipage}{\textwidth}
	\begin{tabular}{cc}%
		\myhspace
			\begin{minipage}{\mpwtwo}%
			\centering%
			\includegraphics[width=\columnwidth]{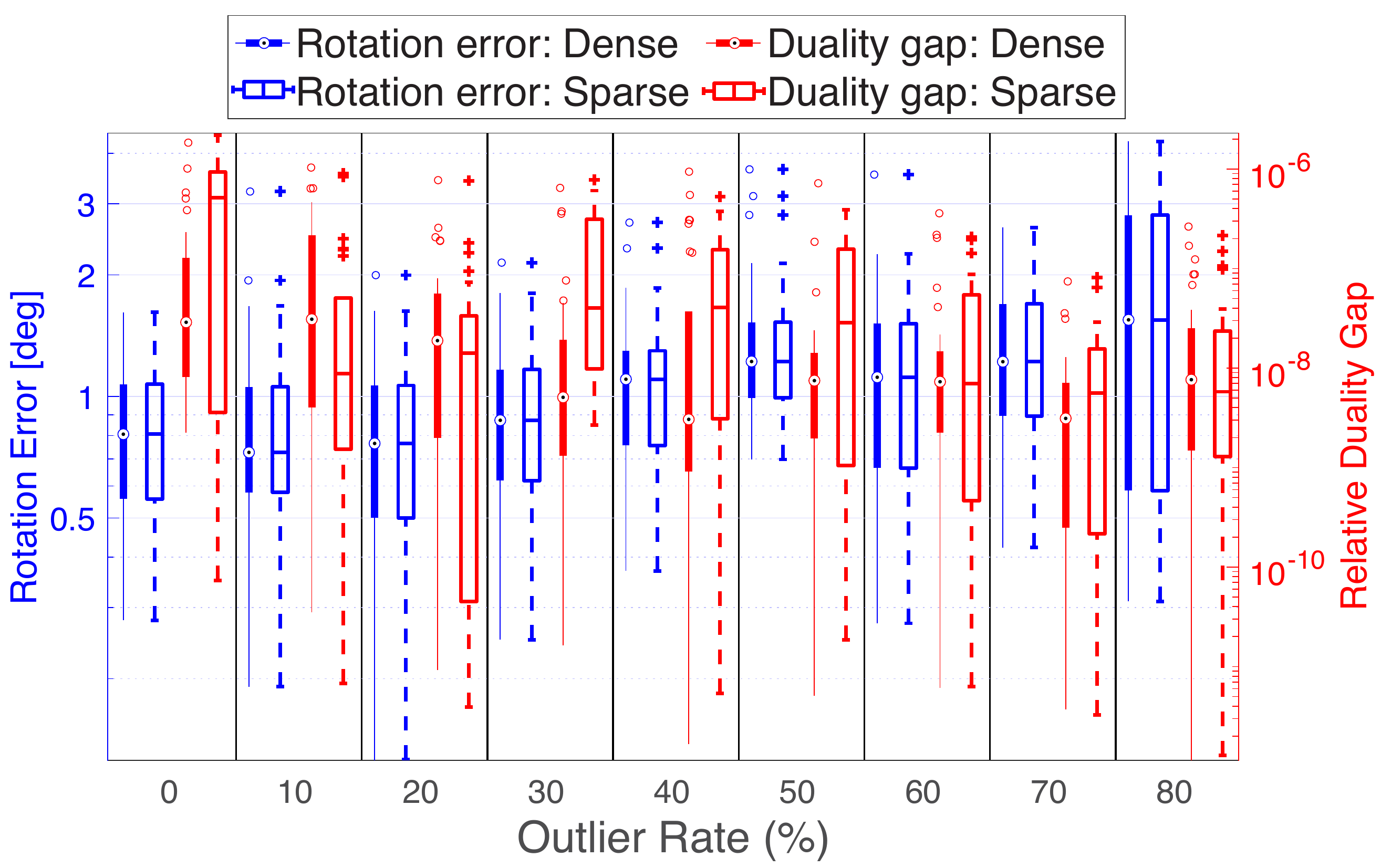} \\
			\myvspace
			{\smaller (a) Single Rotation Averaging}
			\end{minipage}
		&  
			\begin{minipage}{\mpwtwo}%
			\centering%
			\includegraphics[width=\columnwidth]{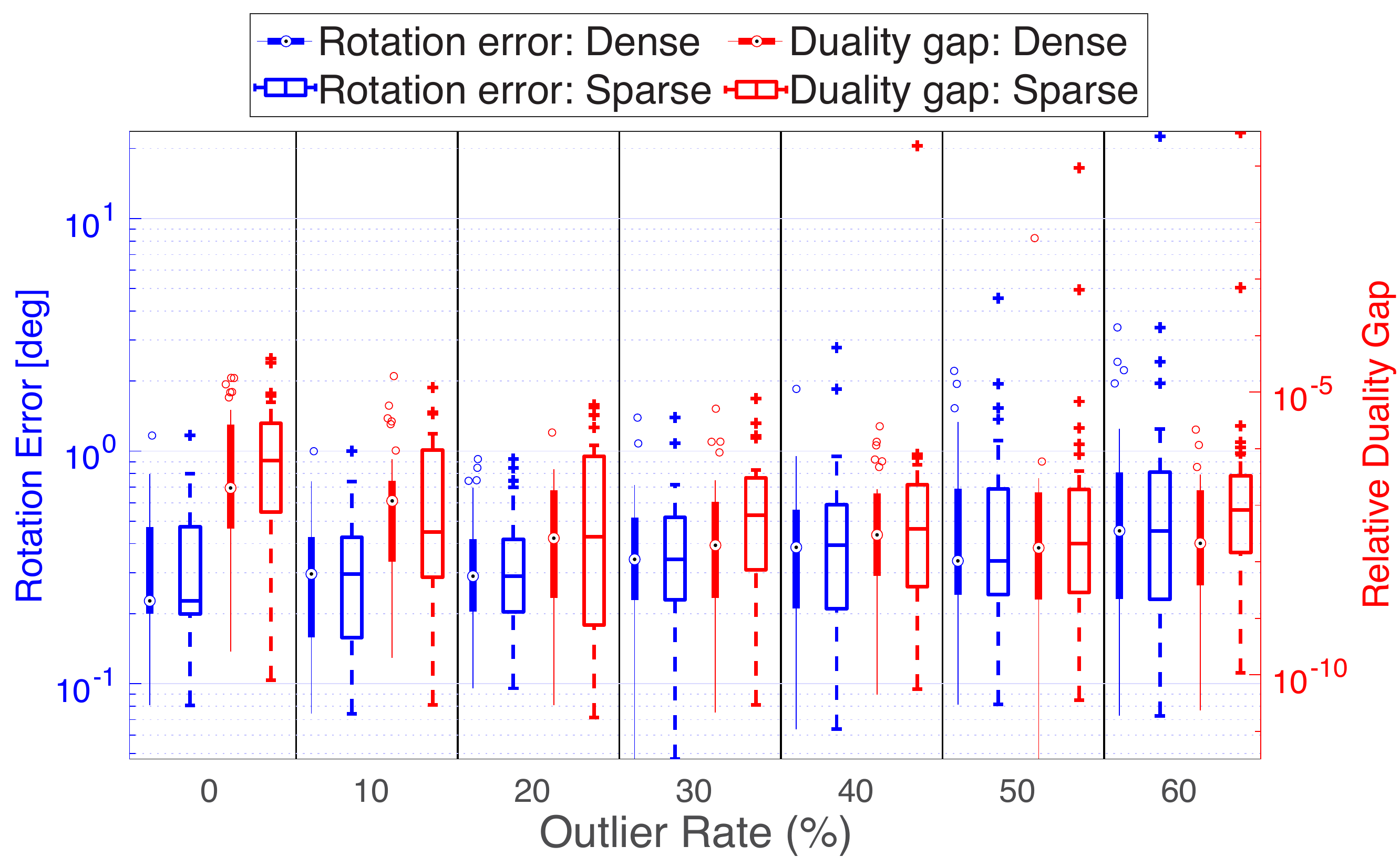} \\
			\myvspace
			{\smaller (b) Shape Alignment}
			\end{minipage} 
		\\
		\myhspace
			\begin{minipage}{\mpwtwo}%
			\centering%
			\includegraphics[width=\columnwidth]{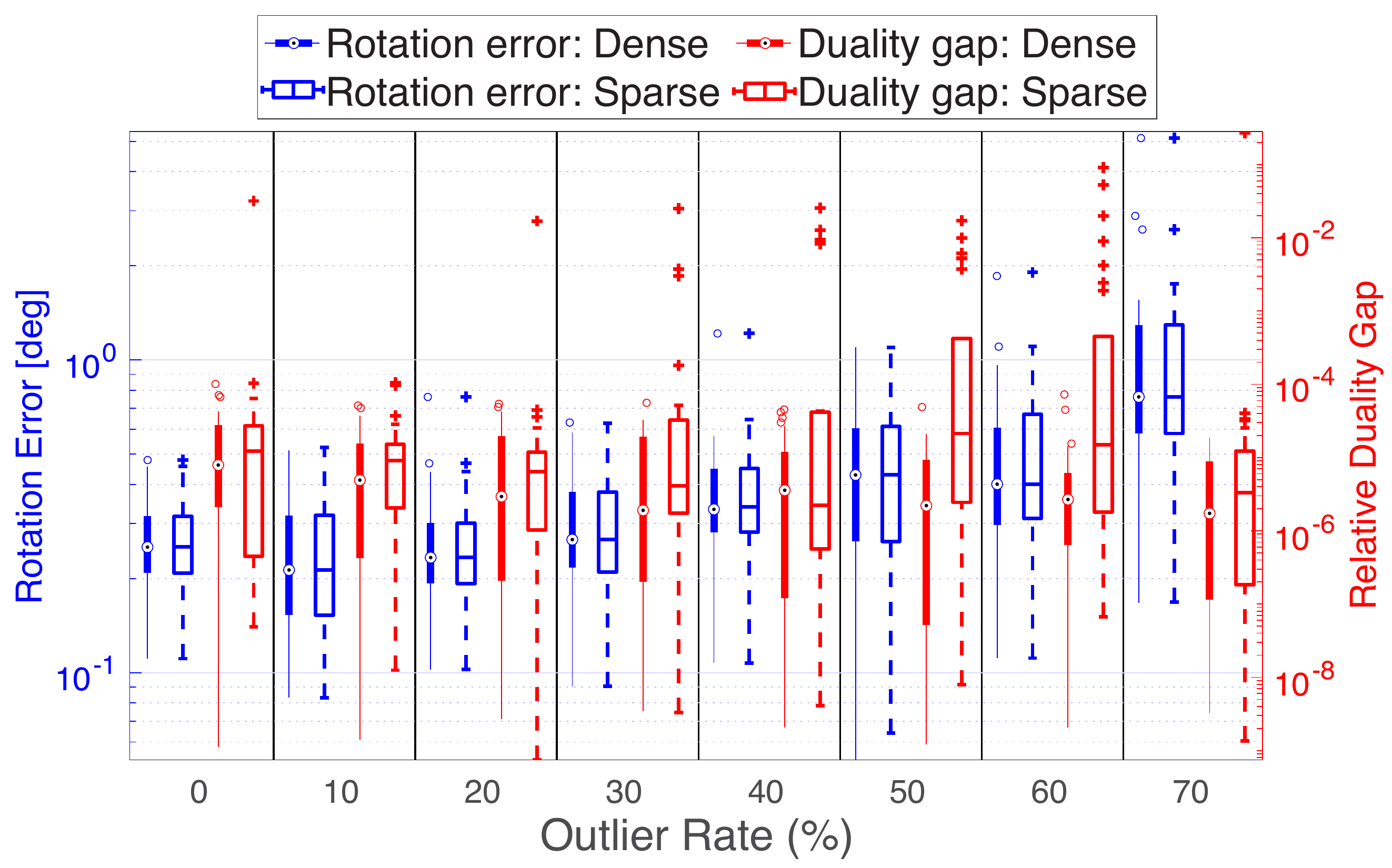} \\
			\myvspace
			{\smaller (c) Point Cloud Registration}
			\end{minipage}
		&  
			\begin{minipage}{\mpwtwo}%
			\centering%
			\includegraphics[width=\columnwidth]{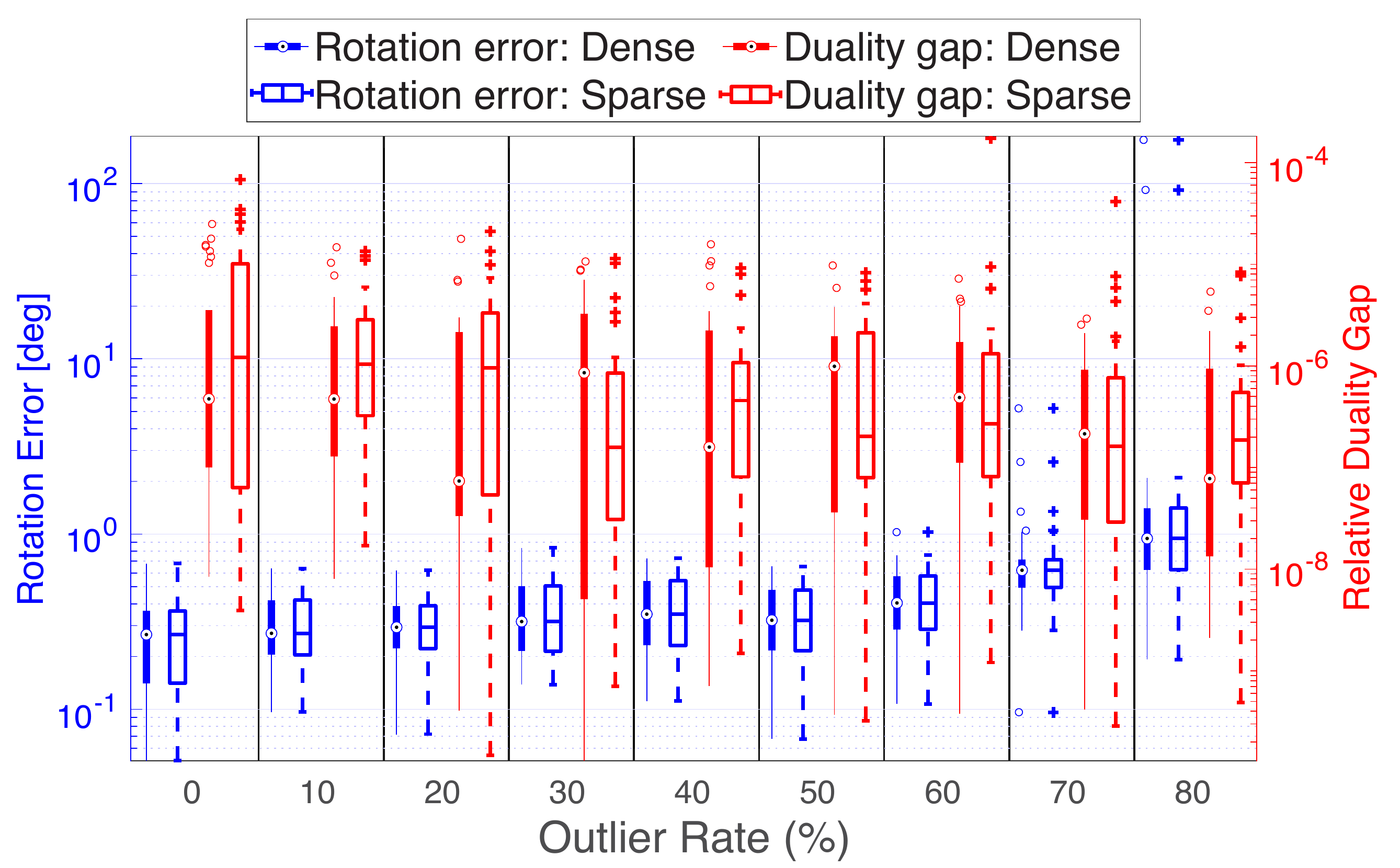} \\
			\myvspace
			{\smaller (d) Mesh Registration}
			\end{minipage}
	\end{tabular}
	\end{minipage} 
	\caption{\footnotesize Dense momemt relaxation vs. sparse moment relaxation on (a) Single Rotation Averaging, (b) Shape Alignment, (c) Point Cloud Registration, and (d) Mesh Registration. \blue{Left axis}: rotation estimation error; \red{right axis}: relative duality gap. $N=10$ and statistics are plotted over 30 Monte Carlo runs.
	\label{fig:supp_dense_sparse}} 
	\vspace{-6mm}
	\end{center}
\end{figure}

%% file: supp-fig-pcr.tex

\newcommand{\mpwthree}{4.5cm}
\begin{figure}[h]
	\begin{center}
	\begin{minipage}{\textwidth}
	\begin{tabular}{ccc}%
		\myhspace \hspace{3mm}
			\begin{minipage}{\mpwthree}%
			\vspace{4mm}
			\centering%
			\includegraphics[width=\columnwidth]{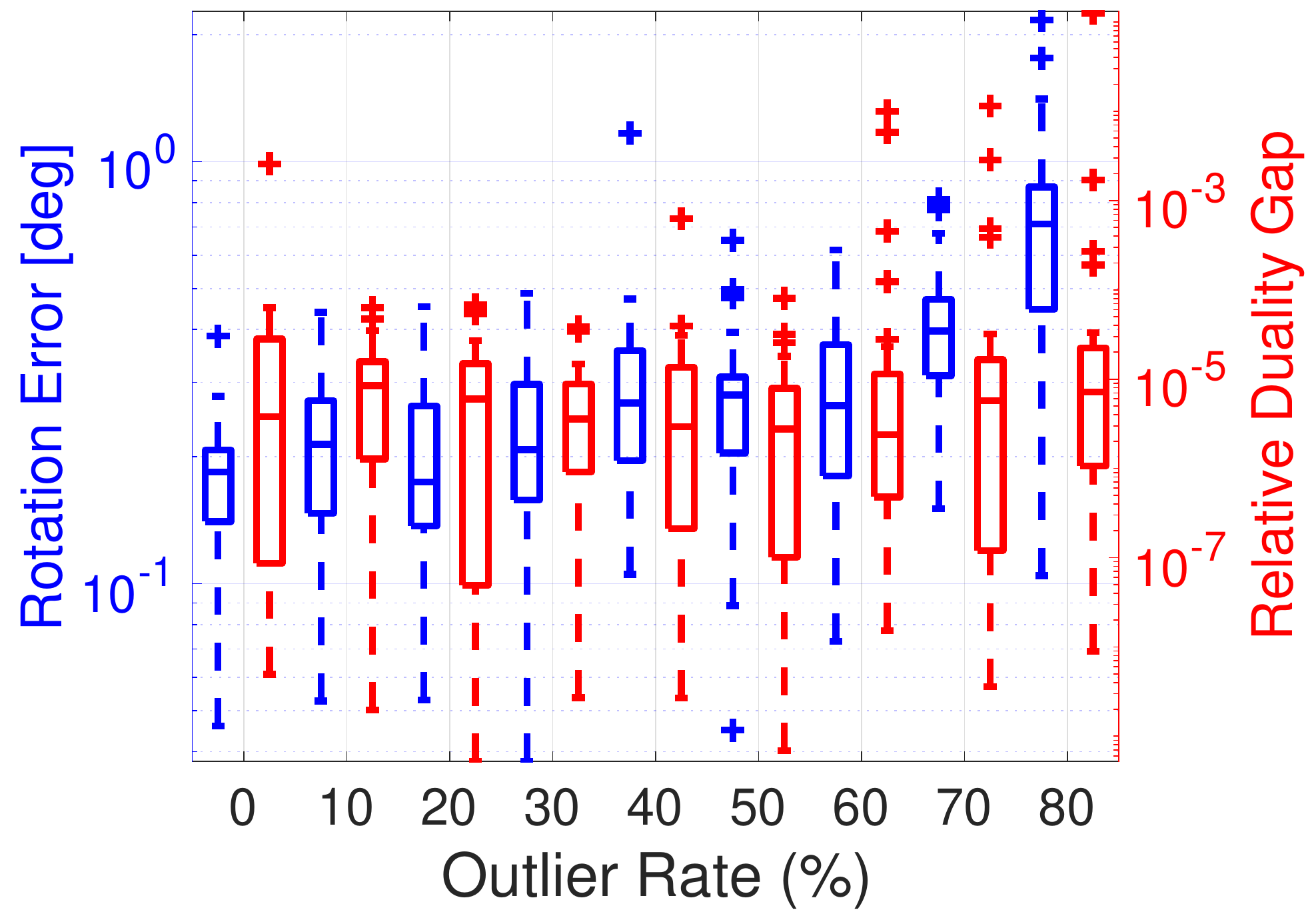} \\
			{\smaller (a) Sparse moment Relaxation}
			\end{minipage}
		&  \hspace{-3mm}
			\begin{minipage}{\mpwthree}%
			\centering%
			\includegraphics[width=\columnwidth]{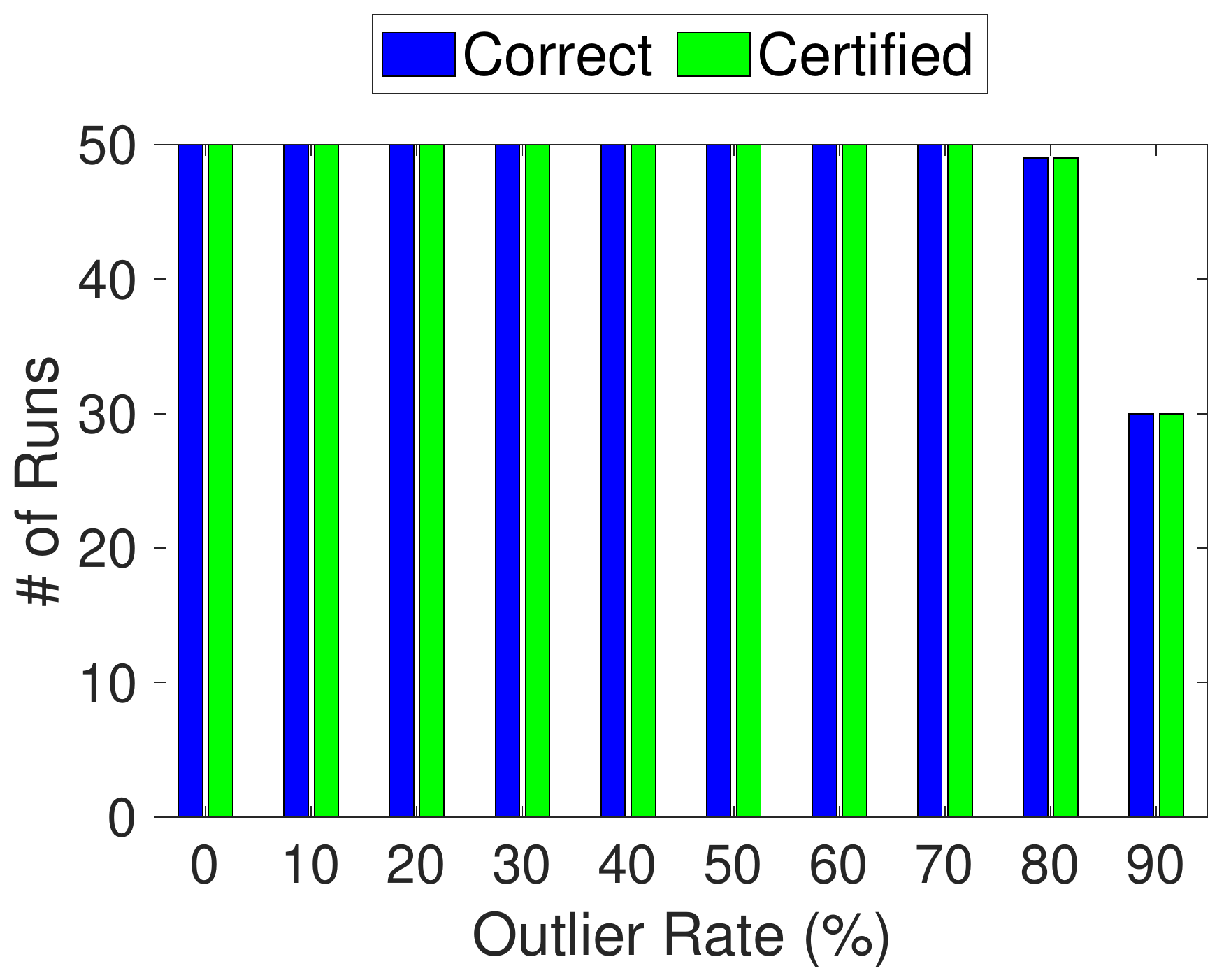} \\
			\myvspace
			{\smaller (b) Dual optimality certification}
			\end{minipage} 
		& \hspace{-3mm}
			\begin{minipage}{\mpwthree}%
			\centering%
			\includegraphics[width=\columnwidth]{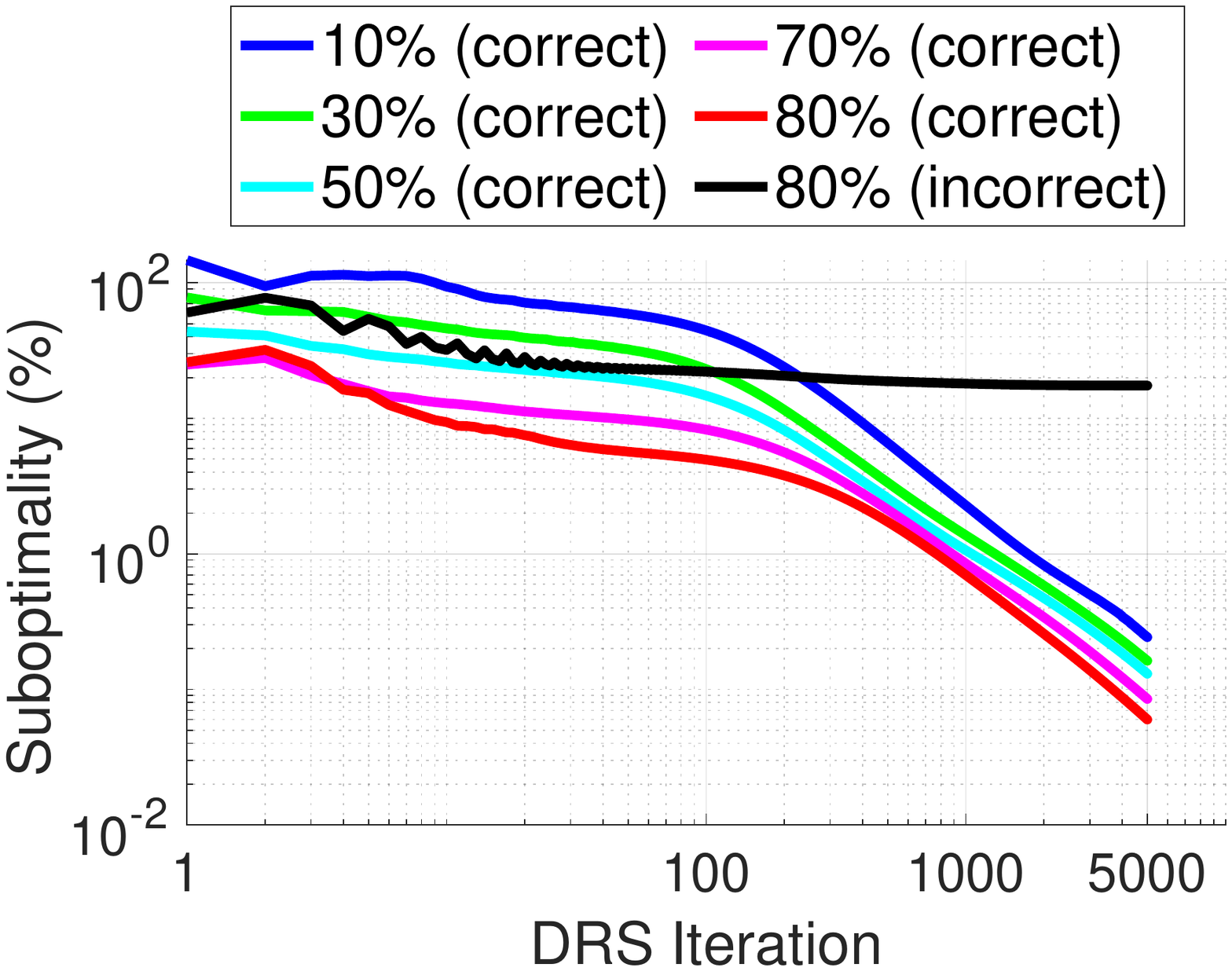} \\
			\myvspace
			{\smaller (c) Convergence of suboptimality}
			\end{minipage}
	\end{tabular}
	\end{minipage} 
	\caption{\footnotesize Certifiable point cloud registration. (a) Sparse moment Relaxation, (b) Dual optimality certification and (c) Convergence of suboptimality.
	\label{fig:supp_pcr}} 
	\vspace{-6mm}
	\end{center}
\end{figure}

%% file: supp-fig-satellite.tex

\begin{figure}[h]
	\begin{center}
	\begin{minipage}{\textwidth}
	\begin{tabular}{cc}%
		\myhspace
			\begin{minipage}{\mpwtwo}%
			\centering%
			\includegraphics[width=\columnwidth]{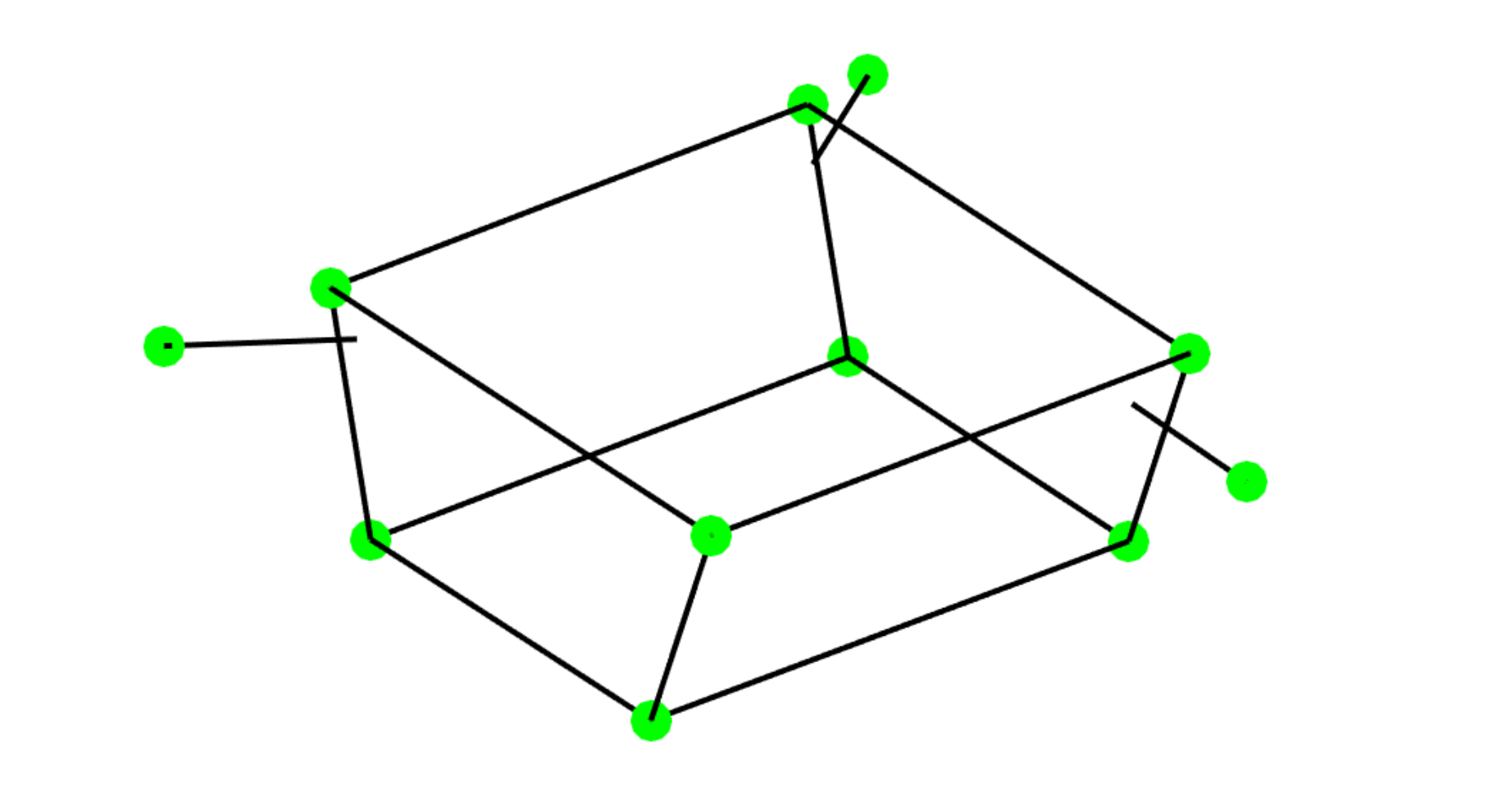} \\
			{\smaller (a) 3D wireframe model of Tango}
			\end{minipage}
		&  
			\begin{minipage}{\mpwtwo}%
			\centering%
			\includegraphics[width=\columnwidth]{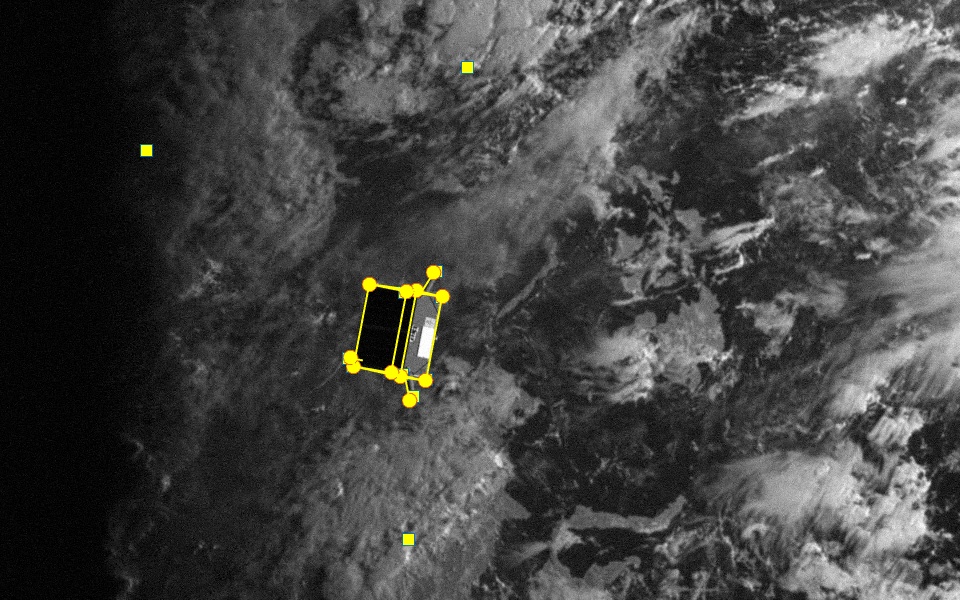} \\
			{\smaller (b) $l=3$ (49\% outlier rate)}
			\end{minipage} 
		\\
		\myhspace
			\begin{minipage}{\mpwtwo}%
			\centering%
			\includegraphics[width=\columnwidth]{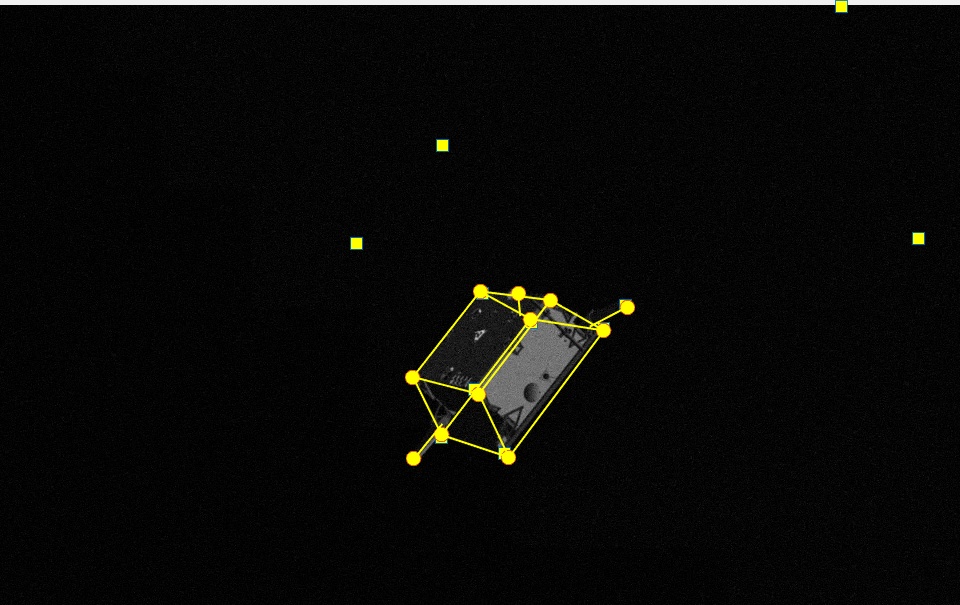} \\
			{\smaller (c) $l=4$ (62\% outlier rate)}
			\end{minipage}
		&  
			\begin{minipage}{\mpwtwo}%
			\centering%
			\includegraphics[width=\columnwidth]{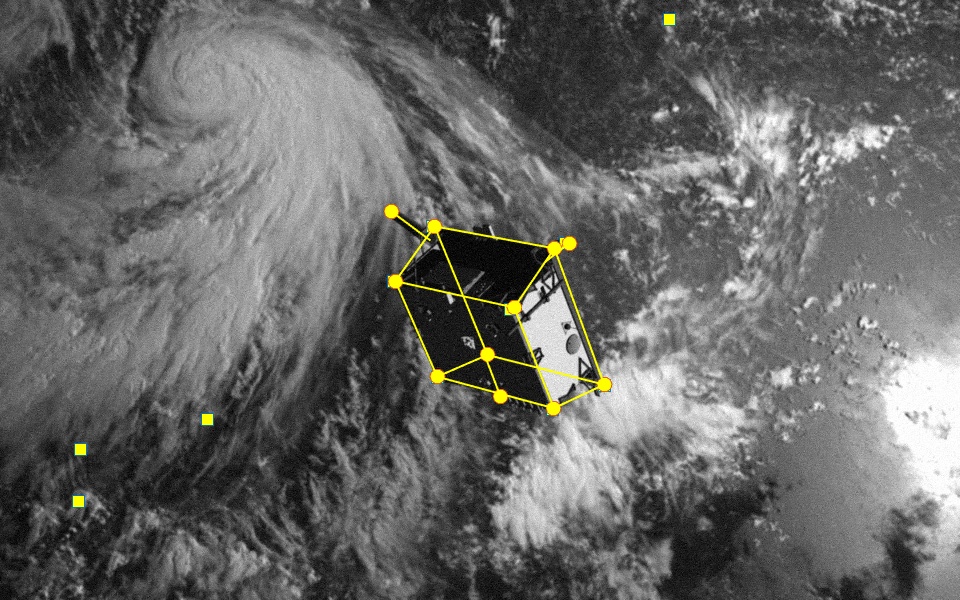} \\
			{\smaller (d) $l=4$ (62\% outlier rate)}
			\end{minipage}
		\\
		\myhspace
			\begin{minipage}{\mpwtwo}%
			\centering%
			\includegraphics[width=\columnwidth]{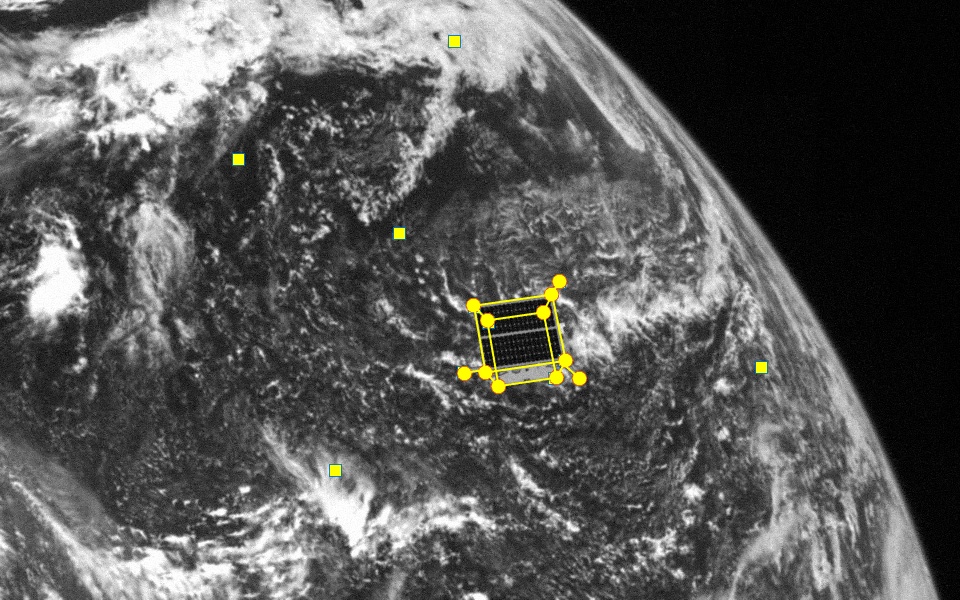} \\
			{\smaller (e) $l=5$ (73\% outlier rate)}
			\end{minipage}
		&  
			\begin{minipage}{\mpwtwo}%
			\centering%
			\includegraphics[width=\columnwidth]{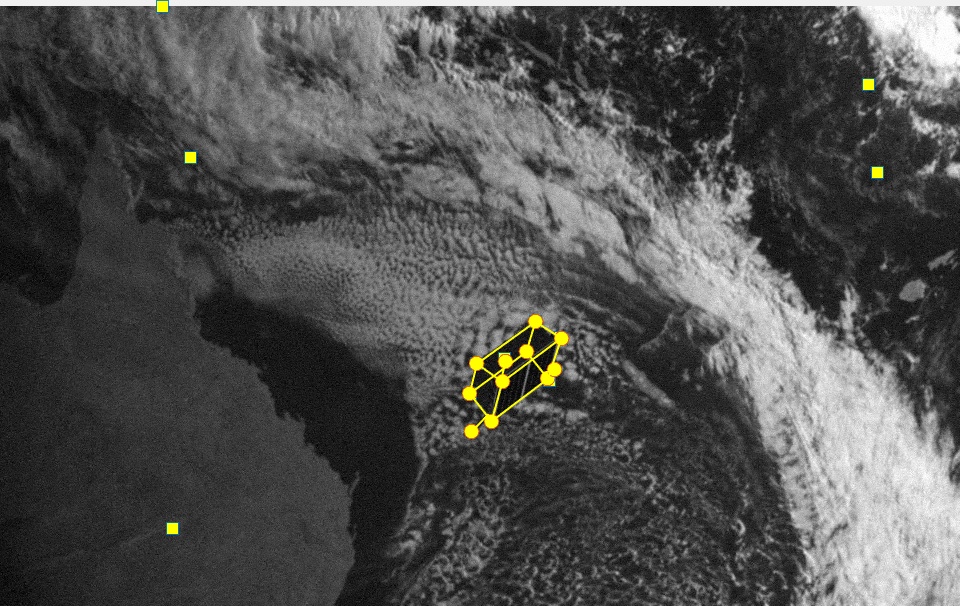} \\
			{\smaller (f) $l=5$ (73\% outlier rate)}
			\end{minipage}
	\end{tabular}
	\end{minipage} 
	\caption{\footnotesize Satellite pose estimation on the \SPEED dataset~\cite{Sharma19arXiv-SPEED}.
	\label{fig:supp_satellite}} 
	\vspace{-6mm}
	\end{center}
\end{figure}

%% file: fig-response.tex
\begin{figure}[t]
\begin{center}
	\hspace{-14mm}
	\begin{minipage}{\columnwidth}
	\begin{center}
			\includegraphics[width=1.1\columnwidth]{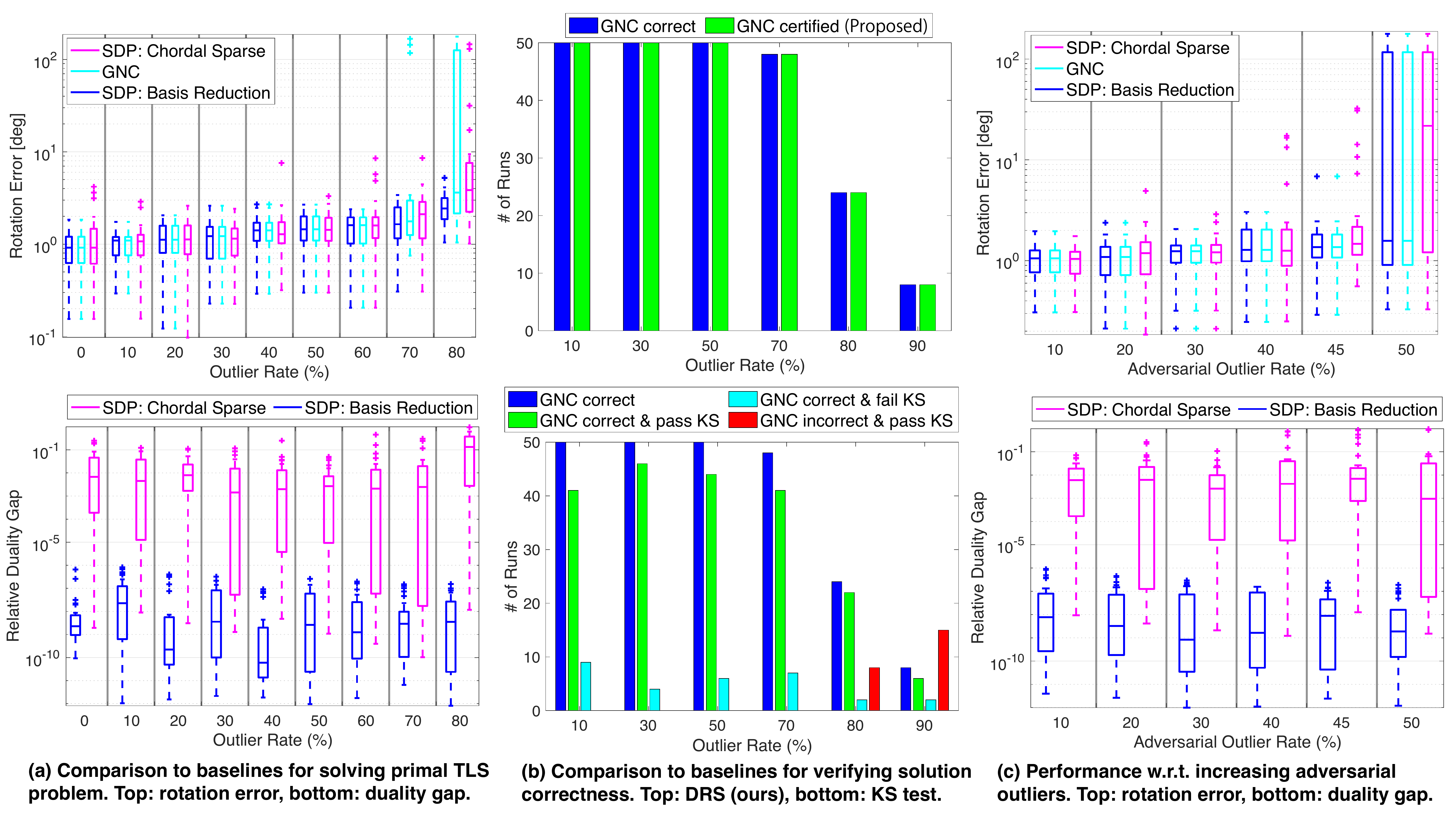}
			\vspace{-8mm}
	\end{center}
	\end{minipage} 
	\end{center}
	\caption{(a) Comparison to primal baselines. (b) Comparison to certification baselines. (c) Adversarial outliers. }
	\label{fig:baselines}
\end{figure}

%% file: main.bbl
\begin{thebibliography}{100}

\bibitem{Agarwal13icra}
P.~Agarwal, G.~D. Tipaldi, L.~Spinello, C.~Stachniss, and W.~Burgard.
\newblock Robust map optimization using dynamic covariance scaling.
\newblock In {\em IEEE Intl. Conf. on Robotics and Automation (ICRA)}, 2013.

\bibitem{Agarwal10eccv}
S.~Agarwal, N.~Snavely, I.~Simon, S.~M. Seitz, and R.~Szeliski.
\newblock Bundle adjustment in the large.
\newblock In {\em European Conf. on Computer Vision (ECCV)}, pages 29--42,
  2010.

\bibitem{Agostinho2019arXiv-cvxpnpl}
S{\'e}rgio Agostinho, Jo{\~a}o Gomes, and Alessio Del~Bue.
\newblock {CvxPnPL}: A unified convex solution to the absolute pose estimation
  problem from point and line correspondences.
\newblock {\em arXiv preprint arXiv:1907.10545}, 2019.

\bibitem{Aholt12ECCV-qcqpTriangulation}
Chris Aholt, Sameer Agarwal, and Rekha Thomas.
\newblock {A QCQP approach to triangulation}.
\newblock In {\em European Conf. on Computer Vision (ECCV)}, pages 654--667.
  Springer, 2012.

\bibitem{Antonante20arxiv-outlierRobustEstimation}
Pasquale Antonante, Vasileios Tzoumas, Heng Yang, and Luca Carlone.
\newblock Outlier-robust estimation: Hardness, minimally-tuned algorithms, and
  applications.
\newblock {\em arXiv preprint arXiv:2007.15109}, 2020.

\bibitem{mosek}
MOSEK ApS.
\newblock {\em The MOSEK optimization toolbox for MATLAB manual. Version 8.1.},
  2017.

\bibitem{Arun87pami}
K.S. Arun, T.S. Huang, and S.D. Blostein.
\newblock Least-squares fitting of two 3-{D} point sets.
\newblock {\em {IEEE} Trans. Pattern Anal. Machine Intell.}, 9(5):698--700,
  sept. 1987.

\bibitem{Audette00mia-surveyMedical}
M.~A. Audette, F.~P. Ferrie, and T.~M. Peters.
\newblock An algorithmic overview of surface registration techniques for
  medical imaging.
\newblock {\em Med. Image Anal.}, 4(3):201--217, 2000.

\bibitem{Bandeira16crm}
A.S. Bandeira.
\newblock A note on probably certifiably correct algorithms.
\newblock {\em Comptes Rendus Mathematique}, 354(3):329--333, 2016.

\bibitem{Bauschke96SIAM-projectionFeasibility}
Heinz~H Bauschke and Jonathan~M Borwein.
\newblock On projection algorithms for solving convex feasibility problems.
\newblock {\em SIAM review}, 38(3):367--426, 1996.

\bibitem{Bauschke04JAT-averagedAlternatingReflections}
Heinz~H Bauschke, Patrick~L Combettes, and D~Russell Luke.
\newblock Finding best approximation pairs relative to two closed convex sets
  in hilbert spaces.
\newblock {\em Journal of Approximation Theory}, 127(2):178--192, 2004.

\bibitem{Bazin12accv-globalRotSearch}
J.~C. Bazin, Y.~Seo, and M.~Pollefeys.
\newblock Globally optimal consensus set maximization through rotation search.
\newblock In {\em Asian Conference on Computer Vision}, pages 539--551.
  Springer, 2012.

\bibitem{Bertsimas13OMS-acceleratedSOSRelaxation}
Dimitris Bertsimas, Robert~M Freund, and Xu~Andy Sun.
\newblock An accelerated first-order method for solving sos relaxations of
  unconstrained polynomial optimization problems.
\newblock {\em Optimization Methods and Software}, 28(3):424--441, 2013.

\bibitem{Black96ijcv-unification}
Michael~J. Black and Anand Rangarajan.
\newblock On the unification of line processes, outlier rejection, and robust
  statistics with applications in early vision.
\newblock {\em Intl. J. of Computer Vision}, 19(1):57--91, 1996.

\bibitem{Blekherman12Book-sdpandConvexAlgebraicGeometry}
Grigoriy Blekherman, Pablo~A Parrilo, and Rekha~R Thomas.
\newblock {\em Semidefinite optimization and convex algebraic geometry}.
\newblock SIAM, 2012.

\bibitem{Bohorquez20arXiv-exactRelaxationRobustRegistration}
Cindy~Orozco Bohorquez, Yuehaw Khoo, and Lexing Ying.
\newblock Maximizing robustness of point-set registration by leveraging
  non-convexity.
\newblock {\em arXiv preprint arXiv:2004.08772}, 2020.

\bibitem{Bosse17fnt}
M.~Bosse, G.~Agamennoni, and I.~Gilitschenski.
\newblock Robust estimation and applications in robotics.
\newblock {\em Foundations and Trends in Robotics}, 4(4):225--269, 2016.

\bibitem{Bouaziz13acmsig-sparseICP}
S.~Bouaziz, A.~Tagliasacchi, and M.~Pauly.
\newblock Sparse iterative closest point.
\newblock In {\em ACM Symp. Geom. Process.}, pages 113--123. Eurographics
  Association, 2013.

\bibitem{Boumal16nips}
N.~Boumal, V.~Voroninski, and A.~Bandeira.
\newblock The non-convex {B}urer--{M}onteiro approach works on smooth
  semidefinite programs.
\newblock In {\em Advances in Neural Information Processing Systems (NIPS)},
  pages 2757--2765. 2016.

\bibitem{Briales17cvpr-registration}
Jesus Briales and Javier Gonzalez-Jimenez.
\newblock {Convex Global 3D Registration with Lagrangian Duality}.
\newblock In {\em IEEE Conf. on Computer Vision and Pattern Recognition
  (CVPR)}, 2017.

\bibitem{Briales18cvpr-global2view}
Jesus Briales, Laurent Kneip, and Javier Gonzalez-Jimenez.
\newblock A certifiably globally optimal solution to the non-minimal relative
  pose problem.
\newblock In {\em IEEE Conf. on Computer Vision and Pattern Recognition
  (CVPR)}, 2018.

\bibitem{Burer03mp}
{Burer, Samuel} and {Monteiro, Renato D C}.
\newblock {A nonlinear programming algorithm for solving semidefinite programs
  via low-rank factorization}.
\newblock {\em Mathematical Programming}, 95(2):329--357, 2003.

\bibitem{Cadena16tro-SLAMsurvey}
C.~Cadena, L.~Carlone, H.~Carrillo, Y.~Latif, D.~Scaramuzza, J.~Neira, I.~Reid,
  and J.J. Leonard.
\newblock Past, present, and future of simultaneous localization and mapping:
  Toward the robust-perception age.
\newblock {\em {IEEE} Trans. Robotics (T-RO)}, 32(6):1309--1332, 2016.

\bibitem{Carlone18ral-robustPGO2D}
L.~Carlone and G.~Calafiore.
\newblock Convex relaxations for pose graph optimization with outliers.
\newblock {\em {IEEE} Robotics and Automation Letters ({RA-L})},
  3(2):1160--1167, 2018.

\bibitem{Carlone16tro-duality2D}
L.~Carlone, G.~Calafiore, C.~Tommolillo, and F.~Dellaert.
\newblock Planar pose graph optimization: Duality, optimal solutions, and
  verification.
\newblock {\em {IEEE} Trans. Robotics (T-RO)}, 32(3):545--565, 2016.

\bibitem{Chatterjee13iccv}
A.~Chatterjee and V.~M. Govindu.
\newblock Efficient and robust large-scale rotation averaging.
\newblock In {\em Intl. Conf. on Computer Vision (ICCV)}, pages 521--528, 2013.

\bibitem{Chaudhury15Jopt-multiplePointCloudRegistration}
Kunal~N Chaudhury, Yuehaw Khoo, and Amit Singer.
\newblock Global registration of multiple point clouds using semidefinite
  programming.
\newblock {\em SIAM Journal on Optimization}, 25(1):468--501, 2015.

\bibitem{Chen19ICCVW-satellitePoseEstimation}
Bo~Chen, Jiewei Cao, Alvaro Parra, and Tat-Jun Chin.
\newblock Satellite pose estimation with deep landmark regression and nonlinear
  pose refinement.
\newblock In {\em Proceedings of the IEEE International Conference on Computer
  Vision Workshops}, 2019.

\bibitem{Chin17slcv-maximumConsensusAdvances}
T.~J. Chin and D.~Suter.
\newblock The maximum consensus problem: recent algorithmic advances.
\newblock {\em Synthesis Lectures on Computer Vision}, 7(2):1--194, 2017.

\bibitem{Choi15cvpr-robustReconstruction}
S.~Choi, Q.~Y. Zhou, and V.~Koltun.
\newblock Robust reconstruction of indoor scenes.
\newblock In {\em IEEE Conf. on Computer Vision and Pattern Recognition
  (CVPR)}, pages 5556--5565, 2015.

\bibitem{Cifuentes17arxiv}
D.~Cifuentes, S.~Agarwal, P.~Parrilo, and R.~Thomas.
\newblock On the local stability of semidefinite relaxations.
\newblock {\em ArXiv preprint: 1710.04287v1}, 2017.

\bibitem{Cifuentes19arXiv-BMguarantees}
Diego Cifuentes and Ankur Moitra.
\newblock Polynomial time guarantees for the burer-monteiro method.
\newblock {\em arXiv preprint arXiv:1912.01745}, 2019.

\bibitem{Combettes11book-proximalSplitting}
Patrick~L Combettes and Jean-Christophe Pesquet.
\newblock Proximal splitting methods in signal processing.
\newblock In {\em Fixed-point algorithms for inverse problems in science and
  engineering}, pages 185--212. Springer, 2011.

\bibitem{Crandall11cvpr}
D.~Crandall, A.~Owens, , N.~Snavely, and D.~Huttenlocher.
\newblock Discrete-continuous optimization for large-scale structure from
  motion.
\newblock In {\em IEEE Conf. on Computer Vision and Pattern Recognition
  (CVPR)}, pages 3001--3008, 2011.

\bibitem{Curto00TAMS-truncatedKmoment}
Ra{\'u}l Curto and Lawrence Fialkow.
\newblock {The truncated complex K-moment problem}.
\newblock {\em Transactions of the American mathematical society},
  352(6):2825--2855, 2000.

\bibitem{Dym17Jopt-exactPMSDP}
Nadav Dym and Yaron Lipman.
\newblock Exact recovery with symmetries for procrustes matching.
\newblock {\em SIAM Journal on Optimization}, 27(3):1513--1530, 2017.

\bibitem{Eriksson18cvpr-strongDuality}
A.~Eriksson, C.~Olsson, F.~Kahl, and T.-J. Chin.
\newblock Rotation averaging and strong duality.
\newblock {\em IEEE Conf. on Computer Vision and Pattern Recognition (CVPR)},
  2018.

\bibitem{Fischler81}
M.~Fischler and R.~Bolles.
\newblock Random sample consensus: a paradigm for model fitting with
  application to image analysis and automated cartography.
\newblock {\em Commun. ACM}, 24:381--395, 1981.

\bibitem{Fredriksson12accv}
J.~Fredriksson and C.~Olsson.
\newblock Simultaneous multiple rotation averaging using lagrangian duality.
\newblock In {\em Asian Conf. on Computer Vision (ACCV)}, 2012.

\bibitem{Fredriksson12accv-rotationaveragingLagrangian}
Johan Fredriksson and Carl Olsson.
\newblock Simultaneous multiple rotation averaging using lagrangian duality.
\newblock In {\em Asian Conf. on Computer Vision (ACCV)}, pages 245--258.
  Springer, 2012.

\bibitem{Gao03PAMI-P3P}
Xiao-Shan Gao, Xiao-Rong Hou, Jianliang Tang, and Hang-Fei Cheng.
\newblock Complete solution classification for the perspective-three-point
  problem.
\newblock {\em {IEEE} Trans. Pattern Anal. Machine Intell.}, 25(8):930--943,
  2003.

\bibitem{Garcia20arXiv-certifiableRelativePose}
Mercedes Garcia-Salguero, Jesus Briales, and Javier Gonzalez-Jimenez.
\newblock {Certifiable Relative Pose Estimation}.
\newblock {\em arXiv preprint arXiv:2003.13732}, 2020.

\bibitem{Giamou19ral-SDPExtrinsicCalibration}
Matthew Giamou, Ziye Ma, Valentin Peretroukhin, and Jonathan Kelly.
\newblock Certifiably globally optimal extrinsic calibration from per-sensor
  egomotion.
\newblock {\em {IEEE} Robotics and Automation Letters}, 4(2):367--374, 2019.

\bibitem{Goemans95JACM-maxCutBound}
Michel~X Goemans and David~P Williamson.
\newblock Improved approximation algorithms for maximum cut and satisfiability
  problems using semidefinite programming.
\newblock {\em Journal of the ACM (JACM)}, 42(6):1115--1145, 1995.

\bibitem{Hager89SIAM-matrixInverse}
William~W Hager.
\newblock Updating the inverse of a matrix.
\newblock {\em SIAM review}, 31(2):221--239, 1989.

\bibitem{Hartley13ijcv}
R.~Hartley, J.~Trumpf, Y.~Dai, and H.~Li.
\newblock Rotation averaging.
\newblock {\em IJCV}, 103(3):267--305, 2013.

\bibitem{Hartley04}
R.~I. Hartley and A.~Zisserman.
\newblock {\em Multiple View Geometry in Computer Vision}.
\newblock Cambridge University Press, second edition, 2004.

\bibitem{Hartley09ijcv-globalRotationRegistration}
R.I. Hartley and F.~Kahl.
\newblock Global optimization through rotation space search.
\newblock {\em Intl. J. of Computer Vision}, 82(1):64--79, 2009.

\bibitem{Hartley11cvpr-l1rotationaveraging}
Richard Hartley, Khurrum Aftab, and Jochen Trumpf.
\newblock {L1 rotation averaging using the Weiszfeld algorithm}.
\newblock In {\em IEEE Conf. on Computer Vision and Pattern Recognition
  (CVPR)}, pages 3041--3048. IEEE, 2011.

\bibitem{Heller14icra-handeyePOP}
Jan Heller, Didier Henrion, and Tomas Pajdla.
\newblock Hand-eye and robot-world calibration by global polynomial
  optimization.
\newblock In {\em IEEE Intl. Conf. on Robotics and Automation (ICRA)}, pages
  3157--3164. IEEE, 2014.

\bibitem{Henrion12handbook-conicProjection}
Didier Henrion and J{\'e}r{\^o}me Malick.
\newblock Projection methods in conic optimization.
\newblock In {\em Handbook on Semidefinite, Conic and Polynomial Optimization},
  pages 565--600. Springer, 2012.

\bibitem{Higham88LA-nearestSPD}
Nicholas~J Higham.
\newblock Computing a nearest symmetric positive semidefinite matrix.
\newblock {\em Linear algebra and its applications}, 103:103--118, 1988.

\bibitem{Horn87josa}
Berthold K.~P. Horn.
\newblock Closed-form solution of absolute orientation using unit quaternions.
\newblock {\em J. Opt. Soc. Amer.}, 4(4):629--642, Apr 1987.

\bibitem{Huber81}
P.~Huber.
\newblock {\em Robust Statistics}.
\newblock John Wiley \& Sons, New York, NY, 1981.

\bibitem{Iglesias20cvpr-PSRGlobalOptimality}
Jos{\'e}~Pedro Iglesias, Carl Olsson, and Fredrik Kahl.
\newblock Global optimality for point set registration using semidefinite
  programming.
\newblock In {\em IEEE Conf. on Computer Vision and Pattern Recognition
  (CVPR)}, 2020.

\bibitem{Izatt17isrr-MIPregistration}
G.~Izatt, H.~Dai, and R.~Tedrake.
\newblock Globally optimal object pose estimation in point clouds with
  mixed-integer programming.
\newblock In {\em Proc. of the Intl. Symp. of Robotics Research (ISRR)}, 2017.

\bibitem{Jegelka13NIPS-DRSreflection}
Stefanie Jegelka, Francis Bach, and Suvrit Sra.
\newblock Reflection methods for user-friendly submodular optimization.
\newblock In {\em Advances in Neural Information Processing Systems (NIPS)},
  pages 1313--1321, 2013.

\bibitem{Jiao20arXiv-VIOpointline}
Yanmei Jiao, Yue Wang, Bo~Fu, Qimeng Tan, Lei Chen, Shoudong Huang, and Rong
  Xiong.
\newblock Globally optimal consensus maximization for robust visual inertial
  localization in point and line map.
\newblock {\em arXiv preprint arXiv:2002.11905}, 2020.

\bibitem{Kahl07IJCV-GlobalOptGeometricReconstruction}
Fredrik Kahl and Didier Henrion.
\newblock Globally optimal estimates for geometric reconstruction problems.
\newblock {\em Intl. J. of Computer Vision}, 74(1):3--15, 2007.

\bibitem{Klein07ismar-PTAM}
Georg Klein and David Murray.
\newblock Parallel tracking and mapping for small ar workspaces.
\newblock In {\em 2007 6th IEEE and ACM international symposium on mixed and
  augmented reality}, pages 225--234. IEEE, 2007.

\bibitem{Kneip2014ECCV-UPnP}
Laurent Kneip, Hongdong Li, and Yongduek Seo.
\newblock {UPnP}: An optimal {o}(n) solution to the absolute pose problem with
  universal applicability.
\newblock In {\em European Conf. on Computer Vision (ECCV)}, pages 127--142.
  Springer, 2014.

\bibitem{Kukelova2008ECCV-automaticGeneratorofMinimalProblemSolvers}
Zuzana Kukelova, Martin Bujnak, and Tomas Pajdla.
\newblock Automatic generator of minimal problem solvers.
\newblock In {\em European Conf. on Computer Vision (ECCV)}, pages 302--315.
  Springer, 2008.

\bibitem{Kuemmerle11icra}
R.~K{\"u}mmerle, G.~Grisetti, H.~Strasdat, K.~Konolige, and W.~Burgard.
\newblock g2o: A general framework for graph optimization.
\newblock In {\em Proc.~of the IEEE Int.~Conf.~on Robotics and Automation
  (ICRA)}, May 2011.

\bibitem{Lajoie19ral-DCGM}
P.~Lajoie, S.~Hu, G.~Beltrame, and L.~Carlone.
\newblock Modeling perceptual aliasing in {SLAM} via discrete-continuous
  graphical models.
\newblock {\em {IEEE} Robotics and Automation Letters ({RA-L})}, 2019.

\bibitem{lasserre01ipco-lasserrehierarchybinary}
Jean~B Lasserre.
\newblock {An explicit exact SDP relaxation for nonlinear 0-1 programs}.
\newblock In {\em International Conference on Integer Programming and
  Combinatorial Optimization}, pages 293--303. Springer, 2001.

\bibitem{Lasserre01siopt-LasserreHierarchy}
Jean~B. Lasserre.
\newblock Global optimization with polynomials and the problem of moments.
\newblock {\em SIAM J. Optim.}, 11(3):796--817, 2001.

\bibitem{lasserre10book-momentsOpt}
Jean-Bernard Lasserre.
\newblock {\em Moments, positive polynomials and their applications}, volume~1.
\newblock World Scientific, 2010.

\bibitem{Laurent09eaag-SOSMomentOptimization}
Monique Laurent.
\newblock Sums of squares, moment matrices and optimization over polynomials.
\newblock In {\em Emerging applications of algebraic geometry}, pages 157--270.
  Springer, 2009.

\bibitem{Laurent09AdM-generalFlatExtension}
Monique Laurent and Bernard Mourrain.
\newblock A generalized flat extension theorem for moment matrices.
\newblock {\em Archiv der Mathematik}, 93(1):87--98, 2009.

\bibitem{Lee20arXiv-robustSRA}
Seong~Hun Lee and Javier Civera.
\newblock Robust single rotation averaging.
\newblock {\em arXiv preprint arXiv:2004.00732}, 2020.

\bibitem{Lofberg04cacsd-yalmip}
Johan L{\"o}fberg.
\newblock {YALMIP: A toolbox for modeling and optimization in MATLAB}.
\newblock In {\em Proceedings of the CACSD Conference}, volume~3. Taipei,
  Taiwan, 2004.

\bibitem{Luo2010SP-sdpRelaxationQuadratic}
Zhi-Quan Luo, Wing-Kin~Ken Ma, Anthony Man-Cho So, Yinyu Ye, and Shuzhong
  Zhang.
\newblock Semidefinite relaxation of quadratic optimization problems.
\newblock {\em IEEE Signal Processing Magazine}, 1053(5888/10), 2010.

\bibitem{markley1988jas-svdAttitudeDeter}
F.~L. Markley.
\newblock Attitude determination using vector observations and the singular
  value decomposition.
\newblock {\em The Journal of the Astronautical Sciences}, 36(3):245--258,
  1988.

\bibitem{Markley14book-fundamentalsAttitudeDetermine}
F.~L. Markley and J.~L. Crassidis.
\newblock {\em Fundamentals of spacecraft attitude determination and control},
  volume~33.
\newblock Springer, 2014.

\bibitem{Maron16tog-PMSDP}
Haggai Maron, Nadav Dym, Itay Kezurer, Shahar Kovalsky, and Yaron Lipman.
\newblock Point registration via efficient convex relaxation.
\newblock {\em ACM Transactions on Graphics (TOG)}, 35(4):1--12, 2016.

\bibitem{Maronna19book-robustStats}
Ricardo~A Maronna, R~Douglas Martin, Victor~J Yohai, and Mat{\'\i}as
  Salibi{\'a}n-Barrera.
\newblock {\em Robust statistics: theory and methods (with R)}.
\newblock John Wiley \& Sons, 2019.

\bibitem{Massey51JASA-KStest}
Frank~J Massey~Jr.
\newblock The kolmogorov-smirnov test for goodness of fit.
\newblock {\em Journal of the American statistical Association},
  46(253):68--78, 1951.

\bibitem{Nesterov18book-convexOptimization}
Yurii Nesterov.
\newblock {\em Lectures on convex optimization}, volume 137.
\newblock Springer, 2018.

\bibitem{Nie14mp-finiteConvergenceLassere}
Jiawang Nie.
\newblock Optimality conditions and finite convergence of lasserre’s
  hierarchy.
\newblock {\em Mathematical programming}, 146(1-2):97--121, 2014.

\bibitem{Nister04pami}
D.~Nist\'er.
\newblock An efficient solution to the five-point relative pose problem.
\newblock {\em {IEEE} Trans. Pattern Anal. Machine Intell.}, 26(6):756--770,
  2004.

\bibitem{Olsson09pami-bnbRegistration}
Carl Olsson, Fredrik Kahl, and Magnus Oskarsson.
\newblock Branch-and-bound methods for euclidean registration problems.
\newblock {\em {IEEE} Trans. Pattern Anal. Machine Intell.}, 31(5):783--794,
  2009.

\bibitem{Ovsjanikov12TOG-functionalMaps}
Maks Ovsjanikov, Mirela Ben-Chen, Justin Solomon, Adrian Butscher, and Leonidas
  Guibas.
\newblock Functional maps: a flexible representation of maps between shapes.
\newblock {\em ACM Transactions on Graphics (TOG)}, 31(4):1--11, 2012.

\bibitem{PajdlaXXwebsite-minimalProblemsInVision}
Tomas Pajdla and Zuzana Kukelova.
\newblock Minimal problems in computer vision.
\newblock \url{http://cmp.felk.cvut.cz/old_pages/mini/}, 2019.

\bibitem{Bustos18pami-GORE}
{\'A}.~{Parra Bustos} and T.~J. Chin.
\newblock Guaranteed outlier removal for point cloud registration with
  correspondences.
\newblock {\em {IEEE} Trans. Pattern Anal. Machine Intell.}, 40(12):2868--2882,
  2018.

\bibitem{Poggio85nature-computationalVision}
Tomaso Poggio, Vincent Torre, and Christof Koch.
\newblock Computational vision and regularization theory.
\newblock {\em nature}, 317(6035):314--319, 1985.

\bibitem{Probst19ICCV-convexRelaxationNonminimal}
Thomas Probst, Danda~Pani Paudel, Ajad Chhatkuli, and Luc Van~Gool.
\newblock Convex relaxations for consensus and non-minimal problems in {3D}
  vision.
\newblock In {\em Intl. Conf. on Computer Vision (ICCV)}, 2019.

\bibitem{Rosen20wafr-scalableLowRankSDP}
David~M. Rosen.
\newblock Scalable low-rank semidefinite programming for certifiably correct
  machine perception.
\newblock In {\em Intl. Workshop on the Algorithmic Foundations of Robotics
  (WAFR)}, 2020.

\bibitem{Rosen18ijrr-sesync}
D.M. Rosen, L.~Carlone, A.S. Bandeira, and J.J. Leonard.
\newblock {SE-Sync}: a certifiably correct algorithm for synchronization over
  the {Special Euclidean} group.
\newblock {\em Intl. J. of Robotics Research}, 2018.

\bibitem{Schonberger16cvpr-SfMRevisited}
Johannes~L Schonberger and Jan-Michael Frahm.
\newblock Structure-from-motion revisited.
\newblock In {\em IEEE Conf. on Computer Vision and Pattern Recognition
  (CVPR)}, pages 4104--4113, 2016.

\bibitem{Sharma19arXiv-SPEED}
Sumant Sharma and Simone D'Amico.
\newblock Pose estimation for non-cooperative rendezvous using neural networks.
\newblock {\em arXiv preprint arXiv:1906.09868}, 2019.

\bibitem{Toh12handbook-SDPT3Implementation}
Kim-Chuan Toh, Michael~J Todd, and Reha~H T{\"u}t{\"u}nc{\"u}.
\newblock {On the implementation and usage of SDPT3--a Matlab software package
  for semidefinite-quadratic-linear programming, version 4.0}.
\newblock In {\em Handbook on semidefinite, conic and polynomial optimization},
  pages 715--754. Springer, 2012.

\bibitem{Tron15rssws3D-dualityPGO3D}
R.~Tron, D.~Rosen, and L.~Carlone.
\newblock On the inclusion of determinant constraints in lagrangian duality for
  {3D SLAM}.
\newblock In {\em Robotics: Science and Systems (RSS), Workshop ``The problem
  of mobile sensors: Setting future goals and indicators of progress for
  {SLAM}''}, 2015.
\newblock
  \linkToPdf{https://www.dropbox.com/s/859umrdf7ldd2kv/2015ws-rss-duality3Ddet.pdf?dl=0}.

\bibitem{Tzoumas19iros-outliers}
V.~Tzoumas, P.~Antonante, and L.~Carlone.
\newblock Outlier-robust spatial perception: Hardness, general-purpose
  algorithms, and guarantees.
\newblock In {\em IEEE/RSJ Intl. Conf. on Intelligent Robots and Systems
  (IROS)}, 2019.

\bibitem{Waki06jopt-SOSSparsity}
Hayato Waki, Sunyoung Kim, Masakazu Kojima, and Masakazu Muramatsu.
\newblock Sums of squares and semidefinite program relaxations for polynomial
  optimization problems with structured sparsity.
\newblock {\em SIAM J. Optim.}, 17(1):218--242, 2006.

\bibitem{Wang20arXiv-cs-tssos}
Jie Wang, Victor Magron, Jean~B Lasserre, and Ngoc Hoang~Anh Mai.
\newblock {CS-TSSOS: Correlative and term sparsity for large-scale polynomial
  optimization}.
\newblock {\em arXiv preprint arXiv:2005.02828}, 2020.

\bibitem{Wang13ima}
L.~Wang and A.~Singer.
\newblock Exact and stable recovery of rotations for robust synchronization.
\newblock {\em Information and Inference: A Journal of the IMA}, 30, 2013.

\bibitem{Weisser18mpc-SBSOS}
Tillmann Weisser, Jean~B Lasserre, and Kim-Chuan Toh.
\newblock Sparse-{BSOS}: a bounded degree {SOS} hierarchy for large scale
  polynomial optimization with sparsity.
\newblock {\em Math. Program. Comput.}, 10(1):1--32, 2018.

\bibitem{Wientapper18cviu-absolutePose}
Folker Wientapper, Michael Schmitt, Matthieu Fraissinet-Tachet, and Arjan
  Kuijper.
\newblock A universal, closed-form approach for absolute pose problems.
\newblock {\em Comput. Vis. Image Underst.}, 173:57--75, 2018.

\bibitem{Wise20arXiv-certifiablyHandeye}
Emmett Wise, Matthew Giamou, Soroush Khoubyarian, Abhinav Grover, and Jonathan
  Kelly.
\newblock Certifiably optimal monocular hand-eye calibration.
\newblock {\em arXiv preprint arXiv:2005.08298}, 2020.

\bibitem{Wolkowicz12book-handbookSDP}
Henry Wolkowicz, Romesh Saigal, and Lieven Vandenberghe.
\newblock {\em Handbook of semidefinite programming: theory, algorithms, and
  applications}, volume~27.
\newblock Springer Science \& Business Media, 2012.

\bibitem{Yang19rss-teaser}
H.~Yang and L.~Carlone.
\newblock A polynomial-time solution for robust registration with extreme
  outlier rates.
\newblock In {\em Robotics: Science and Systems (RSS)}, 2019.

\bibitem{Yang19iccv-QUASAR}
H.~Yang and L.~Carlone.
\newblock A quaternion-based certifiably optimal solution to the {Wahba}
  problem with outliers.
\newblock In {\em Intl. Conf. on Computer Vision (ICCV)}, 2019.

\bibitem{Yang20cvpr-shapeStar}
H.~Yang and L.~Carlone.
\newblock In perfect shape: Certifiably optimal {3D} shape reconstruction from
  {2D} landmarks.
\newblock In {\em IEEE Conf. on Computer Vision and Pattern Recognition
  (CVPR)}, 2020.

\bibitem{Yang20arxiv-teaser}
H.~Yang, J.~Shi, and L.~Carlone.
\newblock {TEASER: Fast and Certifiable Point Cloud Registration}.
\newblock {\em {IEEE} Trans. Robotics (T-RO)}, 2020.

\bibitem{Yang20ral-GNC}
Heng Yang, Pasquale Antonante, Vasileios Tzoumas, and Luca Carlone.
\newblock Graduated non-convexity for robust spatial perception: From
  non-minimal solvers to global outlier rejection.
\newblock {\em {IEEE} Robotics and Automation Letters ({RA-L})},
  5(2):1127--1134, 2020.

\bibitem{Yang16pami-goicp}
J.~Yang, H.~Li, D.~Campbell, and Y.~Jia.
\newblock {Go-ICP}: A globally optimal solution to {3D ICP} point-set
  registration.
\newblock {\em {IEEE} Trans. Pattern Anal. Machine Intell.}, 38(11):2241--2254,
  November 2016.

\bibitem{Yang2014ECCV-optimalEssentialEstimationBnBConsensusMax}
Jiaolong Yang, Hongdong Li, and Yunde Jia.
\newblock Optimal essential matrix estimation via inlier-set maximization.
\newblock In {\em European Conf. on Computer Vision (ECCV)}, pages 111--126.
  Springer, 2014.

\bibitem{Yang10NIPS-relaxedClipping}
Min Yang, Linli Xu, Martha White, Dale Schuurmans, and Yao-liang Yu.
\newblock Relaxed clipping: A global training method for robust regression and
  classification.
\newblock In {\em Advances in neural information processing systems}, pages
  2532--2540, 2010.

\bibitem{Yu12NIPS-robustRegression}
Yao-liang Yu, {\"O}zlem Aslan, and Dale Schuurmans.
\newblock A polynomial-time form of robust regression.
\newblock In {\em Advances in Neural Information Processing Systems (NIPS)},
  pages 2483--2491, 2012.

\bibitem{Zhang15icra-vloam}
J.~Zhang and S.~Singh.
\newblock Visual-lidar odometry and mapping: Low-drift, robust, and fast.
\newblock In {\em IEEE Intl. Conf. on Robotics and Automation (ICRA)}, pages
  2174--2181. IEEE, 2015.

\bibitem{Zhao19arxiv-efficientTwoView}
Ji~Zhao.
\newblock An efficient solution to non-minimal case essential matrix
  estimation.
\newblock 2019.
\newblock arXiv preprint arXiv:1903.09067.

\bibitem{Zhao20cvpr-certifiablyEssential}
Ji~Zhao, Wanting Xu, and Laurent Kneip.
\newblock A certifiably globally optimal solution to generalized essential
  matrix estimation.
\newblock In {\em IEEE Conf. on Computer Vision and Pattern Recognition
  (CVPR)}, 2020.

\bibitem{Zheng18TAC-partialOrthogonality}
Yang Zheng, Giovanni Fantuzzi, and Antonis Papachristodoulou.
\newblock {Fast ADMM for sum-of-squares programs using partial orthogonality}.
\newblock {\em IEEE Transactions on Automatic Control}, 64(9):3869--3876, 2018.

\bibitem{Zhou20ICRA-GRegAlgebraicSolver}
Lipu Zhou, Shengze Wang, and Michael Kaess.
\newblock A fast and accurate solution for pose estimation from 3d
  correspondences.
\newblock In {\em IEEE Intl. Conf. on Robotics and Automation (ICRA)}, 2020.

\bibitem{Zhou17pami-shapeEstimationConvex}
Xiaowei Zhou, Menglong Zhu, Spyridon Leonardos, and Kostas Daniilidis.
\newblock Sparse representation for {3D} shape estimation: A convex relaxation
  approach.
\newblock {\em {IEEE} Trans. Pattern Anal. Machine Intell.}, 39(8):1648--1661,
  2017.

\end{thebibliography}
